\documentclass[12pt]{article}

\usepackage{amsmath,amssymb,amsthm,mathtools,bm}
\usepackage{booktabs,multirow,enumitem,float}
\usepackage{ascmac}
\usepackage{xcolor,mathrsfs}
\usepackage{geometry}
\usepackage[expansion=false]{microtype}
\usepackage{etoolbox}
\usepackage{hyperref}
\makeatletter
\let\manuscript@cite\@cite
\renewcommand{\@cite}[2]{{\upshape\manuscript@cite{#1}{#2}}}
\patchcmd{\@maketitle}{\null}{%
  \rightline{\baselineskip16pt\normalfont\vbox to36pt{%
    \hbox{\shortstack[r]{OMU-Phys-633\\AP- GR- 216}}\vss}}%
}{}{\PackageError{manuscript}{Could not insert preprint number}{}}
\makeatother

\hypersetup{
  pdftitle={Three-dimensional noncompact trans-Sasakian eta-Einstein manifolds},
  pdfauthor={Satsuki Matsuno},
  pdfkeywords={trans-Sasakian manifold, eta-Einstein metric, noncompact manifold, Newman-Penrose formalism, Reeb flow, moduli, Riemannian quotient},
  colorlinks=true,
  linkcolor=blue,
  urlcolor=blue
}

\allowdisplaybreaks
\newcommand{\dd}{\mathop{}\!\mathrm{d}}
\newcommand{\D}{\mathbb D}
\newcommand{\C}{\mathbb C}
\newcommand{\Q}{\mathbb Q}
\newcommand{\R}{\mathbb R}
\newcommand{\Z}{\mathbb Z}
\newcommand{\T}{\mathbb T}

\newcommand{\Ric}{\operatorname{Ric}}
\newcommand{\Aut}{\operatorname{HolAut}}

\newcommand{\StrAut}{\operatorname{StrAut}}
\newcommand{\Lie}{\mathcal L}
\newcommand{\e}{\mathrm e}
\newcommand{\ii}{\mathrm i}
\newcommand{\id}{\operatorname{id}}
\newcommand{\Isom}{\operatorname{Isom}}

\renewcommand{\Im}{\operatorname{Im}}
\renewcommand{\Re}{\operatorname{Re}}
\DeclareMathOperator{\Log}{Log}

\theoremstyle{plain}
\newtheorem{theorem}{Theorem}[section]
\newtheorem{proposition}[theorem]{Proposition}
\newtheorem{lemma}[theorem]{Lemma}
\newtheorem{corollary}[theorem]{Corollary}
\theoremstyle{definition}
\newtheorem{definition}[theorem]{Definition}

\theoremstyle{remark}
\newtheorem{remark}[theorem]{Remark}

\numberwithin{equation}{section}

\title{Three-dimensional noncompact trans-Sasakian $\eta$-Einstein manifolds}
\author{
Satsuki Matsuno\thanks{Department of Physics, Graduate School of Science, Osaka Metropolitan University.
E-mail: Satsuki\_Matsuno@omu.ac.jp.
ORCID: 0000-0003-0841-8033.}
}
\date{}

\begin{document}
\maketitle

\begin{abstract}
We classify connected, simply connected, noncompact three-dimensional trans-Sasakian $\eta$-Einstein manifolds.
In \cite{Matsuno2025}, we showed that, in the compact case, such a manifold is either of $\beta$-Kenmotsu type or homothetic to a Sasakian manifold.
In this paper, in the noncompact case, we find families of trans-Sasakian $\eta$-Einstein manifolds that are neither $\beta$-Kenmotsu nor homothetic to Sasakian manifolds.
These families have infinitely many degrees of freedom.
We give these nontrivial metrics explicitly.
We also give a structure theorem for the case in which the characteristic vector field is Palais-nonregular.
\end{abstract}

\noindent\textbf{Keywords.}
Trans-Sasakian manifold; $\eta$-Einstein metric; noncompact manifold; Newman--Penrose formalism; Palais-nonregular.

\medskip
\noindent\textbf{2020 Mathematics Subject Classification.}
Primary 53C25; Secondary 53C21, 57R30.

\tableofcontents

\section{Introduction}
A trans-Sasakian manifold is a generalization that includes Sasakian manifolds and Kenmotsu manifolds \cite{Oubina}.
They are interesting objects and have been studied extensively \cite{Marrero,PrasadPankajTripathiShukla,AlSolamyKimTripathi}.
Normal almost contact metric (ACM) manifolds are also interesting objects.
The basic properties of three-dimensional normal ACM manifolds were studied by Olszak \cite{Olszak}.
Moreover, in dimension three, a trans-Sasakian manifold is equivalent to a normal ACM manifold \cite{Olszak,AlSolamyKimTripathi}.
Therefore, in dimension three, these two interesting classes coincide and have been studied extensively \cite{PrasadPankajTripathiShukla,AlSolamyKimTripathi}.

A fundamental problem for ACM manifolds is to determine which of them are $\eta$-Einstein.
Olszak computed the Ricci tensor of three-dimensional normal ACM manifolds and derived conditions for them to be $\eta$-Einstein \cite{Olszak}\footnote{Strictly speaking, Olszak used the term $\eta$-Einstein in a sense different from its present definition. However, because he computed the Ricci tensor, the present $\eta$-Einstein condition follows easily as a corollary of his theorem.}.
However, until recently, the question of which three-dimensional normal ACM or three-dimensional trans-Sasakian manifolds are $\eta$-Einstein had been answered only partially.
Here, ``partially'' means that the results were obtained only under the stronger assumption that the structure is $\alpha$-Sasakian or $\beta$-Kenmotsu.

The author recently studied compact three-dimensional trans-Sasakian $\eta$-Einstein manifolds \cite{Matsuno2025}.
Since compactness allows the use of integral techniques, a certain rigidity theorem was obtained from an a priori estimate for the $\eta$-Einstein condition.
As a result, such manifolds were shown to reduce to the $\alpha$-Sasakian or $\beta$-Kenmotsu cases.
In the noncompact case, however, integral techniques cannot be used, so a completely different approach is required, and the problem remained unresolved.

In this paper, we classify connected simply connected noncompact three-dimensional trans-Sasakian $\eta$-Einstein manifolds.
Unlike in the compact case, we find that trans-Sasakian $\eta$-Einstein manifolds that are neither $\beta$-Kenmotsu nor $\alpha$-Sasakian exist with infinitely many degrees of freedom.
We give these nontrivial metrics explicitly.
We also give a structure theorem for the case in which the characteristic vector field is Palais-nonregular \cite{Palais}.

\subsection{Summary of the results}
\label{subsec:summary-results}
Let $(M,\varphi,\xi,\eta,g)$ be a connected, noncompact, three-dimensional trans-Sasakian manifold of type $(\alpha,\beta)$, defined by
\begin{align*}
(\nabla_X\varphi)Y
=
\alpha\{g(X,Y)\xi-\eta(Y)X\}
+\beta\{g(\varphi X,Y)\xi-\eta(Y)\varphi X\}.
\end{align*}
We further assume that it is $\eta$-Einstein and that $\xi$ is complete.
If $\alpha\equiv0$, that is rank one, the universal cover is a warped product
\[
 \widetilde M=\R\times B,\qquad
 \tilde g=\dd t^2+e^{2\Omega(t)}h,
\]
and if $\beta\equiv0$ and $\alpha\not\equiv0$, that is $\alpha$-Sasakian, the metric is homothetic to a Sasakian metric (Proposition~\ref{prop:eta-einstein simple case}).

Therefore, we are mainly interested in a structure such that $\alpha\beta\not\equiv0$.
We call this structure \textbf{proper trans-Sasakian}.
The author showed that no compact three-dimensional proper trans-Sasakian $\eta$-Einstein manifolds exist \cite{Matsuno2025}.
In the non-compact case, the universal cover is isomorphic to one of the following three model spaces (Theorem~\ref{prop:Standard eta-Einstein metrics}).

Let $K\in\R$ be a constant, and let $\D=\{z\in\C:|z|<1\}$ be the unit disk with its standard complex structure.
Let $z=x+\ii y$ and $\Delta_0=\partial_x^2+\partial_y^2$.
An \textbf{admissible pair} for $K>0$ is a pair
$(v,s)\in C^\infty(\D,\R)\times\mathcal O(\D)$ satisfying
\begin{equation}
 \Delta_0v>0,\qquad e^v|s|<1,\qquad s\not\equiv0.
\end{equation}
An \textbf{admissible potential} for $K=0$ is a function
$U\in C^\infty(\D,\R)$ satisfying
\begin{align}
 U>0,\qquad \Delta_0U<0.
\end{align}
An \textbf{admissible potential} for $K<0$ is a function
$U\in C^\infty(\D,\R)$ satisfying
\begin{align}
 U\not\equiv0,\qquad |U|<\frac{\pi}{\sqrt{2|K|}},
\end{align}
and for which there exists a function $H_U\in C^\infty(\D,(0,\infty))$ satisfying
\begin{align}
 \Delta_0U=\sqrt{2|K|}\,H_U\sin(\sqrt{2|K|}U).
\end{align}
We denote the respective sets of data by $\mathcal U_K^{\mathrm{pair}}$ for $K>0$
and $\mathcal U_K^{\mathrm{pot}}$ for $K\leq0$.

For each fixed $K\in\R$, take an admissible pair $(v,s)$ if $K>0$,
or an admissible potential $U$ if $K\leq0$, on the unit disk $\D$.
On $\D_z\times\R_w$, we define the one-form $\eta_K$ as
\begin{equation}
 \eta_K=
 \begin{cases}
  \dd w-(2K)^{-1/2}*_0\dd v,&K>0,\\
  \dd w+*_0\dd U,&K=0,\\
  \dd w-*_0\dd U,&K<0.
 \end{cases}
\end{equation}
On $\D_z\times\R_w$, we define the metrics as follows.
If $K>0$,
\begin{equation}
 \begin{aligned}
 &g_K=A_{K,v,s}|\dd z|^2+\eta_K^2,\\
 &A_{K,v,s}=\frac{\Delta_0v}{2K}\frac{|1-q|^2}{1-|q|^2},
 \qquad q=e^{v-\ii\sqrt{2K}w}\overline{s}.
 \end{aligned}
\end{equation}
If $K=0$,
\begin{equation}
\begin{aligned}
&g_0=A_{0,U}|\dd z|^2+\eta_0^2,\\
&A_{0,U}=-\frac{\Delta_0U}{2U}(w^2+U^2),
\end{aligned}
\end{equation}
If $K<0$,
\begin{equation}
 \begin{aligned}
 &g_K=A_{K,U}|\dd z|^2+\eta_K^2,\\
 &A_{K,U}=H_U\{\cosh(\sqrt{2|K|}w)+\cos(\sqrt{2|K|}U)\}.
 \end{aligned}
\end{equation}
Set $\xi=\partial_w$ and let $\varphi_K$ be the horizontal lift of the standard complex structure on $\D$, with $\varphi_K\xi=0$.
We denote these ACM models by $\mathcal M_K(v,s)$ for $K>0$ and $\mathcal M_K(U)$ for $K\leq0$.

Appendix~\ref{app:Geodesically complete model spaces} gives a completeness criterion.
Explicit families of complete model spaces with infinitely many degrees of freedom are also given.

Therefore, a connected noncompact three-dimensional proper trans-Sasakian $\eta$-Einstein manifold with complete $\xi$ is a quotient of $\mathcal M_K(v,s)$ or $\mathcal M_K(U)$ by a subgroup of the corresponding structure-preserving automorphism group that acts freely and properly discontinuously (Theorem~\ref{thm:global-quotient-model}).
The structure-preserving automorphism group $\StrAut(\mathcal M_K(v,s))$ or $\StrAut(\mathcal M_K(U))$ is given as follows (Propositions~\ref{prop:generalized-rigidity} and~\ref{prop:nonpositive-isomorphisms}).

For $K>0$, let $\mathfrak u=v+\log|s|$ and
$G_{\mathfrak u}=\{\psi\in\Aut(\D):\mathfrak u\circ\psi=\mathfrak u\}$, where $\Aut(\D)$ is the group of holomorphic diffeomorphisms from $\D$ onto itself.
Then $\StrAut(\mathcal M_K(v,s))=\{(\psi,L):\psi\in G_{\mathfrak u},\ L\in\mathcal O(\D),\ e^L=s/(s\circ\psi)\}$.
For $K\leq0$, $\StrAut(\mathcal M_K(U))=\{\psi\in\Aut(\D):U\circ\psi=U\}$.

If $\xi$ is Palais-nonregular, then $K>0$ and the structure has rank three (Theorem~\ref{thm:orbit-trichotomy}).
Therefore, in the Palais-nonregular case, the proper trans-Sasakian manifold is given by $M=\mathcal M_K(v,s)/\Gamma$, where $\Gamma\subset \StrAut(\mathcal M_K(v,s))$ is a Deck group.
Furthermore, there are three nonregular types \textup{(SF)}, \textup{(IR$_0$)}, and \textup{(IR$_1$)} as follows (Theorem~\ref{thm:orbit-trichotomy}).

Type \textup{(SF)} has an effective circle action which is locally free but not free.
Its orbits define a Seifert fibration over a noncompact oriented $2$-orbifold with at least one exceptional fiber.
The Deck groups are described in Theorem~\ref{thm:SF-deck-classification}.

Type \textup{(IR$_0$)} has no periodic characteristic orbit, and every orbit is dense in a principal two-torus. In this case $M\cong\T^2\times\R$.
The Deck groups are described in Theorem~\ref{thm:IR0-deck-classification}.

Type \textup{(IR$_1$)} has exactly one periodic characteristic orbit, and every other orbit is dense in a principal two-torus.
In this case $M\cong S^1\times\R^2$.
The Deck groups are described in Theorem~\ref{thm:IR1-deck-classification}.

\section{Preliminaries}
\label{sec:preliminaries}
\subsection{Three-dimensional Riemannian Newman--Penrose formalism}\label{sec:riemannian NP formalism}
In this section, we summarize the Newman-Penrose formalism on three-dimensional Riemannian manifolds.
The material in this section is based on \cite{Matsuno2025}.
Let $(M,g)$ be an orientable three-dimensional Riemannian manifold.
We can choose a nowhere-vanishing unit vector field $\xi$ and a suitable orthonormal frame $\{\xi,e_2,e_3\}$ at least locally.
Setting $\partial=\frac{1}{\sqrt{2}}(e_2-ie_3)$, we obtain a local complex frame $\{\xi,\partial,\bar{\partial}\}$.
In this setting, the metric satisfies $g(\xi,\xi)=g(\partial,\bar{\partial})=1$ and $g(\xi,\partial)=g(\xi,\bar{\partial})=g(\partial,\partial)=g(\bar{\partial},\bar{\partial})=0$.
Letting $\{\eta,\theta^2,\theta^3\}$ be the dual frame of $\{\xi,e_2,e_3\}$ and setting $\mu=\frac{1}{\sqrt{2}}(\theta^2+i\theta^3)$, we see that the complex dual frame corresponding to the complex frame $\{\xi,\partial,\bar{\partial}\}$ is $\{\eta,\mu,\bar{\mu}\}$ by noting that $\mu(\partial)=1$.
The \textcolor{blue}{\textbf{spin coefficients}} with respect to the frame $\{\xi,\partial,\bar{\partial}\}$ are defined as follows.
\begin{align*}
\kappa=-g(\nabla_\xi\xi,\partial),\ \rho=g(\nabla_{\bar{\partial}}\xi,\partial),\ \sigma=-g(\nabla_\partial\xi,\partial),\\
\epsilon=g(\nabla_\xi\partial,\bar{\partial}),\ \beta=g(\nabla_\partial\partial,\bar{\partial}).
\end{align*}
These are equivalently defined as the coefficients of the following covariant derivatives or Lie brackets.
\begin{align}
&\nabla_\xi \xi = -\bar{\kappa}\partial - \kappa\bar{\partial},\ 
\nabla_\partial \xi = \bar{\rho}\partial - \sigma\bar{\partial},\ 
\nabla_\xi \partial = \kappa\xi + \epsilon\partial,\\
&\nabla_\partial\partial = \sigma\xi + \beta\partial,\ 
\nabla_{\partial} \bar{\partial} = -\bar{\rho}\xi - \beta\bar{\partial},\ \\
&[\xi,\partial] = \kappa\xi + (\epsilon - \bar{\rho})\partial + \sigma\bar{\partial},\\ 
&[\partial,\bar{\partial}] = (\rho - \bar{\rho})\xi + \bar{\beta}\partial - \beta\bar{\partial}.\label{eq:Lie bra 2}
\end{align}

If $\kappa = 0$, the integral curves of $\xi$ are geodesics.
We write $\rho = \Theta + i\omega$.
The real part $\Theta$ is called the expansion of $\xi$, and the imaginary part $\omega$ is called the twist of $\xi$.
The coefficient $\sigma$ is called the complex shear of $\xi$.
The quantities $\Theta$, $\omega$, and $\sigma$ admit the following geometric interpretations.
For the congruence generated by $\xi$, $\Theta$ represents the rate of change of the area of a cross-section orthogonal to $\xi$, $\omega$ represents the angular velocity of the cross-section, and $\sigma$ represents the shear deformation of the cross-section.

The generalized Sachs equations describe the relation between the Ricci tensor and the spin coefficients.
The equations used in this paper are as follows.
\begin{align}
&-\xi(\rho) - \bar{\partial}(\kappa)
    = |\kappa|^{2} + |\sigma|^{2} + \rho^{2}
      + \kappa \bar{\beta} + \tfrac12 \mathrm{Ric}(\xi,\xi),\label{eq:generalized Sachs eq 1}\\
&\xi(\sigma) - \partial(\kappa)
    = \kappa^{2} + 2\sigma\epsilon
      - \sigma(\rho + \bar{\rho})
      - \kappa\beta + \mathrm{Ric}(\partial,\partial),\label{eq:generalized Sachs eq 2}\\
&-\partial(\rho) - \bar{\partial}(\sigma)
    = 2\sigma\bar{\beta} + (\rho - \bar{\rho})\kappa
      + \mathrm{Ric}(\xi,\partial),\label{eq:generalized Sachs eq 3}\\
&\xi(\beta) - \partial(\epsilon)
    = \sigma(\bar{\kappa} - \bar{\beta})
      + \kappa(\epsilon + \bar{\rho})
      + \beta(\epsilon - \bar{\rho})
      - \mathrm{Ric}(\xi,\partial),\label{eq:generalized Sachs eq 4}\\
&\partial(\bar{\beta}) + \bar{\partial}(\beta)
    = |\sigma|^{2} - |\rho|^{2} - 2|\beta|^{2}
      - (\rho - \bar{\rho})\epsilon
      - \mathrm{Ric}(\partial,\bar{\partial})
      + \tfrac12 \mathrm{Ric}(\xi,\xi).\label{eq:generalized Sachs eq 5}
\end{align}
The components of the second Bianchi identity used in this paper are
\begin{align}
&\xi(\operatorname{Ric}(\xi,\partial))
-\frac12\partial(\operatorname{Ric}(\xi,\xi))
+\bar{\partial}(\operatorname{Ric}(\partial,\partial)) \label{eq:2nd bianchi 1}\\
&=
\kappa\operatorname{Ric}(\xi,\xi)
+(\epsilon-2\rho-\bar{\rho})\operatorname{Ric}(\xi,\partial)
+\sigma\operatorname{Ric}(\xi,\bar{\partial})\nonumber \\
&\quad
-\bigl(\bar{\kappa}+2\bar{\beta}\bigr)
\operatorname{Ric}(\partial,\partial)
-\kappa\operatorname{Ric}(\partial,\bar{\partial}).\nonumber\\
&\partial(\text{Ric}(\xi, \overline{\partial})) + \overline{\partial}(\text{Ric}(\xi, \partial)) - \xi(\text{Ric}(\partial, \overline{\partial})) + \frac{1}{2}\xi(\text{Ric}(\xi, \xi))  \label{eq:2nd bianchi 2}\\
&=-(\rho + \overline{\rho}) (\text{Ric}(\xi, \xi) - \text{Ric}(\partial, \bar{\partial})) \notag\\
&- \overline{\sigma}\text{Ric}(\partial, \partial) - \sigma\text{Ric}(\overline{\partial}, \overline{\partial}) - (2\overline{\kappa} + \overline{\beta})\text{Ric}(\xi, \partial) - (2\kappa + \beta)\text{Ric}(\xi, \overline{\partial}).\notag
\end{align}

Even when $\xi$ is globally defined, the complex frame $\{\xi,\partial,\bar{\partial}\}$ may exist only locally.
For example, on $M=S^2\times\R$, if $\xi$ is tangent to the $\R$-factor, then $\partial$ cannot be chosen globally to be tangent to the $S^2$-factor.
In such a case, we choose a new $\partial$ on each overlapping neighborhood.
Under a change of frame $\partial'=e^{i\theta}\partial$, the spin coefficients transform as
\begin{align}
&\kappa'=e^{i\theta}\kappa,\ \rho'=\rho,\ \sigma'=e^{2i\theta}\sigma,\\
&\epsilon'=\epsilon+i\xi(\theta),\ \beta'=e^{i\theta}(\beta+i\partial\theta)
\end{align}
In particular, $\rho$ depends only on the choice of $\xi$.
Although the spin coefficients other than $\rho$ change, one can verify that the generalized Sachs equations and the second Bianchi identities are well defined under this gauge transformation \cite{Matsuno2025}.
Therefore, $\rho$, the generalized Sachs equations, and the second Bianchi identities are well defined even when $\partial$ cannot be chosen globally.

\subsection{Some properties of three-dimensional  trans-Sasakian manifolds}\label{sec:some properties of 3D trans-Sasakian}
In this subsection, we describe almost contact metric structures in the Newman--Penrose formalism and then recall basic facts about three-dimensional trans-Sasakian manifolds.
Let $(M,\varphi,\xi,\eta,g)$ be a three-dimensional ACM manifold.
Here $\xi$ is the characteristic vector field, $\eta$ is the metric dual of $\xi$, and $\varphi\colon\Gamma(TM)\to\Gamma(TM)$ is an endomorphism field satisfying
\begin{align*}
&\varphi^2=-I+\eta\otimes\xi,\ \eta(\xi)=1,\ \varphi(\xi)=0,\\   
&g(\varphi(X),\varphi(Y))=g(X,Y)-\eta(X)\eta(Y),\ X,Y\in\Gamma(TM).
\end{align*}
Let $\{U_\alpha\}_{\alpha\in\Lambda}$ be an open cover of $M$, and let $\{\partial^{(\alpha)},\bar{\partial}^{(\alpha)}\}$ be a complex frame orthogonal to $\xi$ on $U_\alpha$.
We can always choose $\partial^{(\alpha)}$ so that $\varphi(\partial^{(\alpha)})=i\partial^{(\alpha)}$ holds.
Conversely, we can recover all the information about the ACM structure $(\varphi,\xi,\eta,g)$ from the frame $\{\xi,\partial^{(\alpha)},\bar{\partial}^{(\alpha)}\}$ on $U_\alpha\ (\alpha\in\Lambda)$.
Suppose that a frame $\{\xi,\partial^{(\alpha)},\bar{\partial}^{(\alpha)}\}$ is given on an orientable three-dimensional Riemannian manifold $(M,g)$, and $\xi$ is a global unit vector field.
Since the metric $g|_D$ is defined on the distribution $D$ orthogonal to $\xi$, we can uniquely choose a complex structure $J$ on $D$ such that $g|_D(J\cdot,J\cdot)=g|_D(\cdot,\cdot)$ and $J(\partial^{(\alpha)})=i\partial^{(\alpha)}$ on $U_\alpha$.
Defining an endomorphism field $\varphi$ by $\varphi|_D:=J$ and $\varphi(\xi)=0$, we obtain a unique ACM structure.
Thus we may discuss ACM structures on $M$ in terms of the frame $\{\xi,\partial^{(\alpha)},\bar{\partial}^{(\alpha)}\}$ and an open covering $\{U_\alpha\}_{\alpha\in\Lambda}$ on orientable three-dimensional Riemannian manifolds.
The spin coefficients and all other quantities and definitions are the same as in Section \ref{sec:riemannian NP formalism}.

An ACM manifold $(M,\varphi,\xi,\eta,g)$ is said to be normal if the Nijenhuis tensor of the almost complex structure defined on $M\times\mathbb{R}$ by $\varphi,\xi,\eta$ vanishes.
A necessary and sufficient condition for an ACM structure to be normal is given by
\begin{align*}
[\varphi,\varphi]+\xi\otimes d\eta=0,
\end{align*}
where $[\varphi,\varphi]$ denotes the Nijenhuis torsion tensor of $\varphi$.
Moreover, a three-dimensional ACM structure $(\varphi, \xi, \eta, g)$ is normal if and only if \cite{Matsuno2025}
\begin{align*}
\kappa=\sigma=0.
\end{align*}
This condition means that the flow of the characteristic vector field $\xi$ is a shear-free geodesic congruence.

Trans-Sasakian structures provide many examples and have been studied extensively.
An ACM structure is said to be trans-Sasakian of type $(\alpha,\beta_s)$ if
\begin{align*}
(\nabla_X\varphi)Y
=
\alpha\{g(X,Y)\xi-\eta(Y)X\}
+\beta_s\{g(\varphi X,Y)\xi-\eta(Y)\varphi X\}.
\end{align*}
In much of the literature, the symbol $\beta$ is used in place of $\beta_s$, but since $\beta$ is also used for a spin coefficient, we shall use $\beta_s$ to avoid confusion, where the subscript $s$ denotes "Sasakian".
For a three-dimensional almost contact metric structure, the trans-Sasakian condition is equivalent to
\begin{align*}
\kappa=\sigma=0,
\qquad
\rho=\Theta+i\omega=\beta_s+i\alpha.
\end{align*}
Therefore, in dimension three, normal ACM structures and trans-Sasakian structures are equivalent.
Moreover, in this case, we have
\begin{align*}
\dd\eta=-2i\omega\,\mu\wedge\bar\mu.
\end{align*}
The rank is understood pointwise: $\operatorname{rank}_p\eta=1$ if $\omega(p)=0$, and $\operatorname{rank}_p\eta=3$ if $\omega(p)\ne0$, and rank two cannot occur.
We call the structure \textbf{rank one} if $\omega\equiv0$, \textbf{rank three}
if $\omega$ is nowhere zero, and \textbf{mixed rank} if $\omega\not\equiv0$ but $\omega$ has a zero.
The trans-Sasakian structure is called an $\alpha$-Sasakian structure if $(\alpha,\beta_s)=(\omega,0)$.
In dimension three, an $\alpha$-Sasakian structure is also called a quasi-Sasakian structure.
In particular, an $\alpha$-Sasakian structure of twist $\omega=1$ is called a Sasakian structure.
A trans-Sasakian structure is called a $\beta_s$-Kenmotsu structure if $(\alpha,\beta_s)=(0,\Theta)$.
In particular, a $\beta_s$-Kenmotsu structure of expansion $\Theta=1$ is called a Kenmotsu structure.
A trans-Sasakian structure is called a cosymplectic structure if $(\alpha,\beta_s)=(0,0)$.

As seen from the following proposition, on a three-dimensional trans-Sasakian manifold the characteristic vector field defines a conformal foliation by geodesics.
\begin{proposition}\label{prop:normal perp metric lie derivative}
Let $(M^3,\varphi,\xi,\eta,g)$ be a three-dimensional trans-Sasakian manifold.
Then $\xi$ satisfies
\begin{equation}
\mathcal{L}_\xi g=2\Theta(g-\eta^2).
\label{eq:lie-g}
\end{equation}
\end{proposition}

\begin{proof}
A straightforward calculation gives
\begin{align}
&0=-2\sigma=2g(\nabla_{\partial}\xi,\partial)
=g(\nabla_{e_2}\xi,e_2)-g(\nabla_{e_3}\xi,e_3)
-i\{g(\nabla_{e_2}\xi,e_3)+g(\nabla_{e_3}\xi,e_2)\},\\
&2\Theta=2\textrm{Re}\, g(\nabla_{\bar{\partial}}\xi,\partial)
=g(\nabla_{e_2}\xi,e_2)+g(\nabla_{e_3}\xi,e_3).
\end{align}
Therefore,
\[
g(\nabla_{e_2}\xi,e_2)=g(\nabla_{e_3}\xi,e_3)=\Theta,
\qquad
g(\nabla_{e_2}\xi,e_3)+g(\nabla_{e_3}\xi,e_2)=0.
\]
This implies
\[
(\mathcal{L}_\xi g)(e_2,e_2)=(\mathcal{L}_\xi g)(e_3,e_3)=2\Theta,
\qquad
(\mathcal{L}_\xi g)(e_2,e_3)=0
\]
Furthermore, since $0=-\sqrt{2}\kappa=g(\nabla_\xi\xi,e_2)-ig(\nabla_\xi\xi,e_3)$, we obtain $(\mathcal{L}_\xi g)(\xi,\cdot)=0$.
\end{proof}

\section{$\eta$-Einstein metrics on universal covers}
\label{sec:normal-forms}

An ACM manifold $(M,\varphi,\xi,\eta,g)$ is said to be \textcolor{blue}{\textbf{$\eta$-Einstein}} if
\[
\operatorname{Ric}=ag+b\,\eta\otimes\eta,\qquad a,b\in C^\infty(M).
\]
In this section, we classify three-dimensional simply connected noncompact proper trans-Sasakian $\eta$-Einstein manifolds.

\subsection{Rank of structures}
The rank of a three-dimensional trans-Sasakian manifold is either one or three at each point and need not be constant on the manifold.
Zeros of the twist $\omega$ correspond to rank one, but the distribution of zeros is complicated in general, so there are also many papers that assume constant rank.
Fortunately, however, the $\eta$-Einstein condition imposes strong restrictions on the distribution of zeros of $\omega$, so we can obtain strong conclusions about the rank.
In this subsection, we prove the following Proposition~\ref{prop:rank classification}.

\begin{proposition}
\label{prop:rank classification}
Let $(M^3,\varphi,\xi,\eta,g)$ be a connected three-dimensional trans-Sasakian $\eta$-Einstein manifold with complete $\xi$. 

\begin{enumerate}
  \item If $\omega\equiv0$, then the rank is one (trivial).
  \item If $\Theta\equiv0$ and $\omega\not\equiv0$, then the rank is three and $\omega$ is a nonzero constant.
  \item If $\Theta\omega\not\equiv0$, then $K:=\Ric(\xi,\xi)$ is constant. Furthermore, (i) if $K\geq0$, then the rank is three, or (ii) if $K<0$, then the rank is three on a dense subset and the rank is one on its complement.
\end{enumerate}
\end{proposition}

\begin{lemma}
\label{lem:eta-einstein-rho}
Let $(M^3,\varphi,\xi,\eta,g)$ be a connected three-dimensional trans-Sasakian manifold. 
The structure is $\eta$-Einstein if and only if $\partial\rho=0$.
When it is $\eta$-Einstein, the twist $\omega$ satisfies $\Delta_g\omega=0$.
Consequently, either $\omega\equiv0$ or the subset $\{\omega\ne0\}$ is dense.
In the latter case, $K:=\Ric(\xi,\xi)$ is constant on all of $M$, and we have a Riccati equation
\begin{equation}
 \xi(\rho)=-\rho^2-\frac K2.
 \label{eq:riccati}
\end{equation}
\end{lemma}

\begin{proof}
Since $\kappa=\sigma=0$, the generalized Sachs equations \eqref{eq:generalized Sachs eq 2} and \eqref{eq:generalized Sachs eq 3} give $\Ric(\partial,\partial)=0$ and $\Ric(\xi,\partial)=-\partial\rho$.
Since $2\Ric(\partial,\partial)=\Ric(e_2,e_2)-\Ric(e_3,e_3)-2\ii\Ric(e_2,e_3)$, the $\eta$-Einstein condition is exactly $\partial\rho=0$.

The Sachs equation \eqref{eq:generalized Sachs eq 1} gives
$\xi\rho=-\rho^2-\Ric(\xi,\xi)/2$ and $\xi\omega=-2\Theta\omega$.
Since \(\partial\rho=0\) gives $\partial\omega=i\partial\Theta$, we have $\operatorname{grad}\omega=(\xi\omega)\xi+(\bar\partial\omega)\partial+(\partial\omega)\bar\partial
=-2\Theta\omega\,\xi-i(\bar\partial\Theta)\partial+i(\partial\Theta)\bar\partial$.
Furthermore,
\begin{align*}
g(\nabla_\partial\operatorname{grad}\omega,\bar\partial)
&=g\bigl(
-2\Theta\omega\,\nabla_\partial\xi
-i(\partial\bar\partial\Theta)\partial
-i(\bar\partial\Theta)\nabla_\partial\partial
+i(\partial\Theta)\nabla_\partial\bar\partial,\,
\bar\partial
\bigr)\\
&=g\bigl(
(-2\Theta\omega\bar\rho-i\partial\bar\partial\Theta-i\beta\bar\partial\Theta)\partial
+i(\partial\Theta)(-\bar\rho\xi-\beta\bar\partial),\,
\bar\partial
\bigr)\\
&=-2\Theta\omega\bar\rho-i\partial\bar\partial\Theta-i\beta\bar\partial\Theta.\\
g(\nabla_\xi\operatorname{grad}\omega,\xi)
&=g\bigl(
-2\xi(\Theta\omega)\xi
-2\Theta\omega\,\nabla_\xi\xi
-i(\bar\partial\Theta)\nabla_\xi\partial
+i(\partial\Theta)\nabla_\xi\bar\partial,\,
\xi
\bigr)\\
&=g\bigl(
-2\xi(\Theta\omega)\xi
-i\epsilon(\bar\partial\Theta)\partial
+i\bar\epsilon(\partial\Theta)\bar\partial,\,
\xi
\bigr)
=-2\xi(\Theta\omega)
\end{align*}
Thus, noting that $[\partial,\bar\partial]=(\rho-\bar\rho)\xi+\bar\beta\partial-\beta\bar\partial$, we obtain
\begin{align*}
\Delta_g\omega
&=g(\nabla_\partial\operatorname{grad}\omega,\bar\partial)+g(\nabla_{\bar\partial}\operatorname{grad}\omega,\partial)+g(\nabla_\xi\operatorname{grad}\omega,\xi)\\
&=-2\Theta\omega(\rho+\bar\rho)
 -i[\partial,\bar\partial]\Theta
 -i\beta\bar\partial\Theta+i\bar\beta\partial\Theta
 -2\xi(\Theta\omega)
=-4\Theta^2\omega-i(\rho-\bar\rho)\xi\Theta-2\xi(\Theta\omega)\\
&=-4\Theta^2\omega+2\omega\,\xi\Theta
 -2(\omega\,\xi\Theta+\Theta\,\xi\omega)
=-4\Theta^2\omega-2\Theta\,\xi\omega
=-4\Theta^2\omega-2\Theta(-2\Theta\omega)
=0.
\end{align*}



Unique continuation for the smooth elliptic operator $\Delta_g$ \cite{Aronszajn1957}, in the formulation of
\cite[Theorem~2.1]{BrandingMontaldoOniciucRatto}, now shows that if $\omega$ vanishes on a nonempty open set, then $\omega\equiv0$.

The second Bianchi identity \eqref{eq:2nd bianchi 1} gives $\partial K=0$.
Since $K$ is real, we also have $\bar\partial K=0$, and \eqref{eq:Lie bra 2} yields $0=[\partial,\bar\partial]K=2\ii\omega\,\xi K$.
Hence $\dd K=0$ on $\{\omega\ne0\}$.
If $\omega\not\equiv0$, this set is dense, so continuity gives $\dd K=0$ on all of $M$.
The Sachs equation then gives
\eqref{eq:riccati} with constant $K$.
\end{proof}

\begin{remark}
By the above proposition, a mixed rank structure cannot be obtained by gluing a rank-one open region to a rank-three open region.
\end{remark}

We can determine more precisely the range of values of $\rho=\Theta+i\omega$.
We first give the definitions of the terms needed for the discussion.
Let $\xi$ be a complete nowhere-vanishing vector field on a manifold $M$, and let $\Phi_t:\mathbb{R}\times M\to M$ be its flow.
Let $V\subset\R^2$ be an open set and $I\subset\R$ an open interval containing $0$.
Local coordinates $(x,y,t)\in V\times I$ on an open subset of $M$ such that $\xi=\partial_t$ are called a \textcolor{blue}{\textbf{flow box}}.
For $t_0\in I$, $V\times\{t_0\}$ is called a \textcolor{blue}{\textbf{transversal section}}.
$\{x\}\times I$ is called a \textcolor{blue}{\textbf{plaque}}.
We also write $\mathcal{O}(X)$ for the set of holomorphic complex-valued functions on $X$.

\begin{lemma}
\label{lem:transverse-zero}
Let $(M^3,\varphi,\xi,\eta,g)$ be a connected three-dimensional trans-Sasakian $\eta$-Einstein manifold with $\omega\not\equiv0$ and $K\leq0$.
In an arbitrary flow box $(x,y,t)\in V\times I$, the metric is written as
\begin{align*}
g=A(x,y,t)(\dd x^2+\dd y^2)+(\dd t+\alpha)^2,\\
A\in C^\infty(V\times I),\ \alpha\in\Omega^1(V)
\end{align*}
On the transversal section $\Sigma=V\times\{0\}$, set $F(x,y):=\rho(x,y,0)-\sqrt{|K|/2}$.
After shrinking $V$ if necessary, there exist $H\in C^\infty(V,\C),\  s\in\mathcal O(V)$ such that $F=e^{-H}\overline{s}$.
Thus, $F$ is either identically zero on $V$ or has isolated zeros.
For a sufficiently small positively oriented circle $\gamma\subset \Sigma$ around a point at which $s$ has a zero of order $m\geq1$, the winding number is
\begin{align*}
\operatorname{wind}(F\circ\gamma,0):=\frac{1}{2\pi\ii}\int_0^1
       \frac{(F\circ\gamma)'(t)}{F(\gamma(t))}\,\dd t
 =-m.
\end{align*}
\end{lemma}
\begin{proof}
Since $\mathcal L_\xi\eta=0$, we can write $\eta=\dd t+\alpha(x,y),\ \alpha\in\Omega^1(V)$ in the flow box $V\times I$.
By Proposition~\ref{prop:normal perp metric lie derivative}, the characteristic flow induces conformal transformations in the horizontal directions.
We take isothermal coordinates $(x,y)$ for $(g-\eta\otimes\eta)|_{\Sigma}$ and extend them so that they are constant along the flow of $\xi$.
Thus, if we take isothermal coordinates $\{x,y\}$ on each leaf of each transversal foliation $F_t:=\{V\times{t};\ t\in I\}$, the metric in the flow box $V\times I$ is written as $g=A(x,y,t)(\dd x^2+\dd y^2)+(\dd t+\alpha)^2$.
Here, write $\alpha=p(x,y)\dd x+q(x,y)\dd y$.

The horizontal complex vector field is $\partial=\frac{1}{\sqrt{2A}}   \left\{\partial_x-\ii\partial_y-(p-\ii q)\partial_t\right\}$.
Using $\partial\rho=0$ and the Riccati equation $\partial_t\rho=-(\rho-\sqrt{|K|/2})(\rho+\sqrt{|K|/2})$, and restricting to $t=0$, we obtain
\[
 (\partial_x-\ii\partial_y)F
 +(p-\ii q)(\rho|_{t=0}+\sqrt{|K|/2})F=0.
\]
Let $H\in C^\infty(V,\C)$ be a solution of the inhomogeneous Cauchy--Riemann equation $ (\partial_x-\ii\partial_y)H =(p-\ii q)(\rho|_{t=0}+\sqrt{|K|/2})$.
This equation always has a local solution, so shrink $V$ if necessary.
Then $(\partial_x-\ii\partial_y)(e^HF)=0$, so $e^HF$ is antiholomorphic.
Hence there exists a holomorphic function $s\in\mathcal{O}(V)$ such that $F=e^{-H}\bar{s}$.
If $s$ is not identically zero, the zeros of the holomorphic function are isolated.
The factor $e^{-H}$ extends without zeros to the interior of the circle, so its winding number is $0$.
If $s$ has a zero of order $m$, then $\bar{s}$ has winding number $-m$ around that zero.
The winding number of a product is the sum of the winding numbers, so the conclusion follows.
\end{proof}

\begin{proposition}
\label{prop:complete-riccati-branches}
Let $(M^3,\varphi,\xi,\eta,g)$ be a connected three-dimensional trans-Sasakian $\eta$-Einstein manifold with $\omega\not\equiv0$ and complete $\xi$.
Then the following hold.
\begin{enumerate}[label=\textup{(\roman*)}]
\item If $K>0$, then $\omega$ is nonzero at every point.
\item If $K=0$, then $\rho$ and $\omega$ are nonzero at every point.
\item If $K<0$, then $\rho(M)\cap \left(\left(-\infty,-\sqrt{|K|/2}\,\right]\cup\left[\sqrt{|K|/2},\infty\right)\right)=\varnothing$.
\end{enumerate}
In particular, mixed rank occurs only when $K<0$.
\end{proposition}
\begin{proof}
Since $\xi$ is smoothly defined on $M$, $\rho$ is finite and smooth everywhere on $M$.
Along each integral curve of $\xi$, the equation $\dot\rho=-\rho^2-K/2$ has a smooth solution for all real time, and a solution with a real initial value remains real-valued by the uniqueness theorem for ordinary differential equations.

Proof of (i)\\
For $K>0$, the solution with a real initial value is
\[
\rho(t)=\sqrt{\frac K2}\tan\left(\arctan\left(\sqrt{\frac2K}\rho(0)\right)-\sqrt{\frac K2}\,t\right).
\]
It diverges in finite time, so $\rho$ does not take real values.
Therefore, $\omega=\Im\rho$ is nonzero at every point.

Proof of (ii)\\
For $K=0$, the solution with a nonzero real initial value is $\rho(t)=\rho(0)/(1+t\rho(0))$, and it diverges in finite time.
  Therefore, in this case, the only real value that can be taken is $0$.
  Suppose that $\rho$ has a zero $p$, and take a flow box $V\times I$ containing that point.
  If necessary, change the coordinate on $I$ so that $p$ lies in the transversal section $\Sigma=V\times\{0\}$.
  Apply Lemma~\ref{lem:transverse-zero} in the case $K=0$, and set $F=\rho|_{t=0}$.
  If $F$ has an isolated zero, its winding number on a sufficiently small positively oriented circle around that zero is nonzero.
  On the other hand, $F$ avoids nonzero real values on the circle, so its argument can be chosen to be single-valued, and its winding number is $0$, which is a contradiction.
  If $F$ is identically zero on an open subset of the transversal section, then uniqueness for the Riccati equation gives $\rho=0$ on a nonempty open subset of $M$ obtained by moving that open subset by the flow for a short time.
  In particular, $\omega=0$ there, and unique continuation in Lemma~\ref{lem:eta-einstein-rho} gives $\omega\equiv0$, which also contradicts the assumption.
  Therefore, $\rho$ is nonzero and does not take any real values, so $\omega$ is also nonzero at every point.

Proof of (iii)\\
For $K<0$, a real solution $\rho$ with $|\rho(0)|>\sqrt{|K|/2}$ is $\rho(t)=\sqrt{|K|/2}\coth(\sqrt{|K|/2}t+\frac12\log\frac{\rho(0)+\sqrt{|K|/2}}{\rho(0)-\sqrt{|K|/2}})$, and it diverges in finite time, so this case is impossible.
  The remaining possibility to exclude is $\rho=\pm\sqrt{|K|/2}$.
  Suppose that there is a point at which $\rho=\sqrt{|K|/2}$, and set $F=(\rho-\sqrt{|K|/2})|_{t=0}$ on a transversal section through that point.
  By the preceding argument, $\rho$ does not take real values greater than $\sqrt{|K|/2}$, so $F$ avoids the positive real half-line.
  If there is an isolated zero, the winding number on a circle around it is nonzero by Lemma~\ref{lem:transverse-zero}, whereas using an argument on the slit plane gives $0$, which is a contradiction.
  If $F$ is identically zero on an open subset of the transversal section, then uniqueness for the Riccati equation gives $\rho=\sqrt{|K|/2}$ on a nonempty open subset of $M$, and hence $\omega=0$.
  Unique continuation in Lemma~\ref{lem:eta-einstein-rho} gives $\omega\equiv0$, which contradicts the assumption.
  The case $\rho=-\sqrt{|K|/2}$ is similar.
\end{proof}

By the above proposition, when $K\geq0$, we always assume $\omega>0$, making the replacement $(\xi,\eta)\mapsto(-\xi,-\eta)$ if necessary.

\begin{proof}[Proof of Proposition \ref{prop:rank classification}]
If $\omega\equiv0$, then the rank is trivially one.

Next suppose that $\Theta\equiv0$ and $\omega\not\equiv0$.
The $\eta$-Einstein condition $\partial\rho=0$ gives $0=\partial\rho=\ii\partial\omega$.
The imaginary part of \eqref{eq:riccati} gives $\xi(\omega)+2\omega\Theta=0$, which implies $\xi\omega=0$.
Thus $\omega$ is a nonzero constant, and the rank is then three.

Finally, suppose that $\Theta\omega\not\equiv0$.
By Lemma~\ref{lem:eta-einstein-rho}, $K$ is constant.
By Proposition~\ref{prop:complete-riccati-branches}, if $K\geq0$, the rank is three.
By Proposition~\ref{prop:complete-riccati-branches}, the mixed rank case can occur when $K<0$, and by Lemma~\ref{lem:eta-einstein-rho}, the rank is three on a dense subset and the rank is one on its complement.
\end{proof}

In the twist-free case, that is, for $\beta$-Kenmotsu manifolds, and in the expansion-free case, that is, for $\alpha$-Sasakian manifolds, $\eta$-Einstein metrics are already known.
Therefore, we are interested in the case that is neither of these, that is, when $\Theta\omega\not\equiv0$.
Following \cite{BeldjilaliDelloum2024}, we give this structure the following name.
\begin{definition}
Let $(M^3,\varphi,\xi,\eta,g)$ be a connected three-dimensional trans-Sasakian manifold.
This structure is called \textcolor{blue}{\textbf{proper trans-Sasakian}} if $\omega\Theta\not\equiv0$.
\end{definition}

\subsection{Twist-free or expansion-free case}
The results for the globally twist-free case and the globally expansion-free case have already been partially reported in the literature \cite{InoguchiLee2022,PrasadPankajTripathiShukla,ZhaoWangLiu2018}.
They are therefore not essentially new.
We describe them below for a self-contained presentation and for the reader's convenience.

Olszak \cite{Olszak} defined the horizontal conformal transformation of an adapted metric on a three-dimensional ACM manifold by
\begin{align}
g\mapsto g'=e^{2f}g+(1-e^{2f})\eta\otimes\eta,\qquad f\in C^\infty(M)
\end{align}
We then obtain the following proposition.
\begin{proposition}[\cite{Olszak}]\label{prop:trans-Sasakian to Sasakian}
Let $(M^3,\varphi,\xi,\eta,g)$ be a three-dimensional Sasakian manifold.
For every $f\in C^\infty(M,\R)$, the metric
\[
 g'=e^{2f}g+(1-e^{2f})\eta\otimes\eta
\]
together with the unchanged $\varphi,\xi,\eta$ defines a rank-three trans-Sasakian structure.
Conversely, let $(M^3,\varphi,\xi,\eta,g)$ be a three-dimensional rank-three trans-Sasakian manifold.
Replacing $\xi$ with $-\xi$ if necessary, we can assume $\omega>0$.
Then the same formula with $f=\tfrac12\log\omega$ defines a Sasakian structure.
\end{proposition}
\begin{proof}
Since the complex frame associated with $g'=e^{2f}g+(1-e^{2f})\eta\otimes\eta$ can be defined by $\xi'=\xi,\ \partial' = e^{-f}\partial$, we have
\begin{align*}
[\xi',\partial']
&= e^{-f}[\xi,\partial]-\xi(f)\partial'\\
&=e^{-f}\kappa\xi+(\epsilon-\bar{\rho})\partial'
+\sigma\bar{\partial}'-\xi(f)\partial',\\
[\partial',\bar\partial']
&=e^{-2f}[\partial,\bar\partial]
-e^{-2f}\partial(f)\bar\partial
+e^{-2f}\bar\partial(f)\partial\\
&=2ie^{-2f}\omega\xi
+e^{-f}\bigl(\bar\beta+\bar\partial(f)\bigr)\partial'
-e^{-f}\bigl(\beta+\partial(f)\bigr)\bar\partial'
\end{align*}
and hence
\begin{align*}
\rho' &= g'(\nabla'_{\bar\partial'}\xi, \partial') \\
&= \frac{1}{2}\left\{
g'([\bar\partial', \xi],\partial')
+g'([\partial', \bar\partial'],\xi')
-g'([\xi,\partial'], \bar\partial')
\right\}\\
&=\frac{1}{2}\{-(\bar{\epsilon}-\rho)+\xi(f)+2ie^{-2f}\omega-(\epsilon-\bar\rho)+\xi(f)\}\\
&= \Theta+\xi(f)+i\omega e^{-2f}.
\end{align*}
Therefore, we obtain
\begin{align*}
&\omega'=e^{-2f}\omega,\\
&\Theta'=\Theta+\xi(f)
\end{align*}

The first assertion of the proposition is immediate from these formulas.
For the second assertion, it suffices to show the existence of a transformation from any rank-three three-dimensional trans-Sasakian structure to a Sasakian structure.
We can assume $\omega>0$ by replacing $\xi$ with $-\xi$ if necessary.
Indeed, applying the horizontal conformal transformation with $f=\frac{1}{2}\log\omega$ gives $\omega'=1$.
Moreover, $\Theta'=\Theta+\frac{\xi(\omega)}{2\omega}$, while the imaginary part of the generalized Sachs equation \eqref{eq:generalized Sachs eq 1} gives $\xi(\omega)+2\omega\Theta=0$; hence $\Theta'=0$.
\end{proof}

\begin{proposition}\label{prop:eta-einstein simple case}
Let $(M^3,\varphi,\xi,\eta,g)$ be a connected three-dimensional trans-Sasakian $\eta$-Einstein manifold with complete $\xi$.
If the structure has rank one, that is, if $\omega\equiv0$, then the universal cover $(\widetilde{M},\tilde{\varphi},\tilde{\xi},\tilde{\eta},\tilde{g})$ can be written globally as
\[
\widetilde{M}=\mathbb{R}_z\times B,\qquad
\tilde{g}=dz^2+e^{2\Omega(z)}h,
\]
where $(B,h)$ is a connected simply connected Riemannian surface and $\Omega\in C^\infty(\R,\R)$.
If $\Theta\equiv0$ and $\omega\not\equiv0$, then $\omega$ is a nonzero constant and $g$ is homothetic to a Sasakian metric.
\end{proposition}

\begin{proof}
First, suppose that the structure has rank one.
Since $\omega=0$, we have $d\eta=0$.
Passing to the universal covering, we may write $\tilde{\eta}=dz$ for some smooth function $z$ on $\widetilde{M}$.
Since $\tilde{\eta}(\tilde{\xi})=1$, we have $\tilde{\xi}(z)=1$.
Since $\xi$ is complete, the flow $\Phi_t$ of $\tilde{\xi}$ is defined for all $t\in\mathbb{R}$.
Let $B:=z^{-1}(0)$.
Then $F:\mathbb{R}\times B\longrightarrow \widetilde{M},\ F(t,p):=\Phi_t(p)$ is a diffeomorphism.
Hence we may identify $\widetilde{M}$ with $\mathbb{R}\times B$ and write $\tilde{\xi}=\partial_z,\ \tilde{\eta}=dz$.
Since $\ker\tilde{\eta}$ is orthogonal to $\tilde{\xi}$, the metric has no mixed terms and is of the form $\tilde{g}=dz^2+h_z$, where $h_z$ is a Riemannian metric on $B$ depending on $z$.
The $\eta$-Einstein condition is $\partial\Theta=0$.
Therefore, $\Theta$ is constant along the leaves of $\ker\tilde{\eta}$, so $\Theta=\Theta(z)$.
Furthermore, by Proposition~\ref{prop:normal perp metric lie derivative}, for horizontal vector fields $X,Y\in\ker\tilde{\eta}$, we have $(\mathcal{L}_{\tilde{\xi}}\tilde{g})(X,Y)=2\Theta\,\tilde{g}(X,Y)$, which implies $\frac{d}{dz}h_z=2\Theta(z)h_z$.
Thus, we obtain $h_z=e^{2\Omega(z)}h$.
Therefore, the metric is given by $\tilde{g}=dz^2+e^{2\Omega(z)}h$.

Next, suppose that $\Theta\equiv0$ and $\omega\not\equiv0$.
By Proposition~\ref{prop:rank classification}, $\omega$ is a nonzero constant, and we can assume $\omega>0$.
The metric
$g'=\omega g+(1-\omega)\eta\otimes\eta$, obtained by the horizontal conformal transformation, is Sasakian by
Proposition~\ref{prop:trans-Sasakian to Sasakian}.
Applying the $D$-homothetic transformation with factor $\omega$ gives the
Sasakian metric $g''=\omega g'+(\omega^2-\omega)\eta\otimes\eta=\omega^2g$.
\end{proof}

\subsection{Proper case}
In this subsection, we classify metrics on the universal cover of three-dimensional noncompact proper trans-Sasakian $\eta$-Einstein manifolds.
For this purpose, we first discuss the existence of coordinates suitable for describing the metrics.
It is shown that the characteristic vector field $\tilde{\xi}$ on the universal cover induces a trivial principal $\R$-bundle structure and that we can choose a smooth two-dimensional foliation transverse to the fibers.
We then see that, taking $w$ as the fiber coordinate, we can construct global coordinates $x,y$ on the two-dimensional foliation.
To achieve this, different approaches are needed for the cases \(K\leq0\) and \(K>0\).

\begin{lemma}
\label{lem:canonical-nonpositive}
Let $(M^3,\varphi,\xi,\eta,g)$ be a connected three-dimensional trans-Sasakian $\eta$-Einstein manifold with $\omega\not\equiv0$, complete $\xi$ and $K:=\Ric(\xi,\xi)\leq0$.
For $K=0$, replace $\xi$ if necessary so that $\omega>0$.
Let $\Phi_t:\R\times M\to M$ be the flow of $\xi$.
Define real-valued functions $w$ and $U$ by the following formulas.
\begin{align}
w-\ii U=\rho^{-1},\qquad&(K=0),\\
w+\ii U=\frac{1}{\sqrt{2|K|}}\Log
    \left(\frac{\sqrt{|K|/2}+\rho}{\sqrt{|K|/2}-\rho}\right),\qquad&(K<0).
\label{eq:canonical-nonpositive}
\end{align}
Here, $\Log$ is the principal logarithm on $\C\setminus(-\infty,0]$, with $-\pi<\Im\Log<\pi$.
These functions are smoothly defined on all of $M$ and satisfy
\begin{equation}
 \xi w=1,\qquad \xi U=0,\qquad
 \begin{cases}
 U>0,&K=0,\\[1mm]
 \displaystyle |U|<\frac{\pi}{\sqrt{2|K|}},&K<0
  \end{cases}.
 \label{eq:canonical-nonpositive-properties}
\end{equation}
Set $B:=w^{-1}(0)$.
Then the map
\begin{equation}
 \R\times B\longrightarrow M,\qquad (t,p)\longmapsto\Phi_t(p)
 \label{eq:canonical-global-product}
\end{equation}
is a diffeomorphism, and the flow of $\xi$ is a free and proper $\R$-action.
\end{lemma}

\begin{proof}
For $K=0$, Proposition~\ref{prop:complete-riccati-branches} gives $\rho\ne0$, and noting \eqref{eq:riccati}, we have
\[
 \xi(\rho^{-1})=-\rho^{-2}\xi\rho=1,
  \qquad U=\frac{\omega}{|\rho|^2}>0.
\]
Thus, the functions $w$ and $U$ in the conclusion are globally defined.

For $K<0$, the same proposition shows that $f=\frac{\sqrt{|K|/2}+\rho}{\sqrt{|K|/2}-\rho}$ takes values in $\C\setminus(-\infty,0]$.
Also, since $\Im(f)=\frac{2\sqrt{|K|/2}\,\omega}{|\sqrt{|K|/2}-\rho|^2}$, at points where $f$ takes real values, $\rho$ is also real-valued.
When $\rho$ takes real values, we do not have $|\rho|>\sqrt{|K|/2}$, so $f<0$ does not occur.
Therefore, we can take the principal logarithm smoothly at every point.
Using the Riccati equation $\xi\rho=|K|/2-\rho^2$, we obtain
\begin{align*}
 \xi(f)
 =\frac{(\sqrt{|K|/2}-\rho)+(\sqrt{|K|/2}+\rho)}
          {(\sqrt{|K|/2}-\rho)^2}\,\xi\rho
 =\frac{2\sqrt{|K|/2}}{(\sqrt{|K|/2}-\rho)^2}
       \left(\frac{|K|}{2}-\rho^2\right)
 =\sqrt{2|K|}f.
\end{align*}
Thus, $\xi(w+\ii U)=1$.
The range of the principal logarithm also gives $|U|<\pi/\sqrt{2|K|}$.

In either case, $w(\Phi_t(p))=w(p)+t$ and $U(\Phi_t(p))=U(p)$.
Since $\xi w=1$, $0$ is a regular value of $w$, and $B=w^{-1}(0)$ is a smooth surface.
Moving each $p\in M$ along the flow of $\xi$ for time $-w(p)$ brings it to $B$, so $B$ is nonempty, and we obtain the diffeomorphism \eqref{eq:canonical-global-product}.
In these product coordinates, the flow is $(w,p)\mapsto(w+t,p)$ and is free and proper.
\end{proof}

\begin{lemma}
\label{prop:canonical-return-isometry}
Let $(M^3,\varphi,\xi,\eta,g)$ be a connected three-dimensional proper
trans-Sasakian $\eta$-Einstein manifold with complete $\xi$ and
$K:=\Ric(\xi,\xi)>0$.
Let $\Phi_t$ be the flow of $\xi$, and put $T:=2\pi/\sqrt{2K}$.
Then $\rho\circ\Phi_T=\rho$, and $\Phi_T$ preserves the full trans-Sasakian
structure:
\[
 \Phi_T^*g=g,\qquad \Phi_T^*\eta=\eta,\qquad
 \Phi_{T*}\xi=\xi,\qquad
 \Phi_{T*}\circ\varphi=\varphi\circ\Phi_{T*}.
\]
\end{lemma}

\begin{proof}
Proposition~\ref{prop:complete-riccati-branches} shows that $\omega$ is
nowhere zero, so we may assume $\omega>0$ after reversing $(\xi,\eta)$
if necessary.
Define the globally smooth function
\[
 q:=\frac{\rho-\ii\sqrt{K/2}}{\rho+\ii\sqrt{K/2}}.
\]
The denominator does not vanish because $\omega>0$.
The Riccati equation~\eqref{eq:riccati} gives
\[
 \xi q
 =\frac{\ii\sqrt{2K}}{(\rho+\ii\sqrt{K/2})^2}\,\xi\rho
 =-\ii\sqrt{2K}\,q.
\]
Consequently $q\circ\Phi_t=\e^{-\ii\sqrt{2K}t}q$, and hence
$q\circ\Phi_T=q$.
The inverse relation
$\rho=\ii\sqrt{K/2}(1+q)/(1-q)$ implies $\rho\circ\Phi_T=\rho$.
This calculation uses only the Riccati equation, not the existence of a
global orbit space or a normal form.

The identity $\Phi_{t*}\xi=\xi$ follows from the flow property.
Since $\iota_\xi\dd\eta=0$ and $\eta(\xi)=1$, Cartan's formula gives
$\Lie_\xi\eta=0$ and therefore $\Phi_t^*\eta=\eta$ for all $t\in\R$.
Put $h:=g-\eta^2$.
By \eqref{eq:lie-g} and the imaginary part of \eqref{eq:riccati},
\[
 \frac{\dd}{\dd t}\Phi_t^*h
 =2(\Theta\circ\Phi_t)\Phi_t^*h,\qquad
 \frac{\dd}{\dd t}(\omega\circ\Phi_t)
 =-2(\Theta\circ\Phi_t)(\omega\circ\Phi_t).
\]
Thus
\[
 \Phi_t^*h=\frac{\omega}{\omega\circ\Phi_t}\,h.
\]
Since $\omega\circ\Phi_T=\omega$, we obtain $\Phi_T^*h=h$ and
$\Phi_T^*g=g$.
Finally, $\Phi_T$ preserves the orientation determined by the ACM structure.
The endomorphism $\varphi$ is uniquely determined by the metric, $\xi$,
and the orientation of $\xi^\perp$, so $\Phi_{T*}$ commutes with $\varphi$.
\end{proof}

\begin{lemma}
\label{prop:averaged-invariant-metric}
Let $(M^3,\varphi,\xi,\eta,g)$ be a connected three-dimensional proper
trans-Sasakian $\eta$-Einstein manifold with complete $\xi$ and
$K:=\Ric(\xi,\xi)>0$.
Let $\Phi_t$ be its characteristic flow, put $T:=2\pi/\sqrt{2K}$, and define
\begin{equation}
 \bar g:=\frac1T\int_0^T\Phi_t^*g\,\dd t.
 \label{eq:averaged-metric}
\end{equation}
Then $\bar g$ is a smooth Riemannian metric for which $\xi$ is Killing,
and $\bar g(\xi,\cdot)=\eta$.
\end{lemma}

\begin{proof}
Proposition~\ref{prop:canonical-return-isometry} makes the integrand
in \eqref{eq:averaged-metric} $T$-periodic.
Hence, for every $s\in\R$,
\[
 \Phi_s^*\bar g
 =\frac1T\int_s^{s+T}\Phi_t^*g\,\dd t
 =\bar g.
\]
The average of smooth positive-definite tensors is smooth and positive
definite.
Moreover,
\[
 (\Phi_t^*g)(\xi,\cdot)
 =\Phi_t^*(g(\xi,\cdot))=\Phi_t^*\eta=\eta,
\]
so $\bar g(\xi,\cdot)=\eta$.
These assertions require only completeness of $\xi$.
\end{proof}

\begin{lemma}
\label{lem:simply-connected-open-killing-proper}
Let $(N^3,h)$ be a connected, simply connected, noncompact Riemannian manifold.
Every complete nowhere-vanishing Killing vector field $X$ generates a free proper $\R$-action.
\end{lemma}

\begin{proof}
Let $\Psi_t$ be the complete flow of $X$, and put $L:=\{\Psi_t:t\in\R\}\subset\Isom(N,h)$ and $H:=\overline L\subset\Isom(N,h)$.
$\Isom(N,h)$ acts properly on $N$ \cite[Corollary~1.4]{Manoussos2010}.
The closed subgroup $H$ therefore also acts properly.
Since $H$ is connected and abelian, $H\cong\R^a\times\T^b$.
The one-parameter subgroup has the form $\Psi_t=(tu,[tv]),\ (u\in\R^a,\ v\in\R^b)$.

We show that if $H$ is noncompact, then the action of $L$ is free and proper.
Since $H$ is noncompact, $a>0$, and the density of $L$ in $H$ implies $u\ne0$.
Thus $t\mapsto\Psi_t$ is a proper injective homomorphism with closed image isomorphic to $\R$.
The restriction of the proper $H$-action to this closed subgroup is proper.
Its stabilizers are compact, and $\R$ has no nontrivial compact subgroup, so the action is free.

We exclude compactness of $H$.
Suppose that $H$ is compact, and define the one-form
\[
 \alpha:=\frac{h(X,\cdot)}{h(X,X)}.
\]
We have $\alpha(X)=1,\ \Lie_X\alpha=0,\ \iota_X\dd\alpha=0$.
Every element of $H$ is an isometry commuting with the flow of $X$,
and therefore preserves both $X$ and $\alpha$.
Since $N$ is a connected noncompact three-manifold without boundary, we have $H^3_{\mathrm{dR}}(N)=0$.
We assume $\dd\alpha\ne0$.
The three-form $\alpha\wedge\dd\alpha$ is closed, so there is a smooth two-form $\beta$ such that $\dd\beta=\alpha\wedge\dd\alpha$.
Since $H$ is compact, we can take a Haar measure $\mu$ on $H$ such that $\mu(H)=1$.
We define
\[
 \bar\beta:=\int_H k^*\beta\,\dd\mu(k).
\]
Then $\bar\beta$ is smooth and $H$-invariant.
Thus, we have $\dd\bar\beta=\alpha\wedge\dd\alpha,\ \Lie_X\bar\beta=0$.
Set $\theta:=\alpha+\iota_X\bar\beta$.
Then we have $ \dd\theta=\dd\alpha+\Lie_X\bar\beta-\iota_X\dd\bar\beta =\dd\alpha-\iota_X(\alpha\wedge\dd\alpha)=0$ and $\theta(X)=1$.
If $\dd\alpha=0$, the same formula holds for $\theta:=\alpha$.
Simple connectedness gives $\theta=\dd f$ for some
$f\in C^\infty(N,\R)$.
Thus, we have $X(f)=1$.
Therefore, $f(\Psi_t(p))=f(p)+t\ (t\in\R)$.
This is impossible because the orbit is contained in the compact set $H\cdot p$, on which $f$ is bounded.
\end{proof}

\begin{proposition}
\label{prop:orbit-surface}
Let $(M^3,\varphi,\xi,\eta,g)$ be a connected, simply connected, noncompact three-dimensional proper trans-Sasakian $\eta$-Einstein manifold with complete $\xi$.
Then the flow of $\xi$ defines a free and proper $\R$-action.
The quotient $B:=M/\R$ is a simply connected noncompact oriented smooth surface, and the quotient map defines a trivial principal bundle $\R\to M\to B$.
In particular, $B\cong\R^2,\ M\cong B\times\R\cong\R^3$.
The horizontal conformal class descends to a conformal structure on $B$,
whose conformal type is $\D$ or $\C$.
\end{proposition}

\begin{proof}
Lemma~\ref{lem:eta-einstein-rho} shows that $K=\Ric(\xi,\xi)$ is constant.
If $K\leq0$, apply Lemma~\ref{lem:canonical-nonpositive} directly to the lifted structure.
It gives a free proper $\R$-action on $M$.
If $K>0$, Proposition~\ref{prop:averaged-invariant-metric} gives a smooth invariant metric $\bar g$ on $M$.
$\xi$ is a complete nowhere-vanishing Killing field with respect to $\bar g$.
Lemma~\ref{lem:simply-connected-open-killing-proper} therefore shows that its flow is free and proper.

Since the action is free and proper, the orbit space $B$ is a smooth
surface and $M\to B$ is a principal $\R$-bundle
\cite[Theorem~1.6.3]{Tammo}.
The homotopy exact sequence and contractibility of the fiber give
$\pi_1(B)\cong\pi_1(M)=0$.
In particular, $B$ is orientable.
The bundle admits a smooth global section and is trivial
\cite[Chapter~I, Theorem~5.7]{KobayashiNomizu}.

We prove that $B$ is noncompact.
First suppose that $K=\Ric(\xi,\xi)\geq0$, that is the structure has rank three.
In a principal-bundle trivialization $M\cong B\times\R$, write $\xi=\partial_t$.
The identities $\eta(\xi)=1$ and $\Lie_\xi\eta=0$ imply $ \eta=\dd t+\gamma,\ \gamma\in\Omega^1(B)$.
The rank-three condition implies that $\dd\gamma$ is a nowhere-vanishing two-form on $B$.
If $B$ were compact, orienting it by $\dd\gamma$ would give $ 0<\int_B\dd\gamma=0$, where the last equality is Stokes' theorem.
Thus $B$ is noncompact.

Next, suppose that $K<0$.
By Lemma~\ref{lem:canonical-nonpositive}, the globally defined functions $w,U$ satisfy \eqref{eq:canonical-nonpositive}.
We now use the canonical trivialization supplied by that lemma,
identifying $B$ with $w^{-1}(0)$.
Thus $M\cong\R\times B$, $\xi=\partial_w$, and $U$ is a function on $B$.
We can calculate $\omega$ as follows:
\begin{align}
&\frac{\sqrt{|K|/2}+\rho}{\sqrt{|K|/2}-\rho} = \e^{\sqrt{2|K|}(w+\ii U)},\nonumber\\
\rho
 &=
 \sqrt{\frac{|K|}{2}}\,
 \frac{\e^{\sqrt{2|K|}(w+\ii U)}-1}
      {\e^{\sqrt{2|K|}(w+\ii U)}+1}
 =
 \sqrt{\frac{|K|}{2}}\,
 \frac{\e^{2\sqrt{2|K|}w}-1
       +2\ii\e^{\sqrt{2|K|}w}\sin(\sqrt{2|K|}U)}
      {\e^{2\sqrt{2|K|}w}+1
       +2\e^{\sqrt{2|K|}w}\cos(\sqrt{2|K|}U)}\nonumber\\
 &=
 \sqrt{\frac{|K|}{2}}\,
 \frac{\sinh(\sqrt{2|K|}w)+\ii\sin(\sqrt{2|K|}U)}
      {\cosh(\sqrt{2|K|}w)+\cos(\sqrt{2|K|}U)},\label{eq:rho K negative}\\
\omega(w,p)
 &=
 \sqrt{\frac{|K|}{2}}\,
 \frac{\sin(\sqrt{2|K|}\,U(p))}
      {\cosh(\sqrt{2|K|}\,w)+\cos(\sqrt{2|K|}\,U(p))},
 \qquad p\in B.\nonumber
\end{align}
The bound $|\sqrt{2|K|}\,U(p)|<\pi$ implies, for each fixed $p$, the denominator of $\omega(w,p)$ is minimized at $w=0$.
Consequently, we have $|\omega(w,p)|\leq|\omega(0,p)|$.

Suppose that $B$ is compact.
The restriction $\omega|_B$ attains a maximum or minimum at some point $p_0\in B$.
Therefore, $\omega$ also attains a maximum or minimum at $p_0$.
Lemma~\ref{lem:eta-einstein-rho} gives $\Delta_g\omega=0$.
The strong maximum principle therefore implies that $\omega$ is constant, a contradiction.
Hence $B$ is noncompact.

Since $B$ is a simply connected noncompact smooth surface,
it is diffeomorphic to $\R^2$.
The bundle trivialization then gives
$M\cong B\times\R\cong\R^3$.

Finally, $\Lie_\xi\eta=0$ makes the horizontal distribution
flow-invariant, and \eqref{eq:lie-g} shows that its metric changes
conformally along each fiber.
Its conformal class therefore descends to $B$.
Together with the induced orientation, this defines a Riemann surface
structure on $B$.
The uniformization theorem gives conformal type $\D$ or $\C$.
\end{proof}

We now define the model spaces.
Let $K\in\R$ be a constant, and let $\D=\{z\in\C:|z|<1\}$ be the unit disk with its standard complex structure.
Let $z=x+\ii y$ and $\Delta_0=\partial_x^2+\partial_y^2$.
An \textbf{admissible pair} for $K>0$ is a pair
$(v,s)\in C^\infty(\D,\R)\times\mathcal O(\D)$ satisfying
\begin{equation}
 \Delta_0v>0,\qquad e^v|s|<1,\qquad s\not\equiv0.
 \label{eq:admissible-positive}
\end{equation}
An \textbf{admissible potential} for $K=0$ is a function
$U\in C^\infty(\D,\R)$ satisfying
\begin{align}
 U>0,\qquad \Delta_0U<0.
 \label{eq:admissible-zero}
\end{align}
An \textbf{admissible potential} for $K<0$ is a function
$U\in C^\infty(\D,\R)$ satisfying
\begin{align}
 U\not\equiv0,\qquad |U|<\frac{\pi}{\sqrt{2|K|}},
\end{align}
and for which there exists a function $H_U\in C^\infty(\D,(0,\infty))$ satisfying
\begin{align}
 \Delta_0U=\sqrt{2|K|}\,H_U\sin(\sqrt{2|K|}U).
 \label{eq:admissible-negative-variable-rank}
\end{align}
We denote the respective sets of data by $\mathcal U_K^{\mathrm{pair}}$ for $K>0$
and $\mathcal U_K^{\mathrm{pot}}$ for $K\leq0$.

\begin{remark}
In the $K<0$ case, let $U\in C^\infty(\D,\R)$ be an admissible potential.
By definition,
\begin{equation}
 H_U=\frac{\Delta_0U}{\sqrt{2|K|}\sin(\sqrt{2|K|}U)}
\end{equation}
is smooth even at the zeros of $U$, but this is not true for arbitrary $U$.
Thus, \eqref{eq:admissible-negative-variable-rank} may also be regarded as the condition that there exists a smooth $H_U$ defined by this quotient.

Such an $H_U$ is also unique if it exists.
Indeed, fix a given pair of functions \(U,H_U\), and define $V(z):=2|K|H_U(z)\int_0^1\cos\!\left(t\sqrt{2|K|}\,U(z)\right)\,\dd t$.
This is a smooth function on \(\D\).
Then we have
\[
\begin{aligned}
V(z)U(z)
&=\sqrt{2|K|}H_U(z)\int_0^1
\frac{\dd}{\dd t}\sin\!\left(t\sqrt{2|K|}\,U(z)\right)\,\dd t\\
&=\sqrt{2|K|}H_U(z)\sin\!\left(\sqrt{2|K|}\,U(z)\right)
=\Delta_0U(z).
\end{aligned}
\]
Consider the linear elliptic equation $\Delta_0u-Vu=0$ for an unknown function \(u\).
By the above calculation, \(U\) is a solution of this equation.
If \(U\) is identically zero on a nonempty open set, then Aronszajn's unique continuation theorem~\cite{Aronszajn1957} gives \(U\equiv0\) on the whole connected domain \(\D\).
Therefore, if \(U\not\equiv0\), the zero set contains no nonempty open set, and \(\{U\neq0\}\) is dense in \(\D\).
\end{remark}

Take an admissible pair $(v,s)$ if $K>0$, or an admissible potential $U$ if $K\leq0$, on the unit disk $\D$.
On $\D_z\times\R_w$, we define the one-form $\eta_K$ as
\begin{equation}
 \eta_K=
 \begin{cases}
  \dd w-(2K)^{-1/2}*_0\dd v,&K>0,\\
  \dd w+*_0\dd U,&K=0,\\
  \dd w-*_0\dd U,&K<0.
 \end{cases}
 \label{eq:model-eta}
\end{equation}

On $\D_z\times\R_w$, we define the metrics as follows.
If $K>0$,
\begin{equation}
 \begin{aligned}
 &g_K=A_{K,v,s}|\dd z|^2+\eta_K^2,\\
 &A_{K,v,s}=\frac{\Delta_0v}{2K}\frac{|1-q|^2}{1-|q|^2},
 \qquad q=e^{v-\ii\sqrt{2K}w}\overline{s}.
 \end{aligned}
 \label{eq:model-positive}
\end{equation}
If $K=0$,
\begin{equation}
\begin{aligned}
&g_0=A_{0,U}|\dd z|^2+\eta_0^2,\\
&A_{0,U}=-\frac{\Delta_0U}{2U}(w^2+U^2),
\end{aligned}
\label{eq:model-zero}
\end{equation}
If $K<0$,
\begin{equation}
 \begin{aligned}
 &g_K=A_{K,U}|\dd z|^2+\eta_K^2,\\
 &A_{K,U}=H_U\{\cosh(\sqrt{2|K|}w)+\cos(\sqrt{2|K|}U)\}.
 \end{aligned}
 \label{eq:model-negative}
\end{equation}
Set $\xi=\partial_w$ and let $\varphi_K$ be the horizontal lift of the
standard complex structure, with $\varphi_K\xi=0$.
We denote these ACM models by $\mathcal M_K(v,s)$ for $K>0$ and
$\mathcal M_K(U)$ for $K\leq0$.

\begin{screen}
\begin{theorem}
\label{prop:Standard eta-Einstein metrics}
Let $(M^3,\varphi,\xi,\eta,g)$ be a connected, simply connected, noncompact three-dimensional proper trans-Sasakian $\eta$-Einstein manifold with complete $\xi$.
If $K\geq0$, replace $(\xi,\eta)$ by $(-\xi,-\eta)$ if necessary so that $\omega>0$.
Then the structure is isomorphic to $\mathcal M_K(v,s)$ for $K>0$, or to $\mathcal M_K(U)$ for $K\leq0$, with some admissible data.
In particular, $M\cong\D\times\R\cong\R^3$.
Conversely, every such model is proper trans-Sasakian $\eta$-Einstein with $\Ric(\xi,\xi)=K$.
\end{theorem}
\end{screen}

\begin{proof}
Proposition~\ref{prop:orbit-surface} gives a trivial principal $\R$-bundle over the simply connected oriented surface $B:=M/\R\simeq\R^2$, with its induced conformal structure $\D$ or $\C$.
Since $\Lie_\xi\eta=0$, in a principal-bundle trivialization we can write
\begin{equation}
 g=h_t+(\dd t+\alpha)^2,\qquad \xi=\partial_t,
 \qquad \alpha\in\Omega^1(B),
 \label{eq:pre-normal-form}
\end{equation}
where $h_t$ is a family of Riemannian metrics on $B$.
By Proposition~\ref{prop:normal perp metric lie derivative}, all $h_t$ have the same conformal structure.
Thus, in a global isothermal coordinate $z=x+\ii y$ on $B$,
$h_t=A(z,t)(\dd x^2+\dd y^2)$.
The one-form $\alpha$ and, for $K\leq0$, the canonical potential $U$ are globally defined on $B$.
For $K\leq0$, choose the canonical product of
Lemma~\ref{lem:canonical-nonpositive}.

\medskip
\noindent\textbf{(i) $K=0$ case.}\\
By Lemma \ref{lem:canonical-nonpositive}, $\rho=(w-\ii U)^{-1}$ with $U>0$.
The equation $\partial\rho=0$ is equivalent to $\partial(w-\ii U)=0$.
Let $X=\partial_x-\alpha_x\partial_w$, $Y=\partial_y-\alpha_y\partial_w$. Then we obtain $0=(X-\ii Y)(w-\ii U)=(\partial_x-\alpha_x\partial_w)(w-\ii U)-\ii(\partial_y-\alpha_y\partial_w)(w-\ii U)=-(\alpha_x+U_y)+\ii(\alpha_y-U_x)$.
Thus we have $\alpha=-U_y\dd x+U_x\dd y=*_0(U_x\dd x+U_y\dd y)=*_0\dd U$.
Moreover, $\Theta=w/(w^2+U^2)$, so integrating the horizontal metric equation
$\partial_wh_w=2\Theta h_w$ gives
\[
 g=(w^2+U^2)h'+(\dd w+*_0\dd U)^2.
\]
Writing $h'=H(\dd x^2+\dd y^2)$ with $H>0$, we have $\mu\wedge\bar\mu=-\ii(w^2+U^2)H\dd x\wedge\dd y$ and $\dd\eta=\dd(-U_y\dd x+U_x\dd y)=(U_{xx}+U_{yy})\dd x\wedge\dd y$.
The identities $d\eta=-2\ii\omega\mu\wedge\bar\mu$ and $\omega=U/(w^2+U^2)$ give $(U_{xx}+U_{yy})\dd x\wedge\dd y=-2\ii\omega\,\mu\wedge\bar\mu=-2UH\dd x\wedge\dd y$. It follows that $\Delta_0U=-2HU<0$.
If $B=\C$, $-U$ is a bounded-above subharmonic function, and by the planar Liouville theorem \cite[Corollary~3.5.4]{ArmitageGardiner}, $U$ must be constant.
This gives the contradiction $0=\Delta_0U=-2HU<0$.
Thus $B=\D$.
Solving for $H$ gives \eqref{eq:model-zero}.

\medskip
\noindent\textbf{(ii) $K<0$ case.}\\
By Lemma \ref{lem:canonical-nonpositive}, we have
\begin{align*}
\rho
=\sqrt{\frac{|K|}{2}}\,
\frac{e^{\sqrt{2|K|}(w+iU)}-1}
     {e^{\sqrt{2|K|}(w+iU)}+1}
=\sqrt{\frac{|K|}{2}}\,
\tanh\!\left(\sqrt{\frac{|K|}{2}}(w+iU)\right),
 \qquad |U|<\frac{\pi}{\sqrt{2|K|}}.
 \label{eq:standard-rho-negative}
\end{align*}
The derivative of the right-hand side with respect to $w+\ii U$ is $|K|/2-\rho^2$, which is nonzero by Proposition~\ref{prop:complete-riccati-branches}, and thus $0=\partial\rho=(|K|/2-\rho^2)\partial(w+\ii U)$ implies $\partial(w+\ii U)=0$.
As in the $K=0$ case, we have $\alpha=U_y\dd x-U_x\dd y=-*_0\dd U$.
By \eqref{eq:rho K negative}, we have
\begin{equation}
 \Theta=\sqrt{\frac{|K|}{2}}\frac{\sinh(\sqrt{2|K|}w)}{D(w,U)},\qquad
 \omega=\sqrt{\frac{|K|}{2}}\frac{\sin(\sqrt{2|K|}U)}{D(w,U)},
 \label{eq:standard-expansion-negative}
\end{equation}
where $D(w,U)=\cosh(\sqrt{2|K|}w)+\cos(\sqrt{2|K|}U)$.
Integrating the horizontal metric equation gives
\begin{equation}
 g=D(w,U)h'+(\dd w-*_0\dd U)^2,\qquad
 h'=H(\dd x^2+\dd y^2),
\end{equation}
where $h'$ is a globally defined Riemannian metric on $B$.
$\dd\eta=-\Delta_0U \dd x\wedge\dd y$ is also given by
\begin{align*}
\mathrm d\eta
&=-2i\omega\,\mu\wedge\bar\mu
=-2i\sqrt{\frac{|K|}{2}}\frac{\sin(\sqrt{2|K|}U)}{D(w,U)}
  \bigl(-iD(w,U)H\,\mathrm dx\wedge\mathrm dy\bigr)\\
&=-\sqrt{2|K|}\,H\sin(\sqrt{2|K|}U)\,\mathrm dx\wedge\mathrm dy,
\end{align*}
which implies $\Delta_0 U=\sqrt{2|K|}\sin(\sqrt{2|K|}U)H$.
The potential is not identically zero because $\omega\not\equiv0$.

If $B=\C$, we have $\Delta_0(U^2)=2(U_x^2+U_y^2) +2\sqrt{2|K|}\,UH\sin(\sqrt{2|K|}U)\geq0$.
The bound on $U$ makes $U^2$ a bounded-above subharmonic function on $\C$.
The planar Liouville theorem used in \textup{(i)} again forces $U\equiv0$, a contradiction.
Therefore $B=\D$, and $H=H_U$ gives \eqref{eq:model-negative}.

\medskip
\noindent\textbf{(iii) $K>0$.}\\
Define a new variable by
\begin{align*}
q:=\frac{\rho-\ii \sqrt{K/2}}{\rho+\ii \sqrt{K/2}}.
\end{align*}
The Riccati equation $\xi\rho=-(\rho-\ii \sqrt{K/2})(\rho+\ii \sqrt{K/2})$ and \(\partial\rho=0\) give
\[
  \partial q=0,\qquad
  \xi q
  =\frac{\ii\sqrt{2K}}{(\rho+\ii \sqrt{K/2})^2}\xi\rho
  =-\ii \sqrt{2K}\frac{\rho-\ii \sqrt{K/2}}{\rho+\ii \sqrt{K/2}}
  =-\ii\sqrt{2K} q.
\]
Thus, for some \(F\in C^\infty(B,\C)\), we can write $q(z,t)=\e^{-\ii\sqrt{2K} t}F(z)$.
Since the horizontal basis can be written as $ X:=\partial_x-\alpha_x\partial_t,\  Y:=\partial_y-\alpha_y\partial_t$, \(\partial q=0\) gives
\[
  0=(X-\ii Y)q
  =\e^{-\ii\sqrt{2K} t}
   \left\{(\partial_x-\ii\partial_y)F
   +\ii\sqrt{2K}(\alpha_x-\ii\alpha_y)F\right\}.
\]
Here, the inhomogeneous Cauchy--Riemann equation for $\mathcal H\in C^\infty(B,\C)$
\[
  (\partial_x-\ii\partial_y)\mathcal H
  =\ii\sqrt{2K}(\alpha_x-\ii\alpha_y),\qquad
  \mathcal H=u+\ii\vartheta,
\]
can be solved globally on \(B=\D\) or $\C$
\cite{Forster}.
Then $(\partial_x-\ii\partial_y)(\e^{\mathcal H}F)=0$, so for some \(s\in\mathcal O(B)\), we can write
\[
  \e^{\mathcal H}F=\overline{s}.
\]
Setting $v:=-u,\ w:=t+\frac{\vartheta}{\sqrt{2K}}$, we can write
\[
  q=\e^{v-\ii\sqrt{2K} w}\overline{s}.
\]
$w$ is a new global fiber coordinate.

Taking the real and imaginary parts of the equation for \(\mathcal H\) gives $u_x+\vartheta_y=\sqrt{2K}\alpha_y,\  \vartheta_x-u_y=\sqrt{2K}\alpha_x$, so $\eta$ is
\[
\begin{aligned}
  \eta
  =\dd t+\alpha
  =\dd w +\alpha-\frac{1}{\sqrt{2K}}\dd\vartheta
  =\dd w +\frac{1}{\sqrt{2K}} *_0\dd u
  =\dd w-\frac{1}{\sqrt{2K}} *_0\dd v.
\end{aligned}
\]
At this point, $\dd\eta=-\frac{\Delta_0v}{\sqrt{2K}}\dd x\wedge\dd y$.
Also, since \(g=A|\dd z|^2+\eta^2\), we have $\mu\wedge\bar\mu=-\ii A\,\dd x\wedge\dd y$ and $\dd\eta=-2\ii\omega\,\mu\wedge\bar\mu=-2\omega A\,\dd x\wedge\dd y$.
Therefore, $\Delta_0v=2\sqrt{2K}\omega A>0$.
From the inverse transformation of \(q=(\rho-\ii \sqrt{K/2})/(\rho+\ii \sqrt{K/2})\), we have
\begin{align}
\rho&=\ii\sqrt{\frac{K}{2}}\frac{1+q}{1-q},\label{eq:K positive rho q}\\
\omega&=\sqrt{\frac{K}{2}}\frac{1-|q|^2}{|1-q|^2},\nonumber
\end{align}
and hence
\[
  A
  =\frac{\Delta_0v}{2\sqrt{2K}\omega}
  =\frac{\Delta_0v}{2K}
    \frac{|1-q|^2}{1-|q|^2}.
\]
Also, \(q=(\rho-\ii \sqrt{K/2})/(\rho+\ii \sqrt{K/2})\) gives
\[
  |q|^2
  =\frac{\Theta^2+(\omega-\sqrt{K/2})^2}{\Theta^2+(\omega+\sqrt{K/2})^2}<1.
\]
Therefore, $\e^v|s|=|q|<1$.
Thus, $\Delta_0v>0$ and $\e^v|s|<1$ hold on $B$.

Finally, we exclude \(B=\C\).
At a point where $s=0$, we have $q=0$, so $\rho=i\sqrt{K/2}$.
Since $s\in\mathcal{O}(B)$, either the zeros of $s$ are isolated or $s\equiv0$.
If $s\equiv0$, then $\rho\equiv i\sqrt{K/2}$, so the structure is $\alpha$-Sasakian, contradicting the assumption. Therefore, $s\not\equiv0$, and $\{s\ne0\}$ is dense in $B$.
Setting $b:=\e^{2v}|s|^2=|q|^2$, we have \(0\leq b<1\) and \(b\not\equiv0\).
Writing \(v_z=(v_x-\ii v_y)/2\), we obtain
\begin{align*}
\partial_zb
&=\partial_z\!\left(\e^{2v}s\overline s\right)
=2v_z\e^{2v}s\overline s+\e^{2v}s'\overline s
=\e^{2v}\overline s\left(s'+2v_zs\right)\\
\partial_{\bar z}\partial_zb
&=\e^{2v}
 \left(2v_{\bar z}\overline s+\overline{s'}\right)
 \left(s'+2v_zs\right)
 +2\e^{2v}|s|^2v_{z\bar z}
=\e^{2v}|s'+2v_zs|^2
 +2\e^{2v}|s|^2v_{z\bar z}\\
\Delta_0b
&=4\e^{2v}|s'+2v_zs|^2
 +8\e^{2v}|s|^2v_{z\bar z}
=4\e^{2v}|s'+2v_zs|^2
 +2\e^{2v}|s|^2\Delta_0v\\
&=2b\,\Delta_0v
    +4\e^{2v}|s'+2v_zs|^2\geq0.
\end{align*}
If \(B=\C\), then \(b\) is a bounded-above subharmonic function, so it is constant by the Liouville theorem.
However, this contradicts the fact that \(\{s\neq0\}\neq\varnothing\) and \(\Delta_0b>0\) there.
Thus, \(B=\D\).
We have already shown that $s\not\equiv0$, so $(v,s)\in\mathcal U_K^{\mathrm{pair}}$,
and \eqref{eq:model-positive} follows.

\medskip
\noindent\textbf{Conversely.}
For any admissible data, the coefficient $A$ in each model is smooth and strictly positive.
$\xi=\partial_w$ is unit and satisfies $\iota_\xi\dd\eta=0$, hence $\xi$ is geodesic.
Since $g-\eta^2=A|\dd z|^2$, $\xi$ is shear-free, so the structures are trans-Sasakian.
By following the above calculations in reverse, we also see that it is \(\eta\)-Einstein and $\omega\Theta\not\equiv0$.
\end{proof}

Theorem~\ref{prop:Standard eta-Einstein metrics} assumes only completeness of $\xi$, not geodesic completeness of $g$.
Geodesic completeness of the model metrics is an additional condition.
The criteria and complete examples are given in Appendix~\ref{app:Geodesically complete model spaces}.

\begin{corollary}
\label{cor:properties of K positive model space}
Let $\mathcal M_K(v,s)$ be a model space for $K>0$, and let $p\in\D$ be a zero of $s$.
The characteristic flow line $\ell=\{p\}\times\R$ satisfies
$\Theta|_\ell\equiv0$ and $\omega|_\ell\equiv\sqrt{K/2}$.
Moreover, this structure is $\alpha$-Sasakian if and only if $s\equiv0$.
\end{corollary}
\begin{proof}
Equation~\eqref{eq:K positive rho q} gives
\begin{align*}
\rho(z,w)&=\ii\sqrt{\frac K2}\frac{1+q(z,w)}{1-q(z,w)}\\
q(z,w)&=\e^{v(z)-\ii\sqrt{2K}w}\overline{s(z)}.
\end{align*}
If \(s(p)=0\), then \(q(p,w)=0\), so $\rho(p,w)=\ii\sqrt{K/2}$.

If \(s\equiv0\), then \(q\equiv0\), so \(\Theta\equiv0\), and \(\mathcal M_K(v,s)\) is \(\alpha\)-Sasakian.
Also, since
\[
\rho
=\frac{-\sqrt{2K}\Im q+\ii\sqrt{K/2}(1-|q|^2)}{|1-q|^2}
\]
we have
\[
\Theta=-\sqrt{2K}\frac{\Im q}{|1-q|^2}.
\]
If \(\mathcal M_K(v,s)\) is \(\alpha\)-Sasakian, then \(\Theta\equiv0\), so \(\Im q(z,w)\equiv0\).
For each fixed \(z\in\D\), if \(\e^{v(z)}\overline{s(z)}\neq0\), then \(q(z,w)=\e^{v(z)}\overline{s(z)}\e^{-\ii\sqrt{2K}w}\) traces a circle centered at the origin as \(w\) varies, so it cannot always take real values.
Therefore, \(\e^{v(z)}\overline{s(z)}=0\), and \(s(z)=0\). This holds for every \(z\in\D\), so \(s\equiv0\).
\end{proof}

\begin{corollary}
\label{cor:mixed rank-locus}
Let $\mathcal M_K(U)$ be a model space for $K<0$.
Then this structure has rank one on $U^{-1}(0)\times\R$ and rank three on its complement.
The interior of the zero set $\mathrm{Int}(U^{-1}(0))$ is empty.
\end{corollary}
\begin{proof}
The formula for $\omega$ in \eqref{eq:standard-expansion-negative} and
$|\sqrt{2|K|}U|<\pi$ give $\omega=0$ exactly when $U=0$.
Since $\omega\not\equiv0$, Lemma~\ref{lem:eta-einstein-rho} shows that $\{\omega\ne0\}$ is dense. Hence $\textrm{Int}(U^{-1}(0))=\emptyset$.
\end{proof}

Flows of vector fields are broadly classified as Palais-regular or Palais-nonregular.
The definition is as follows.

\begin{definition}[Palais regularity]
\label{def:palais-regularity}
The vector field $\xi$ is called \textcolor{blue}{\textbf{regular in the sense of Palais}} if, for every point $p\in M$, there exists a flow box $U$ containing $p$ such that each orbit of $\xi$ meets $U$ in at most one plaque \cite{Palais} \cite[Definition~2, p.~93]{Morimoto1964}.
A vector field that is not Palais-regular is called Palais-nonregular.
\end{definition}

\begin{corollary}
Let $(M^3,\varphi,\xi,\eta,g)$ be a connected noncompact three-dimensional proper trans-Sasakian $\eta$-Einstein manifold with complete $\xi$.
On a universal cover of $M$, the lifted characteristic vector field $\tilde{\xi}$ is regular.
\end{corollary}
\begin{proof}
The universal covering space of $M$ is isomorphic to one of the model spaces.
On model spaces, the characteristic flow action is free and proper.
A free proper $\R$-action defines a principal bundle, whose local trivializations give distinguished flow boxes.
Hence $\tilde{\xi}$ is Palais-regular.
\end{proof}

In Appendix \ref{app:Palais-nonregular quotient constructions}, we shall investigate the Palais-nonregular case.

\section{Structure-preserving isomorphism}
\label{sec:Structure-preserving isomorphism}

Let $(M_i,\varphi_i,\xi_i,\eta_i,g_i),\ (i=1,2)$ be ACM manifolds.
We call a diffeomorphism $f:M_1\to M_2$ which preserves their ACM structures a \textcolor{blue}{\textbf{structure-preserving isomorphism}}.
We call the group formed by structure-preserving isomorphisms from the model space $\mathcal M_K$ to itself the \textcolor{blue}{\textbf{structure-preserving automorphism group}} and denote it by \textcolor{blue}{$\StrAut(\mathcal M_K)$}.
The following theorem follows immediately from Theorem~\ref{prop:Standard eta-Einstein metrics}.

\begin{screen}
\begin{theorem}
\label{thm:global-quotient-model}
Let $(M^3,\varphi,\xi,\eta,g)$ be a connected noncompact proper trans-Sasakian $\eta$-Einstein manifold with complete $\xi$ and $K=\Ric(\xi,\xi)$.
If $K\geq0$, replace $(\xi,\eta)$ by $(-\xi,-\eta)$ if necessary so that $\omega>0$.
Then we have
\[
 M\cong
 \begin{cases}
 \mathcal M_K(v,s)/\Gamma,&K>0,\\
 \mathcal M_K(U)/\Gamma,&K\leq0,
 \end{cases}
\]
where $\Gamma$ is a discrete subgroup of the corresponding structure-preserving automorphism group acting freely and properly discontinuously.
\end{theorem}
\end{screen}

In the rest of this section, we determine the structure-preserving automorphisms of the model spaces $\mathcal{M}_K(v,s)$ and $\mathcal{M}_K(U)$.
We write \textcolor{blue}{$\Aut(\D)$} for the group of holomorphic automorphisms of $\D$, and $\mathcal{O}^*(\D)$ for the set of holomorphic functions with no zeros on $\D$.

\subsection{\texorpdfstring{$K>0$}{K > 0} case}\label{subsec:straut gp K positive case}
On the model space \(\mathcal M_K(v,s)\), we define the \textcolor{blue}{\textbf{singular potential}} $\mathfrak u_{v,s}:\D\times\R\to\R\cup\{-\infty\}$ by
\[
  \mathfrak u_{v,s}:=v+\log|s|.
\]

\begin{lemma}
\label{lem:singular-potential-gauge}
Let \(v,V\in C^\infty(\D,\R)\) and \(s,S\in\mathcal O(\D)\setminus\{0\}\).
If $v+\log|s|=V+\log|S|$, then \(s\) and \(S\) have the same zero set and the same order at each zero.
Therefore, \(h:=S/s\) extends to an element of \(\mathcal O^*(\D)\), and
\(V=v-\log|h|\) and \(S=hs\).
The converse also holds.
Thus, admissible pairs with the same singular potential are mapped to one another by this transformation.
\end{lemma}

\begin{proof}
Comparing both sides, we see that $s$ and $S$ have the same zero set.
If we write $s=(z-p)^d a,\  S=(z-p)^e b,\ a(p)b(p)\neq0$ near a common zero \(p\), then $(d-e)\log|z-p|=V-v+\log|b|-\log|a|$.
The right-hand side is smooth at \(z=p\), so \(d=e\).
Thus, \(S/s=b/a\) is holomorphic and nowhere zero even at the zeros, and
the original equality gives \(V=v-\log|h|\).
The converse is trivial.
\end{proof}

\begin{proposition}
\label{prop:generalized-rigidity}
A structure-preserving isomorphism \(F:\mathcal M_K(v,s)\to\mathcal M_K(V,S)\) is given by
\begin{align}
F_{\psi,L}(z,w):=\left(\psi(z),w+\frac{1}{\sqrt{2K}}\Im L(z)\right),\\
\psi\in\Aut(\D),\ 
\mathfrak u_{v,s}=\mathfrak u_{V,S}\circ\psi,\ 
e^L=\frac{s}{S\circ\psi}\in \mathcal O^*(\D).
\end{align}
\end{proposition}

\begin{proof}
Let \(F:\mathcal M_K(v,s)\to\mathcal M_K(V,S)\) be a structure-preserving isomorphism.
\begin{enumerate}
  \item We can write \(F(z,w)=(\psi(z),w+f(z))\), and \(\psi\in\Aut(\D)\).\\
  \hspace*{2em} Proof : Since $F_*\partial_w=\partial_w$, we can write $F(z,w)=(\psi(z),w+f(z))$.
  Here, $f\in C^\infty(\D,\R)$, and $\psi:\D\to\D$ is a diffeomorphism.
  Preservation of $\eta_K$ implies that $F_*$ preserves the horizontal distribution.
  Since $\varphi_K$ on the horizontal distribution is the lift of the standard complex structure on the base,
  projecting $F_*\varphi_K=\varphi_KF_*$ to the base gives $\dd\psi\circ\ii=\ii\circ\dd\psi$.
  Therefore, $\psi$ is holomorphic, and $\psi\in\Aut(\D)$.

  \item We have \(L:=V\circ\psi-v+\ii\sqrt{2K} f\in\mathcal O(\D)\).\\
  \hspace*{2em} Proof :  By preservation of \(\eta\) and the fact that pullback by a holomorphic map commutes with the Hodge star on one-forms, we have
  \begin{align*}
    &\eta_{v,s}=\dd w-\frac{1}{\sqrt{2K}} *_0\dd v
    =F^*\eta_{V,S}\\
    &=\dd(w+f)-\frac{1}{\sqrt{2K}}\psi^*(*_0\dd V)
    =\dd w+\dd f-\frac{1}{\sqrt{2K}} *_0\dd(V\circ\psi)\\
    &\therefore\ \dd f=\frac{1}{\sqrt{2K}} *_0\dd(V\circ\psi-v).
  \end{align*}
  In general, for a complex function $h=a+ib$, the Cauchy--Riemann equations can be written as $\dd b=\ast_0\dd a$, so $L:=V\circ\psi-v+\ii\sqrt{2K} f$ satisfies the Cauchy--Riemann equations.

  \item We have $\mathfrak u_{v,s}=\mathfrak u_{V,S}\circ\psi,\ \e^L=\frac{s}{S\circ\psi}\in\mathcal O^*(\D)$ and $F(z,w)=\left(\psi(z),w+\frac{\Im L(z)}{\sqrt{2K}}\right)$.\\
  \hspace*{2em} Proof :  Since \(F\) is structure-preserving, $    \rho_{v,s}=\rho_{V,S}\circ F$.
  Hence, by \eqref{eq:K positive rho q}, we have
  \[
    \frac{1+\e^{v-\ii\sqrt{2K} w}\overline s}
         {1-\e^{v-\ii\sqrt{2K} w}\overline s}
    =
    \frac{1+\e^{V\circ\psi-\ii\sqrt{2K}(w+f)}
                    \overline{S\circ\psi}}
         {1-\e^{V\circ\psi-\ii\sqrt{2K}(w+f)}
                    \overline{S\circ\psi}}.
  \]
  The map \(q\mapsto(1+q)/(1-q)\) is injective, so we obtain
  \begin{align*}
    &\e^{v-\ii\sqrt{2K} w}\overline s
    =\e^{V\circ\psi-\ii\sqrt{2K}(w+f)}\overline{S\circ\psi}\\
    &\therefore\ s=\e^{V\circ\psi-v+\ii\sqrt{2K} f}(S\circ\psi)
    =\e^L(S\circ\psi).
  \end{align*}
  Also, $\mathfrak u_{v,s}=v+\log|s|=v+\Re L+\log|S\circ\psi|=V\circ\psi+\log|S\circ\psi|=\mathfrak u_{V,S}\circ\psi$.
  Finally, since \(\Im L=\sqrt{2K} f\), we also see that $F(z,w)=\left(\psi(z),w+\frac{1}{\sqrt{2K}}\Im L(z)\right)$.
\end{enumerate}
Finally, we verify that the metric is preserved.
Since preservation of $\eta_K$ is clear from reversing the above calculations, it suffices to verify that the horizontal part is preserved.
Since $L$ and $\psi$ are holomorphic, we have $\Delta_0v=\Delta_0(V\circ\psi-\Re L) =|\psi'|^2(\Delta_0V)\circ\psi$.
Thus we obtain
\begin{align*}
 &F^*(A_{K,V,S}|dz|^2)=F^*\!\left(
   \frac{\Delta_0V}{2K}
   \frac{|1-\e^{V-\ii\sqrt{2K}w}\overline S|^2}
        {1-\e^{2V}|S|^2}|\dd z|^2
   \right)\\
 &=\frac{|\psi'|^2(\Delta_0V)\circ\psi}{2K}
   \frac{|1-\e^{V\circ\psi-\ii\sqrt{2K}(w+f)}\overline{S\circ\psi}|^2}
        {1-\e^{2V\circ\psi}|S\circ\psi|^2}|\dd z|^2
 =\frac{\Delta_0v}{2K}
   \frac{|1-\e^{v-\ii\sqrt{2K}w}\overline s|^2}
        {1-\e^{2v}|s|^2}|\dd z|^2\\
  &=A_{K,v,s}|dz|^2
\end{align*}
\end{proof}

Let \((v,s)\) be an admissible pair, let \(\mathfrak u:=\mathfrak u_{v,s}\) be its singular potential, and set $G_{\mathfrak u}:=\{\psi\in\Aut(\D):\mathfrak u\circ\psi=\mathfrak u\}$.
Then the set
\begin{equation}
  \left\{(\psi,L):\psi\in G_{\mathfrak u},\ 
     e^L=\frac{s}{s\circ\psi}\in\mathcal O^*(\D)\right\}
\end{equation}
can be identified with \(\StrAut(\mathcal M_K(v,s))\) under the correspondence \((\psi,L)\mapsto F_{\psi,L}\), and the composition is given by
\begin{align*}
 F_{\psi,L}\circ F_{\phi,N}(z,w)
 &=F_{\psi,L}\bigl(\phi(z),w+\lambda^{-1}\Im N(z)\bigr),\ (\lambda=\sqrt{2K})\\
 &=\bigl(\psi\circ\phi(z),
         w+\lambda^{-1}\Im N(z)+\lambda^{-1}\Im L(\phi(z))\bigr)\\
 &=\bigl(\psi\circ\phi(z),
         w+\lambda^{-1}\Im(N+L\circ\phi)(z)\bigr)
=F_{\psi\circ\phi,\,N+L\circ\phi}.
\end{align*}
Thus, the product in this representation is
\begin{equation}
  (\psi,L)(\phi,N)=(\psi\circ\phi,N+L\circ\phi).
\end{equation}
The identity is $(\id,0)$, and the inverse is $(\psi,L)^{-1}=(\psi^{-1},-L\circ\psi^{-1})$.

\subsection{\texorpdfstring{$K\leq0$}{K <= 0} case}

\begin{proposition}
\label{prop:nonpositive-isomorphisms}
A structure-preserving diffeomorphism $F:\mathcal M_K(U)\to\mathcal M_K(V)$ is given by $F(z,w)=(\psi(z),w)$, where $\psi\in\Aut(\D)$ satisfies $U=V\circ\psi$.
Therefore, we have
\[
 \StrAut(\mathcal M_K(U)):=\{\psi\in\Aut(\D):U\circ\psi=U\}.
\]
\end{proposition}
\begin{proof}
As in the proof of Proposition~\ref{prop:generalized-rigidity}, we can write $F(z,w)=(\psi(z),w+f(z)),\ f\in C^\infty(\D,\R),\ \psi\in\Aut(\D)$.
Also, since $\psi$ is a diffeomorphism, $|\psi'|^2>0$.

Put $\tau_0=1$, and $\tau_K=-1$ for $K<0$.
Since pullback by a holomorphic map commutes with the Hodge star on one-forms, the condition that $F$ preserves $\eta_K$ yields the following identities.
\begin{align}
 &\dd w+\tau_K*_0\dd U
 =F^*(\dd w+\tau_K*_0\dd V)
 =\dd w+\dd f+\tau_K*_0\dd(V\circ\psi)\\
 &\dd f=\tau_K*_0\dd(U-V\circ\psi)\\
 &\Delta_0U=\Delta_0(V\circ\psi)
 =|\psi'|^2(\Delta_0V)\circ\psi.
 \label{eq:nonpositive-eta-comparison}
\end{align}
Furthermore, preservation of $g_K$ and $\eta_K$ implies that the horizontal coefficients satisfy the following relation for every $w\in\R$.
\begin{equation}
 A_{K,V}(\psi(z),w+f(z))|\psi'(z)|^2=A_{K,U}(z,w)
 \label{eq:nonpositive-horizontal-comparison}
\end{equation}

When $K=0$, using \eqref{eq:nonpositive-eta-comparison} to write out \eqref{eq:nonpositive-horizontal-comparison} gives
\begin{align*}
-\frac{\Delta_0U}{2(V\circ\psi)}
       \{(w+f)^2+(V\circ\psi)^2\}
 =-\frac{\Delta_0U}{2U}(w^2+U^2)\\
 \frac{1}{V\circ\psi}w^2
 +\frac{2f}{V\circ\psi}w
 +\frac{f^2}{V\circ\psi}+V\circ\psi
 =\frac{1}{U}w^2+U.
\end{align*}
Comparing the coefficients of $w^2$ and $w$ gives $U=V\circ\psi$ and $f=0$.

When $K<0$, equation~\eqref{eq:nonpositive-horizontal-comparison} is
\begin{align*}
 &(H_V\circ\psi)|\psi'|^2
 \{\cosh(\sqrt{2|K|}(w+f))+
                  \cos(\sqrt{2|K|}(V\circ\psi))\}
 =H_U\{\cosh(\sqrt{2|K|}w)+\cos(\sqrt{2|K|}U)\}.
\end{align*}
We have
\begin{align*}
 \frac{\textrm{L.H.S}}{(H_V\circ\psi)|\psi'|^2}
 =\cosh(\sqrt{2|K|}f)\cosh(\sqrt{2|K|}w)
 +\sinh(\sqrt{2|K|}f)\sinh(\sqrt{2|K|}w)
 +\cos(\sqrt{2|K|}(V\circ\psi)).
\end{align*}
Comparing the coefficients of $\sinh(\sqrt{2|K|}w)$ gives $ (H_V\circ\psi)|\psi'|^2\sinh(\sqrt{2|K|}f)=0$. Hence $f=0$.
Next, comparing the coefficients of $\cosh(\sqrt{2|K|}w)$ and the constant terms gives
\begin{align*}
H_U=(H_V\circ\psi)|\psi'|^2,\quad
\cos(\sqrt{2|K|}U)&=\cos(\sqrt{2|K|}(V\circ\psi)).
\end{align*}
Also, \eqref{eq:nonpositive-eta-comparison} and the admissibility condition give
\begin{align*}
 &\sqrt{2|K|}\,H_U\sin(\sqrt{2|K|}U)
 =\Delta_0U
 =|\psi'|^2(\Delta_0V)\circ\psi
 =\sqrt{2|K|}\,|\psi'|^2(H_V\circ\psi)
                       \sin(\sqrt{2|K|}(V\circ\psi))\\
 &=\sqrt{2|K|}\,H_U\sin(\sqrt{2|K|}(V\circ\psi)).
\end{align*}
Thus, $\sin(\sqrt{2|K|}U)=\sin(\sqrt{2|K|}(V\circ\psi))$.
Since both the sines and cosines agree, $e^{\ii\sqrt{2|K|}U} =e^{\ii\sqrt{2|K|}(V\circ\psi)}$. Also, both arguments belong to $(-\pi,\pi)$ by the admissibility condition, so $U=V\circ\psi$.
Thus, for $K\le0$, we have $F(z,w)=(\psi(z),w)$.
Also, reversing the above calculations makes it clear that a map of this form is a structure-preserving diffeomorphism.
\end{proof}

\section{Conclusions and discussions}
In this paper, we classified three-dimensional noncompact trans-Sasakian $\eta$-Einstein manifolds under the assumption that the characteristic vector field is complete.
In the proper case, the universal cover is diffeomorphic to $\mathbb D\times\mathbb R$, and its metric is explicitly described by admissible data depending on the sign of $K=\operatorname{Ric}(\xi,\xi)$.
We also determined the structure-preserving isomorphisms between the model spaces and obtained a description of the corresponding quotient manifolds.

These results complement the classification in the compact case obtained in \cite{Matsuno2025}.
Unlike in the compact case, proper trans-Sasakian $\eta$-Einstein manifolds exist in the noncompact case and have infinitely many degrees of freedom.
Furthermore, as the explicitly constructed families of geodesically complete models show, this flexibility is not lost even when geodesic completeness is imposed.
On the other hand, there are strong restrictions on the topology of the universal cover and the conformal type of the orbit space of the characteristic flow.
Mixed rank structures are also possible when $K<0$, and we constructed geodesically complete examples of them.

Furthermore, we also considered the characteristic flow on quotient manifolds.
When the characteristic flow is Palais-nonregular, its orbit structure is classified into the three types (SF), (IR$_0$), and (IR$_1$).
We also characterized the Deck groups corresponding to each type using automorphisms of the model spaces.

\section*{Acknowledgements}
I had many discussions with Fumihiro Ueno, and he also checked the calculations.
He pointed out ways to improve several calculations and made many contributions toward completing this paper.
I express our gratitude here.

During the preparation of this manuscript, AI was used to check proofs, search the literature for prior work, identify existing theorems applicable to the proofs, and proofread the English. AI was not used to generate any theorems or proofs.

This work was partly supported by MEXT Promotion of Distinctive Joint Research Center Program JPMXP0723833165 and Osaka Metropolitan University Strategic Research Promotion Project (Development of International Research Hubs).

\appendix

\section{Geodesically complete model spaces}
\label{app:Geodesically complete model spaces}
Here, we give sufficient conditions for geodesic completeness of the model spaces and explicit complete families.

\begin{proposition}
\label{prop:radial-completeness-criterion}
Let the model space be $\mathcal M_K(v,s)$ or $\mathcal M_K(U)$.
Define $b_K\in C^0(\D,\R_{>0})$ by
\begin{equation}
 b_K(z)=
 \begin{cases}
 \displaystyle\frac{\Delta_0v(z)}{2K}
       \frac{1-e^{v(z)}|s(z)|}{1+e^{v(z)}|s(z)|},&K>0,\\
 \displaystyle-\frac12U(z)\Delta_0U(z),&K=0,\\
 H_U(z)\{1+\cos(\sqrt{2|K|}U(z))\},&K<0.
  \end{cases}
 \label{eq:horizontal-lower-bound}
\end{equation}
Then, if $b_K(z)|\dd z|^2$ is complete, the corresponding model is geodesically complete.
In particular, for $0\leq r<1$, we define $m_\bullet\in C^0([0,1),\R_{>0})$ as
\begin{align}
 m_+(r)&=\min_{|z|=r}
   \left\{\Delta_0v\frac{1-e^v|s|}{1+e^v|s|}\right\},
 \label{eq:mplus-complete}\\
 m_0(r)&=\min_{|z|=r}\{-U\Delta_0U\},
 \label{eq:mzero-complete}\\
 m_-(r)&=\min_{|z|=r}\{\sqrt{2|K|}H_U(1+\cos(\sqrt{2|K|}U))\}.
 \label{eq:mminus-complete}
\end{align}
For some $r_0\in(0,1)$, if we have
\begin{equation}
 \int_{r_0}^1\sqrt{m_\bullet(r)}\,\dd r=\infty,
 \label{eq:radial-length-divergence}
\end{equation}
then the corresponding model space is geodesically complete.
\end{proposition}

\begin{proof}
Write $g=A_K(z,w)|\dd z|^2+(\dd w+\alpha(z))^2$.
First, we show that $A_K(z,w)\geq b_K(z)$ for every $w\in\R$.
For $K>0$, we have $|q|=e^v|s|<1$, and the claim follows from
\[
 \frac{|1-q|^2}{1-|q|^2}
 \geq\frac{(1-|q|)^2}{(1-|q|)(1+|q|)}
 =\frac{1-|q|}{1+|q|}.
\]
For $K=0$, it suffices to use $w^2+U^2\geq U^2$ and $-\Delta_0U/(2U)>0$.
For $K<0$, we have
\[
 A_K=H_U\{\cosh(\sqrt{2|K|}w)+\cos(\sqrt{2|K|}U)\}
 \geq H_U\{1+\cos(\sqrt{2|K|}U)\}=b_K.
\]

Suppose that a piecewise \(C^1\) curve $\gamma:[0,T)\longrightarrow\D\times\R,\ \gamma(t)=(z(t),w(t))$ has finite length with respect to $g$.
Since $|\dot\gamma|_g^2=A_K(z,w)|\dot z|^2+\bigl(\dot w+\alpha_z(\dot z)\bigr)^2\ge b_K(z)|\dot z|^2$,
for \(0\le s<t<T\), we have
$d_{b_K}(z(s),z(t))\le\int_s^t\sqrt{b_K(z(u))}\,|\dot z(u)|\,\dd u\le\int_s^t|\dot\gamma(u)|_g\,\dd u\to0\ (t,s\to T)$.
Thus, \(z(t)\) is Cauchy with respect to \(d_{b_K}\), and by the completeness assumption there exists \(z_*\in\D\) satisfying $d_{b_K}(z(t),z_*)\longrightarrow0$.
Let $C$ be a sufficiently small compact set containing $z_*$.
On $C$, $d_{b_K}$ can be bounded below by a constant multiple of the Euclidean distance, so \(z(t)\to z_*\) also holds in the Euclidean sense.
Since \(\alpha\) is smooth and \(C\) is compact, there exists a suitable constant $c_1$ such that $|\dot w|\le c_1|\dot\gamma|_g$.
Therefore, $|w(t)-w(s)|\le c_1\int_s^t|\dot\gamma(u)|_g\,\dd u
\longrightarrow0,\ (t,s\to T)$, so \(w(t)\) is Cauchy in the real numbers and hence converges to \(w_*\in\R\).
Thus, $\gamma(t)\longrightarrow(z_*,w_*)\in\D\times\R$, and every curve of finite \(g\)-length has an endpoint.
Thus, $g$ is geodesically complete.

Set
\[
 C_K=\begin{cases}(2K)^{-1},&K>0,\\1/2,&K=0,\\(2|K|)^{-1/2},&K<0\end{cases}.
\]
Then $b_K(z)\geq C_Km_\bullet(|z|)$.
Set $J(r)=\int_0^r\sqrt{m_\bullet(u)}\,\dd u$.
Along any curve $z(t)$ of finite $d_{b_K}$-length, we have almost everywhere
\[
 \left|\frac{\dd}{\dd t}J(|z(t)|)\right|
 \leq\sqrt{m_\bullet(|z(t)|)}|\dot z(t)|
 \leq C_K^{-1/2}\sqrt{b_K(z(t))}|\dot z(t)|.
\]
Therefore, $J(|z(t)|)$ is bounded, and $z(t)$ stays in some closed disk $\{|z|\leq r_*\}\Subset\D$.
Since $b_K$ has a positive lower bound, a finite-length curve also has an endpoint limit at an interior point of $\D$ in the Euclidean sense.
Thus, the length metric of $b_K|\dd z|^2$ is complete, and the first part of the argument can be applied.
\end{proof}

\begin{proposition}
\label{prop:explicit-complete-families}
Set $z=x+\ii y\in\D,\ r=|z|$, and for $B_0>3/2$ and sufficiently small $\varepsilon\in\R$, define
\begin{align*}
V_\varepsilon(z)=\frac{(1+\varepsilon h(z))}{\sqrt{B_0-\log(1-r^2)}}.
\end{align*}
Here, $h\in C^\infty(\overline\D,\R)$ is fixed arbitrarily.
\begin{enumerate}[label=\textup{(\roman*)}]
\item Assume $K>0$.
For every $s\in\mathcal O(\D),\ s\not\equiv0$ satisfying $|s|\leq1$,
\[
 (v,s)=(-\sqrt{2K}\,V_\varepsilon,s)
\]
is an admissible pair, and $\mathcal M_K(v,s)$ is geodesically complete.
In particular, for $N\in\Z_{\geq0}$, $\theta\in\R$, distinct points
$a_1,\ldots,a_N\in\D$, and positive integers $m_1,\ldots,m_N$, choosing the finite Blaschke product
\[
 s(z)=e^{\ii\theta}\prod_{j=1}^N
       \left(\frac{z-a_j}{1-\overline{a_j}z}\right)^{m_j}
\]
realizes $s$ with zeros at $a_i$ of order $m_i$.

\item Assume $K=0$. $U=V_\varepsilon$ is an admissible potential, and
$\mathcal M_0(U)$ is geodesically complete.

\item Assume $K<0$ and $B_0>\max\{3/2,2|K|/\pi^2\}$.
Then $U=\pi/\sqrt{2|K|}-V_\varepsilon$ is a positive admissible potential, and
$\mathcal M_K(U)$ has rank three and is geodesically complete.

\item Assume $K<0$ and $h(-x,y)=h(x,y)$.
Fix $0<\delta<\pi/\sqrt{2|K|}$.
Set
\begin{equation}
 a_\varepsilon=1-V_\varepsilon,\qquad
 U_\varepsilon=\delta x\,a_\varepsilon,\qquad
 Q_\varepsilon=
 \frac{\Delta_0a_\varepsilon+2\,\partial_xa_\varepsilon/x}{a_\varepsilon}.
 \label{eq:mixed-complete-data}
\end{equation}
At $x=0$, extend $\partial_xa_\varepsilon/x$ smoothly.
For sufficiently small $|\varepsilon|$, we have $0<a_\varepsilon<1$ and
$Q_\varepsilon>0$, and the function defined by
\begin{equation}
 H_{U_\varepsilon}
 =Q_\varepsilon\frac{U_\varepsilon}{\sqrt{2|K|}\sin(\sqrt{2|K|}U_\varepsilon)},\qquad
 H_{U_\varepsilon}|_{x=0}=\frac{Q_\varepsilon}{2|K|}
 \label{eq:mixed-complete-H}
\end{equation}
is a smooth positive function.
$U_\varepsilon$ is an admissible potential, and $\mathcal M_K(U_\varepsilon)$ is
a geodesically complete mixed rank model.
Its rank is one on $x=0$ and three on $x\ne0$.
\end{enumerate}
\end{proposition}

\begin{proof}
Define
\begin{equation}
 L(r)=B_0-\log(1-r^2),\qquad F(r)=L(r)^{-1/2},\qquad
 V_\varepsilon(z)=(1+\varepsilon h(z))F(r).
\end{equation}
We first give the estimates common to all items. Direct differentiation gives
\begin{equation}
 \Delta_0F
 =\frac{2\{\tfrac32r^2-L\}}{(1-r^2)^2L^{5/2}}<0,
 \label{eq:radial-F-laplacian}
\end{equation}
and, as $r\to1$,
\begin{equation}
 -\Delta_0F\sim\frac{2}{(1-r^2)^2L^{3/2}},\qquad
 F(-\Delta_0F)\sim\frac{2}{(1-r^2)^2L^2}.
 \label{eq:radial-F-asymptotics}
\end{equation}
The function $h$ and its derivatives of finite order are bounded on the closed disk.
Furthermore,
\[
 \Delta_0V_\varepsilon=\Delta_0F(1+\varepsilon R_h),\qquad
 R_h=h+\frac{2\langle\nabla F,\nabla h\rangle+F\Delta_0h}{\Delta_0F}.
\]
Here, $R_h$ is bounded on the disk.
Indeed, near the boundary,
\[
 \frac{|\nabla F|}{-\Delta_0F}=O(1-r^2),\qquad
 \frac{F}{-\Delta_0F}=O((1-r^2)^2L),
\]
and both tend to zero.
Thus, taking $|\varepsilon|$ sufficiently small gives
\begin{equation}
 \tfrac12F\leq V_\varepsilon\leq\tfrac32F,\qquad
 -\Delta_0V_\varepsilon\geq\tfrac12(-\Delta_0F),\qquad
 V_\varepsilon(-\Delta_0V_\varepsilon)
       \geq\tfrac14F(-\Delta_0F).
 \label{eq:complete-family-comparison}
\end{equation}
Also, since $L'(r)=2r/(1-r^2)$,
\begin{equation}
 \int^1\sqrt{F(-\Delta_0F)}\,\dd r
 \asymp\int^1\frac{\dd r}{(1-r^2)L(r)}
 \asymp\int^\infty\frac{\dd L}{L}=\infty.
 \label{eq:complete-family-integral}
\end{equation}
Here, $\asymp$ means that the integrands are bounded above and below by positive constant multiples of each other near the endpoint.

\medskip
\noindent\textbf{\textup{(i)} $K>0$ case}\\
We have $\Delta_0v=\sqrt{2K}(-\Delta_0V_\varepsilon)>0,\ 
 e^v|s|\leq e^{-\sqrt{2K}V_\varepsilon}<1$.
Thus, the admissibility condition holds.
The function $t\mapsto(1-t)/(1+t)$ is decreasing for $0\leq t<1$, so $ m_+(r)\geq\sqrt{2K}\min_{|z|=r} \left\{(-\Delta_0V_\varepsilon) \tanh(\sqrt{2K}V_\varepsilon/2)\right\}$.
The convergence $V_\varepsilon\to0$ as $r\to1$ is uniform, and $\tanh t/t\to1$, so near the boundary we have $m_+(r)\geq C_+F(r)(-\Delta_0F(r))$ for some constant $C_+>0$.
Applying \eqref{eq:complete-family-integral} and Proposition~\ref{prop:radial-completeness-criterion} gives completeness.
Also, a finite Blaschke product satisfies \(|s|\leq1\) on \(\D\), and its order of vanishing at each \(a_j\) is \(m_j\).

\medskip
\noindent\textbf{\textup{(ii)} $K=0$ case}\\
By \eqref{eq:complete-family-comparison}, $U>0$ and $\Delta_0U<0$, and $m_0(r)=\min_{|z|=r}V_\varepsilon(-\Delta_0V_\varepsilon) \geq\tfrac14F(r)(-\Delta_0F(r))$, so completeness follows from the same divergent integral.

\medskip
\noindent\textbf{\textup{(iii)} $K<0$ and rank three case}\\
The maximum of $F$ is $B_0^{-1/2}<\pi/\sqrt{2|K|}$, so taking $|\varepsilon|$ still smaller gives $0<V_\varepsilon<\pi/\sqrt{2|K|}$.
Thus, $U=\pi/\sqrt{2|K|}-V_\varepsilon$ satisfies $0<U<\pi/\sqrt{2|K|}$ and $\Delta_0U=-\Delta_0V_\varepsilon>0$.
Then $H_U:=\frac{\Delta_0U}{\sqrt{2|K|}\sin(\sqrt{2|K|}U)}>0$ is a smooth function, and \(U\) is an admissible potential.
By \(U>0\) and Corollary~\ref{cor:mixed rank-locus}, the model has rank three.
Since $\sqrt{2|K|}\,H_U\bigl(1+\cos(\sqrt{2|K|}U)\bigr)
=\Delta_0U\cot\!\left(\frac{\sqrt{2|K|}U}{2}\right)
\ (U\ne0)$, together with $\tan t\geq t$ for $0\leq t<\pi/2$, this gives
\[
 \begin{aligned}
 m_-(r)
 =\min_{|z|=r}\{(-\Delta_0V_\varepsilon)\tan(\sqrt{2|K|}V_\varepsilon/2)\}
 \geq\sqrt{\frac{|K|}{2}}\min_{|z|=r}V_\varepsilon(-\Delta_0V_\varepsilon)
 \geq\frac{\sqrt{2|K|}}{8}F(r)(-\Delta_0F(r)).
 \end{aligned}
\]
Thus, we obtain completeness.

\medskip
\noindent\textbf{\textup{(iv)} $K<0$ and mixed rank case}\\
For a smooth function $f$ that is even in $x$, write $ \mathscr D f:=\Delta_0f+\frac2x\partial_xf$.
Since $\partial_xf/x=\int_0^1\partial_x^2f(tx,y)\,\dd t$, this expression is smooth across $x=0$.
In particular, for radial $F$, we have
\begin{equation}
 \mathscr DF=F''+\frac3rF'
 =-\frac{2(2-r^2)L-3r^2}{(1-r^2)^2L^{5/2}}<0.
 \label{eq:mixed-radial-operator}
\end{equation}

This inequality follows from $B_0>3/2$, and $r=0$ is also interpreted as a smooth limit.
Since $h$ is even, $\mathscr Dh$ is bounded on the closed disk, and we have
\[
 \mathscr D V_\varepsilon
 =\mathscr DF(1+\varepsilon\widetilde R_h),\qquad
 \widetilde R_h
 =h+\frac{2\langle\nabla F,\nabla h\rangle+F\mathscr Dh}{\mathscr DF}.
\]

Using \eqref{eq:mixed-radial-operator}, the same boundary estimates as above show that $\widetilde R_h$ is bounded.
Furthermore, since $\sup F=B_0^{-1/2}<1$, for sufficiently small $|\varepsilon|$, we have $0<V_\varepsilon<1,\  -\mathscr D V_\varepsilon\geq\tfrac12(-\mathscr DF)>0$.
Therefore, $a_\varepsilon=1-V_\varepsilon$ satisfies $0<a_\varepsilon<1$, and we have $Q_\varepsilon=\frac{-\mathscr D V_\varepsilon}{1-V_\varepsilon}      \geq\tfrac12(-\mathscr DF)>0$ and 
\begin{equation}
 \Delta_0U_\varepsilon
 =\delta x\,\mathscr Da_\varepsilon
 =Q_\varepsilon U_\varepsilon.
 \label{eq:mixed-potential-equation}
\end{equation}

We have $|U_\varepsilon|<\delta<\pi/\sqrt{2|K|}$, and $u/\{\sqrt{2|K|}\sin(\sqrt{2|K|}u)\}$ is a smooth positive even function with value $1/(2|K|)$ at $u=0$.
Thus, \eqref{eq:mixed-complete-H} is positive and smooth, and \eqref{eq:mixed-potential-equation} is equivalent to $\Delta_0U_\varepsilon=\sqrt{2|K|}H_{U_\varepsilon}\sin(\sqrt{2|K|}U_\varepsilon)$.
This gives the admissibility condition, and \(U_\varepsilon^{-1}(0)=\{x=0\}\cap\D\).
By Corollary~\ref{cor:mixed rank-locus}, the model has rank one on \(U_\varepsilon^{-1}(0)\times\R\) and rank three on its complement.

For completeness, extend $u\cot(\sqrt{2|K|}u/2)$ smoothly
at $u=0$ with value $2/\sqrt{2|K|}$.
Let $d_\delta>0$ be the positive minimum of this function on $[-\delta,\delta]$.
Therefore at every point, we have $\sqrt{2|K|}H_{U_\varepsilon}(1+\cos(\sqrt{2|K|}U_\varepsilon)) =Q_\varepsilon\,U_\varepsilon\cot(\sqrt{2|K|}U_\varepsilon/2) \geq\tfrac{d_\delta}{2}(-\mathscr DF)$.
Therefore, $m_-(r)\geq\tfrac{d_\delta}{2}(-\mathscr DF(r))$.
On the other hand, \eqref{eq:mixed-radial-operator} gives
\[
 -\mathscr DF(r)\sim\frac{2}{(1-r^2)^2L^{3/2}},\qquad
 \int^1\sqrt{-\mathscr DF(r)}\,\dd r
 \asymp\int^\infty L^{-3/4}\,\dd L=\infty.
\]
From Proposition~\ref{prop:radial-completeness-criterion}, geodesic completeness follows.
\end{proof}

\section{Some facts on disk automorphisms}
Here, we state the known results on the disk automorphism group used in Appendix \ref{app:Palais-nonregular quotient constructions}.
Conjugation by the Cayley map \(\mathcal C(z)=\ii(1+z)/(1-z)\)
identifies \(\Aut(\D)\) with \(\Aut(\mathbb H)=\operatorname{PSL}(2,\R)\).
Here, \(\mathbb H=\{z\in\C:\operatorname{Im}z>0\}\), and
the compact-open topology corresponds to the usual Lie group topology under this identification
\cite[\S\S1.1--1.2]{FunarFuchsian}.

A subgroup $P$ of a topological group $G$ is called \textbf{discrete} if the identity element $e$ is an isolated point of $P$.
Equivalently, there exists an open neighborhood $V$ of $e$ in $G$ such that $V\cap P=\{e\}$.
Since $\Aut(\D)$ is metrizable, if a subgroup $P\subset\Aut(\D)$ is nondiscrete, then there exists a sequence of pairwise distinct elements $\psi_j\in P\setminus\{\id\}$ such that $\psi_j\to\id$.

\begin{lemma}
\label{lem:disk-automorphism-normal-form}
Let $\D=\{z\in\C:|z|<1\}$ be the unit disk and $\Aut(\D)$ its group of holomorphic automorphisms.
For every \(a\in\D\),
\begin{equation}
  \phi_a(z):=\frac{z-a}{1-\overline a z}
  \label{eq:phi-a}
\end{equation}
is a disk automorphism and satisfies \(\phi_a(a)=0\).
Also, every \(\psi\in\Aut(\D)\) can be written uniquely as
\begin{equation}
  \psi(z)=\e^{\ii\theta}\frac{z-a}{1-\overline a z},
  \qquad a\in\D,\quad \e^{\ii\theta}\in S^1.
  \label{eq:AutD-explicit}
\end{equation}
This is the classification of disk automorphisms
\cite[Propositions~3.5.2--3.5.3]{LeblCA}.
As a consequence, the stabilizer of \(p\in\D\) becomes the
rotation group \(\{z\mapsto\e^{\ii\theta}z:\e^{\ii\theta}\in S^1\}\) under conjugation by \(\phi_p\).
In particular, a disk automorphism fixing two distinct points in the disk is the identity map.
\end{lemma}

\begin{lemma}
\label{lem:AutD-action-proper}
Let $\D=\{z\in\C:|z|<1\}$ be the unit disk and $\Aut(\D)$ its group of holomorphic automorphisms.
The orientation-preserving isometry group of the hyperbolic metric
\[
  g_{\D}=\frac{4|\dd z|^2}{(1-|z|^2)^2}
\]
is \(\Aut(\D)\) \cite[\S1.1]{FunarFuchsian}, so \(\Aut(\D)\curvearrowright\D\) is proper.
Furthermore, a subgroup \(P\subset\Aut(\D)\) is discrete if and only if
it acts properly discontinuously on \(\D\)
\cite[Proposition~1.2.1]{FunarFuchsian}.
\end{lemma}

\begin{corollary}
\label{cor:point-stabilizer-cyclic}
Let $\D=\{z\in\C:|z|<1\}$ be the unit disk and $\Aut(\D)$ its group of holomorphic automorphisms.
If \(P\subset\Aut(\D)\) is discrete, then for every \(p\in\D\), $P_p:=\{\psi\in P:\psi(p)=p\}$ is a finite cyclic group.
This follows from the description of the stabilizer in Lemma~\ref{lem:disk-automorphism-normal-form}
and the fact that every discrete subgroup of \(S^1\) is finite cyclic
\cite[Lemma~2.2.2(ii)]{KatokFuchsian}.
\end{corollary}

\begin{lemma}
\label{lem:disk-automorphism-dynamics}
Let $\D=\{z\in\C:|z|<1\}$ be the unit disk, and let $\Aut(\D)$ be its group of holomorphic automorphisms.
Assume that \(\psi\in\Aut(\D)\setminus\{\id\}\).
\begin{enumerate}[label=\textup{(\roman*)}]
\item If \(\psi\) has a fixed point in $\D$, then after conjugation moving this point to \(0\),
      we can write \(\psi(z)=\e^{2\pi\ii\alpha}z\).
      If \(\alpha\in\Q\), then \(\langle\psi\rangle\) is finite.
      If \(\alpha\notin\Q\), the closure of \(\langle\psi\rangle\) is the full rotation group
      \cite[\S1.2, p.~4]{BrinStuck}.
\item If \(\psi\) has no fixed point in $\D$, then after passing to \(\mathbb H\) and conjugating,
      it has one of the following standard forms \cite[\S1.2.1]{FunarFuchsian}:
      \begin{align}
        D_c(z)&:=c z,
        &&c>0,\ c\neq1
        &&\text{(hyperbolic)},
        \label{eq:hyperbolic-normal-form}\\
        N_b(z)&:=z+b,
        &&b\in\R\setminus\{0\}
        &&\text{(parabolic)}.
        \label{eq:parabolic-normal-form}
      \end{align}
      In the hyperbolic case, there are two distinct fixed points on the boundary $\partial\D$, and in the parabolic case, there is one fixed point.
      In both cases, \(\langle\psi\rangle\) is discrete, and nonzero powers have no fixed point in $\D$.
\item Let \(G=\operatorname{PSL}(2,\R)\), and write $C_G(\psi):=\{\chi\in G:\chi\psi=\psi\chi\}$ for the centralizer of \(\psi\in G\setminus\{\id\}\).
If \(\psi\) is hyperbolic, parabolic, or elliptic,
      then \(C_G(\psi)\setminus\{\id\}\) consists of all elements of the same type as \(\psi\)
      whose fixed-point set in \(\mathbb H\cup\R\cup\{\infty\}\) agrees with that of \(\psi\)
      \cite[Theorem~2.3.3, p.~35]{KatokFuchsian}.
      In particular, in the standard coordinates of \textup{(ii)},
      \[
        C_G(D_c)=\{D_a:a>0\},\qquad
        C_G(N_b)=\{N_t:t\in\R\},
      \]
      and both are one-parameter groups isomorphic to \((\R,+)\).
\end{enumerate}
\end{lemma}

\section{On the Palais-nonregular case and quotient constructions}
\label{app:Palais-nonregular quotient constructions}
In this section, we classify the possible Palais-nonregular
characteristic flows on the quotient manifold.
An action of a Lie group $G$ on $M$ is called \textcolor{blue}{\textbf{locally free}} if every stabilizer $G_p$ is discrete.

\begin{screen}
\begin{theorem}[Palais-nonregular orbit trichotomy]
\label{thm:orbit-trichotomy}
Let $(M^3,\varphi,\xi,\eta,g)$ be a connected noncompact proper trans-Sasakian $\eta$-Einstein manifold with complete Palais-nonregular $\xi$ and $K=\Ric(\xi,\xi)$.
Then we have $K>0$ and the structure has rank three.
Let $\Phi_t$ be the flow of $\xi$, and put $L:=\{\Phi_t:t\in\R\}$.
The closure of $L$ is taken in $\Isom(M,\bar g)$, where $\bar g$ is the invariant
metric of Proposition~\ref{prop:averaged-invariant-metric}.
Then one of the following occurs.

\begin{enumerate}
\item[\textup{(SF)}] The image $L$ is a circle and acts locally freely but not freely.
Every characteristic orbit is periodic, and $M$ is a Seifert fiber space over a noncompact oriented $2$-orbifold.

\item[\textup{(IR$_0$)}] The closure $H=\overline L$ is a two-torus, $M/H\simeq\R$, and there is an $H$-equivariant diffeomorphism
\begin{equation}
  M\simeq H\times\R\simeq\T^2\times\R,
  \label{eq:IR0-equivariant-product}
\end{equation}
where $H$ acts by left translation on the first factor.
There are no periodic characteristic orbits. Every orbit is dense in a principal two-torus.

\item[\textup{(IR$_1$)}] The closure $H=\overline L$ is a two-torus and $M/H\simeq[0,\infty)$.
The unique singular orbit has isotropy $K_0\cong S^1$, and there is an $H$-equivariant diffeomorphism
\begin{equation}
  M\simeq H\times_{K_0}\R^2\simeq S^1\times\R^2,
  \label{eq:IR1-associated-bundle}
\end{equation}
where $K_0$ acts on $\R^2$ by the standard rotations.
There is exactly one periodic characteristic orbit and every other orbit is dense in a principal two-torus.
\end{enumerate}
\end{theorem}
\end{screen}
\begin{proof}
Suppose first that $K\leq0$.
By Lemma~\ref{lem:canonical-nonpositive}, the characteristic flow action is free and proper.
A free proper $\R$-action defines a principal bundle, whose local trivializations give distinguished flow boxes.
Hence this is Palais-regular.
Then the Palais-nonregular case cannot occur when $K\le0$.
Thus $K>0$.
From Proposition \ref{prop:complete-riccati-branches}, the rank is three.

By Proposition~\ref{prop:averaged-invariant-metric}, the characteristic flow acts by $\bar g$-isometries.
Thus $L\subset\Isom(M,\bar g)$, a Lie group acting properly on $M$
\cite{MyersSteenrod,Manoussos2010}.
Since $\xi$ is nowhere zero, its flow has discrete stabilizers.

We first consider the case in which $L$ is closed and show that case \textup{(SF)} occurs.
\begin{enumerate}[label={\large\textbullet}]
  \item A closed subgroup of a properly acting Lie group also acts on $M$ properly. Thus the $L$-action is proper, and $L\cong\R$ or $L\cong S^1$.
  
  \item If $L\cong\R$, then $\xi$ is regular, thus this case is impossible.\\
  \hspace*{1em} Properness makes every stabilizer compact, and $\R$ has no nontrivial compact subgroup.
  Hence the action is free.
  Hence $M\to M/L$ is a principal $\R$-bundle \cite[Theorem 1.6.3]{Tammo}, so its orbit foliation is simple and therefore Palais-regular.

  \item If $L\cong S^1$ and the action is free, then $\xi$ is regular, as in the case $L\cong\R$ , thus this case is impossible.
  
  \item Therefore Palais-nonregularity forces $L\cong S^1$ to act locally freely but not freely.
  
  \item The quotient $M/L$ is an oriented noncompact $2$-orbifold, and $M\to M/L$ is a Seifert fibration.\\
  \hspace*{1em} Every orbit is periodic.
  If $L_p$ is a finite stabilizer, the slice theorem gives the local model $L\times_{L_p}D^2\longrightarrow D^2/L_p$.
  The quotient map for the compact group $L$ is proper.
  Hence $M/L$ cannot be compact because $M$ is noncompact.
\end{enumerate}

We now consider the case in which $L$ is not closed.
We first show that $H:=\overline{L}\cong\T^b$ for some $b\geq2$.
\begin{enumerate}[label={\large\textbullet}]
  \item $H\cong\R^a\times\T^b$ and acts effectively and properly on $M$.\\
  \hspace*{1em} Since $H$ is a closed subgroup of $\Isom(M,\bar g)$, it is a Lie group.
  Since $L$ is connected and abelian, so is $H$.
  The classification of connected abelian Lie groups gives $H\cong\R^a\times\T^b$.
  Its action is effective because $H\subset\Isom(M,\bar g)$, and the restriction to the closed subgroup $H$ of the proper $\Isom(M,\bar g)$-action remains proper.
  \item We have $a=0$ and $b\geq2$.\\
  \hspace*{1em} For some $x\in\R^a$ and $y\in\R^b$, $\Phi_t=(tx,[ty])$.
  If $x\neq0$, the map $t\mapsto(tx,[ty])$ is proper and therefore has closed image, contrary to the assumption that $L$ is not closed.
  Thus $x=0$, and the density of $L$ in $H$ forces $a=0$.
  Every one-parameter subgroup of a one-torus is closed, so the nonclosedness of $L$ implies $b\geq2$.
\end{enumerate}

We recall the Principal Orbit Theorem \cite[(1.8.2)]{Tammo}.
Let a compact Lie group $G$ act smoothly on a manifold $N$, and suppose that $N/G$ is connected.
For a closed subgroup $K\subset G$, set
\[
  N_{(K)}
  :=
  \{p\in N:G_p\text{ is conjugate in }G\text{ to }K\},
  \qquad
  N^K
  :=
  \{p\in N:K\subset G_p\}.
\]
Then the following holds.
\begin{enumerate}[label={\large\textbullet}]
  \item There is a unique isotropy type $(J)$ such that $N_{(J)}$ is open and dense in $N$.
  \item $N_{(J)}/G$ is connected.
  \item For every isotropy group $K$, there exists $g\in G$ such that $gJg^{-1}\subset K$.
  \item $N^J$ meets every $G$-orbit.
\end{enumerate}
The type $(J)$ is called the principal orbit type, and $N_{(J)}$ is the associated principal orbit bundle.
For $p\in N_{(J)}$, the orbit $G\cdot p$ is called a principal orbit.

We apply the theorem with $G=H$ and $N=M$.
The hypotheses hold because $H$ is a torus and $M/H$, being the continuous image of the connected space $M$, is connected.

We now show that $H\cong\T^2$ and that the $H$-action has cohomogeneity one.
\begin{enumerate}[label={\large\textbullet}]
  \item Let $J$ be the principal isotropy subgroup. Then $J=\{e\}$.\\
  \hspace*{1em} Since $H$ is abelian, every element $j\in J$ fixes the open dense principal orbit bundle $M_{(J)}$ pointwise.
  The fixed-point set of $j$ is closed, so $j$ fixes all of $M$.
  Effectivity gives $J=\{e\}$.
  \item We have $\dim H=2$.\\
  \hspace*{1em} The principal orbits have dimension $\dim H$, and hence $\dim H\leq3$.
  If $\dim H=3$, every principal orbit is both open and compact, hence also closed, and connectedness gives $M=H\cdot p$, contradicting the noncompactness of $M$.
  Since $b\geq2$, it follows that $\dim H=2$.
\end{enumerate}

We next show that every nonprincipal orbit is of the form $H\cdot p\simeq S^1$.
\begin{enumerate}[label={\large\textbullet}]
  \item If $\dim(H\cdot p)=2$, then $H\cdot p$ is a principal orbit.\\
  \hspace*{1em} For any $p\in M$, the abelianity of $H$ implies that every $k\in H_p$ fixes $H\cdot p$ pointwise.
  Hence $\dd k_p$ is the identity on $T_p(H\cdot p)$.
  Every element of the connected group $H$ preserves the orientation of $M$, so $\dd k_p$ is also the identity on the normal line because the orbit has codimension one.
  Thus $\dd k_p=\id$.
  An isometry of a connected Riemannian manifold is determined by its value and differential at one point, so $k=e$ by effectivity.
  Therefore $H_p=\{e\}$, and hence $(H_p)=(J)$.
  Thus $H\cdot p$ is a principal orbit.
  \item The case $\dim(H\cdot p)=0$ is impossible.\\
  \hspace*{1em} At a fixed point the isotropy representation gives an injective homomorphism $\T^2\to\mathrm{SO}(T_pM)\cong\mathrm{SO}(3)$.
  This is impossible because every abelian subalgebra of $\mathfrak{so}(3)$ has dimension at most one.
  \item Every nonprincipal orbit is diffeomorphic to $S^1$.\\
  \hspace*{1em} From the above discussion, every nonprincipal orbit satisfies $\dim(H\cdot p)=1$.
  At a nonprincipal orbit, both $H_p$ and $H/H_p$ are one-dimensional Lie groups.
  Since the projection $H\to H/H_p$ is continuous and $H$ is compact and connected, $H/H_p$ is also compact and connected.
  Therefore $H\cdot p\simeq H/H_p\simeq S^1$.
\end{enumerate}

Let $\pi:M\to B:=M/H$ be the projection onto the orbit space.
We now show that $B\simeq\R$ or $B\simeq[0,\infty)$.
\begin{enumerate}[label={\large\textbullet}]
  \item We use \cite[Lemma~3.1.3 and its proof]{Muzzy}: If a compact connected Lie group $G$ acts on an $n$-dimensional manifold $N$, and its principal orbits have dimension $d$ with $n-d\leq1$, then $N/G$ is a manifold, possibly with boundary.
  \item We have $B\simeq S^1,\ \R,\ [0,1]$, or $[0,\infty)$.\\
  \hspace*{1em} Here $n=3$ and $d=2$, so $B$ is a connected one-dimensional manifold with boundary, and its boundary points are precisely the images of the singular circle orbits.
  The classification of connected one-dimensional manifolds gives $B\simeq S^1,\ \R,\ [0,1]$ or $[0,\infty)$.
  \item Since $M$ is noncompact, the only possibilities are $B\simeq\R$ or $B\simeq[0,\infty)$.
\end{enumerate}

Suppose first that $B\simeq\R$. We obtain case \textup{(IR$_0$)} as follows.
\begin{enumerate}[label={\large\textbullet}]
  \item We have $M\simeq\T^2\times\R$.\\
  \hspace*{1em} Every orbit is principal, so $\pi:M\to\R$ is a principal $H$-bundle.
  Since $\R$ is contractible, this bundle is trivial, giving the $H$-equivariant diffeomorphism $M\simeq\T^2\times\R$.
  \item Every characteristic orbit is dense in a principal two-torus.\\
  \hspace*{1em} If $\ker\Phi=T\mathbb Z$ for some $T>0$, then
  $L=\Phi(\R)\cong\R/\ker\Phi=\R/T\mathbb Z\cong S^1$, contradicting the assumption that $L$ is not closed.
  Hence $\ker\Phi=\{0\}$.
  Since every orbit is principal, $H_p=\{e\}$ for all $p\in M$, and no characteristic orbit is periodic.
  Since $H\cdot p$ is compact and $\overline L=H$, we have $\overline{L\cdot p}=H\cdot p\simeq H\simeq\mathbb T^2$.
\end{enumerate}

Suppose next that $B\simeq[0,\infty)$. We obtain case \textup{(IR$_1$)} as follows.
\begin{enumerate}[label={\large\textbullet}]
  \item We have $M\simeq H\times_{K_0}\R^2,\quad K_0\cong S^1$.\\
  \hspace*{1em} There is exactly one singular orbit $S=H/K_0$, where $K_0$ is its one-dimensional isotropy subgroup.
  For $p\in S$, let $V:=\nu_pS\cong\R^2$ be the normal tangent space with respect to $S$.
  Since $H$ is abelian, $K_0$ fixes $S$ pointwise, and its derivative acts faithfully and orientation-preservingly on $V$.
  Thus $K_0\cong S^1$ and its action on $V$ is the standard rotation action.
  The structure theorem \cite[Section~6, case~2]{SchwachhoeferDiskBundles} supplies an $H$-equivariant diffeomorphism $F:H\times_{K_0}V\longrightarrow M$, where
  $H\times_{K_0}V:=(H\times V)/{\sim}$ and the equivalence relation is given by $(h,v)\sim(hk^{-1},kv)$ for $k\in K_0$.
  The identification $V\cong\R^2$ proves the required diffeomorphism.
  \item We have $M\simeq S^1\times\R^2$.\\
  \hspace*{1em} A connected circle subgroup $K_0\subset\T^2$ has a complementary circle $K_1$ with $H=K_0\times K_1$.
  Hence $H\times_{K_0}\R^2 \cong K_1\times\R^2 \cong  S^1\times\R^2$.
  \item On each principal orbit, the characteristic orbit is dense in the principal two-torus, as in case \textup{(IR$_0$)}.
  On the singular orbit, the characteristic flow induced by $L$ is periodic and fills the singular circle.
\end{enumerate}
\end{proof}

We now describe the Deck groups of the three types
\textup{(SF)}, \textup{(IR$_0$)}, and \textup{(IR$_1$)} in
Theorem~\ref{thm:orbit-trichotomy}.
Define the surjective homomorphism
\begin{align}
\pi:\StrAut(\mathcal M_K(v,s))\to G_{\mathfrak u}\label{eq:pi def}\\
\pi(\psi,L)=\psi.
\end{align}
Also, define $\tau\in \StrAut(\mathcal M_K(v,s))$ by
\begin{align}
\tau(z,w):=\left(z,w+\frac{2\pi}{\sqrt{2K}}\right).\label{eq:tau def}
\end{align}
We have $\ker\pi=\langle\tau\rangle\cong\Z$.
For any \(\psi\in G_{\mathfrak u}\), we refer to choosing a suitable $L$ and constructing an element $(\psi,L)\in \StrAut(\mathcal M_K(v,s))$ as lifting $\psi$.
All automorphisms commute with every vertical translation $(z,w)\mapsto (z,w+t)$.
When \(s=1\), we have \(\mathfrak u=v\), and we write \(G_v:=\{\psi\in\Aut(\D):v\circ\psi=v\}\) for the stabilizer in this case.

\begin{lemma}
\label{lem:flow-type-criterion}
Let $\Gamma$ be a subgroup of \(\StrAut(\mathcal M_K(v,s))\) acting freely and properly discontinuously.
Set $M_\Gamma:=\mathcal M_K(v,s)/\Gamma$, and let \(\Phi_t\) be the quotient characteristic flow on \(M_\Gamma\) induced by the characteristic flow \(\widetilde\Phi_t(z,w)=(z,w+t)\) on \(\mathcal M_K(v,s)\).
Also, let \(\pi\) and \(\tau\) be the maps given by \eqref{eq:pi def} and \eqref{eq:tau def}.
Then the following hold.
\begin{enumerate}[label=\textup{(\roman*)}]
\item $\Phi_t$ is (SF) if and only if
\(\Gamma\cap\langle\tau\rangle\neq\{1\}\) and $\Phi_t$ is Palais-nonregular.
\item If \(\Gamma\cap\langle\tau\rangle\neq\{1\}\), then \(\pi(\Gamma)\) is discrete.
\item \(\Gamma\cap\langle\tau\rangle=\{1\}\) and $\Phi_t$ is Palais-nonregular if and only if \(\pi(\Gamma)\) is nondiscrete.
\end{enumerate}
\end{lemma}

\begin{proof}
Set $T:=2\pi/\sqrt{2K}$.
We prove (i).
\begin{enumerate}
  \item $\Phi_{t_0}=\id\Leftrightarrow\widetilde\Phi_{t_0}\in\Gamma\cap\langle\tau\rangle$\\
  \hspace*{2em} Proof :  Write \(p_\Gamma:\mathcal M_K(v,s)\to M_\Gamma\) for the quotient projection.
  If \(\Phi_{t_0}=\id\), then \(p_\Gamma\circ\widetilde\Phi_{t_0}=p_\Gamma\), so \(\widetilde\Phi_{t_0}\in\Gamma\).
  Also, since $\widetilde\Phi_t(z,w)=(z,w+t)$, we have \(\pi(\widetilde\Phi_{t_0})=\id\), so $\widetilde\Phi_{t_0}\in\ker\pi=\langle\tau\rangle$.
  Thus, \(\widetilde\Phi_{t_0}\in\Gamma\cap\langle\tau\rangle\).
  Conversely, if \(\widetilde\Phi_{t_0}\in\Gamma\cap\langle\tau\rangle\), then \(\Phi_{t_0}=\id\) on the quotient.

  \item Thus, $\Gamma\cap\langle\tau\rangle=\{1\}$ is equivalent to $\ker(t\mapsto\Phi_t)=\{0\}$, and $\Gamma\cap\langle\tau\rangle=\langle\tau^r\rangle,\ (r\geq1)$ is equivalent to $\ker(t\mapsto\Phi_t)=rT\Z$.

  \item If $\Phi_t$ is (SF), then \(\Gamma\cap\langle\tau\rangle\neq\{1\}\) and $\Phi_t$ is Palais-nonregular.\\
  \hspace*{2em} Proof :  In (SF), the effective flow group is a circle, so \(\ker(t\mapsto\Phi_t)\neq\{0\}\), and therefore \(\Gamma\cap\langle\tau\rangle\neq\{1\}\).
  Also, (SF) is Palais-nonregular by definition.

  \item If \(\Gamma\cap\langle\tau\rangle\neq\{1\}\) and $\Phi_t$ is Palais-nonregular, then $\Phi_t$ is (SF).\\
  \hspace*{2em} Proof :  Since $\Gamma\cap\langle\tau\rangle=\langle\tau^r\rangle,\ (r\geq1)$, the effective flow group is
  \(\R/rT\Z\). Since \(\xi\) is nowhere zero, this circle action is locally free.
  If it were free, it would be a free proper circle action and would give a principal bundle, so it would be Palais-regular, contrary to the assumption.
  Thus, this action is nonfree, and $\Phi_t$ is (SF).
\end{enumerate}

We prove \textup{(ii)}.
Assume that \(\Gamma\cap\langle\tau\rangle=\langle\tau^r\rangle\), \(r\geq1\).
Suppose that \(\pi(\Gamma)\) is nondiscrete and derive a contradiction.
Take a sequence of distinct elements \(\psi_\nu\in \pi(\Gamma)\) with \(\psi_\nu\to\id\), and choose a lift \(\gamma_\nu\) of each \(\psi_\nu\) in \(\Gamma\).
Take any \(z_*\in\D\), and write \(\gamma_\nu(z_*,0)=(\psi_\nu(z_*),a_\nu)\).
For each \(\nu\), choose an integer \(q_\nu\) such that \(0\leq a_\nu-rq_\nu T<rT\), and set \(\gamma'_\nu:=\tau^{-rq_\nu}\gamma_\nu\in\Gamma\).
Then \(\gamma'_\nu(z_*,0)=\bigl(\psi_\nu(z_*),a_\nu-rq_\nu T\bigr)\).
Since \(S:=\{z_*\}\cup\{\psi_\nu(z_*):\nu\geq1\}\) is a compact set in \(\D\), \(Q:=S\times[0,rT]\) is a compact set in \(\mathcal M_K(v,s)\).
Both \((z_*,0)\) and all \(\gamma'_\nu(z_*,0)\) belong to \(Q\), so \(\gamma'_\nu Q\cap Q\neq\varnothing\).
On the other hand, \(\gamma'_\nu\) are distinct because \(\pi(\gamma'_\nu)=\psi_\nu\).
Thus, infinitely many elements satisfy \(\gamma'_\nu Q\cap Q\neq\varnothing\), contrary to proper discontinuity.
Therefore, \(\pi(\Gamma)\) is discrete.

We prove one direction of \textup{(iii)}, namely, that if \(\Gamma\cap\langle\tau\rangle=\{1\}\) and $\Phi_t$ is Palais-nonregular, then \(\pi(\Gamma)\) is nondiscrete.
By the contrapositive, it suffices to assume that \(\Gamma\cap\langle\tau\rangle=\{1\}\) and \(\pi(\Gamma)\) is discrete and to show that $\Phi_t$ is Palais-regular.
\begin{enumerate}
  \item \(\pi(\Gamma)\) acts freely on \(\D\).\\
  \hspace*{2em} Proof :  Suppose that \(\psi\in\pi(\Gamma)\) satisfies \(\psi(z)=z\) for some \(z\in\D\).
  By Corollary~\ref{cor:point-stabilizer-cyclic}, \(\psi^k=\id\) for some \(k\geq1\).
  For a lift \(\gamma\in\Gamma\) of \(\psi\)\footnote{In this case, the lift can be shown to be unique, but uniqueness is not needed in the argument here.}, we have $    \gamma^k\in\Gamma\cap\ker\pi    =\Gamma\cap\langle\tau\rangle =\{1\}$.
  \(\gamma\) is a translation on \(\{z\}\times\R\), and since \(\gamma^k=1\), its translation displacement is zero.
  Thus, \(\gamma\) fixes this fiber.
  Since the action of \(\Gamma\) is free, \(\gamma=1\), and therefore \(\psi=\id\).

  \item $\Gamma\simeq\pi(\Gamma)$.\\
  \hspace*{2em} Proof :  \(\pi|_\Gamma:\Gamma\to\pi(\Gamma)\) is surjective by definition. Its kernel is \(\Gamma\cap\ker\pi=\Gamma\cap\langle\tau\rangle=\{1\}\), so it is also injective.

  \item Since \(\pi(\Gamma)\) is discrete and acts freely, \(q:\D\to\D/\pi(\Gamma)\) is a covering map.
  Define $\pi_\Gamma:M_\Gamma\longrightarrow\D/\pi(\Gamma)$ by $\pi_\Gamma([z,w]):=[z]$.
  Let $\bar V\subset\D/\pi(\Gamma)$ be an evenly covered neighborhood. Write $q^{-1}(\bar V)=\coprod_i V_i,\ V_i\simeq \bar{V}$, and take one of the $V_i$ to be $V$.
  Then $\pi^{-1}_\Gamma(\bar{V})\simeq V\times\R$.\\
  \hspace*{2em} Proof :  Define $f:\pi^{-1}_\Gamma(\bar{V})\to V\times\R$ as follows.
  For $x\in\pi_\Gamma^{-1}(\bar V)$, choose \((z,w)\in\D\times\R\) such that $x=p_\Gamma(z,w)$.
  Then $q(z)=\pi_\Gamma(x)\in\bar V$, so \(z\in q^{-1}(\bar V)\).
  Take \(\psi\in\pi(\Gamma)\) moving \(z\) to \(V\) and its unique lift \(\gamma\in\Gamma\). Then $\gamma(z,w)=\bigl(\psi(z),w+f_\gamma(z)\bigr)\in V\times\R$.
  The definition $f(x):=(\psi(z),w+f_\gamma(z))$ is then well-defined.
  Let \(p_\Gamma:\D\times\R\to M_\Gamma=(\D\times\R)/\Gamma\).
  If \(f(x)=f(y)\), then $x=p_\Gamma(f(x))=p_\Gamma(f(y))=y$, so \(f\) is injective.
  For any \((u,t)\in V\times\R\), set $x:=p_\Gamma(u,t)$. Then \(f(x)=(u,t)\), so it is surjective.
  
  \item $\Phi_t$ is Palais-regular.\\
  \hspace*{2em} Proof :  The above $\pi^{-1}_\Gamma(\bar{V})\simeq V\times\R$ gives a smooth local trivialization, under which $\Phi_t$ is a translation in the \(\R\)-factor.
  Restricting this to a small interval in the \(\R\)-direction gives a flow box meeting each characteristic orbit in at most one plaque.
  Thus, $\Phi_t$ is Palais-regular.
\end{enumerate}

We prove the converse direction of \textup{(iii)}, namely, that if \(\pi(\Gamma)\) is nondiscrete, then \(\Gamma\cap\langle\tau\rangle=\{1\}\) and $\Phi_t$ is Palais-nonregular.
Assume that \(\pi(\Gamma)\) is nondiscrete.
\begin{enumerate}
  \item By the contrapositive of (ii), \(\Gamma\cap\langle\tau\rangle=\{1\}\).

  \item $\Phi_t$ is Palais-nonregular.\\
  \hspace*{2em} Proof :  Take a sequence of nonidentity elements \(\psi_j\in\pi(\Gamma)\) with \(\psi_j\to\id\).
  By Lemma \ref{lem:disk-automorphism-normal-form}, each \(\psi_j\) has at most one fixed point in the disk, so the set of fixed points of the $\psi_j$ is countable.
  Since $\D$ is uncountable, we can choose \(z_*\in\D\) such that \(\psi_j(z_*)\neq z_*\) for all \(j\).
  Let \(p_\Gamma:\D\times\R\to M_\Gamma=(\D\times\R)/\Gamma\), and set $[z,w]:=p_\Gamma(z,w),\ x_*:=[z_*,0]\in M_\Gamma$.
  Let $B$ be any flow box containing \(x_*\).
  Taking a sufficiently small open neighborhood \(V\subset\D\) of \(z_*\) and an interval $I\subset\R$ gives $V\times I\subset B$.
  We can arrange that two points in $V\times I\subset B$ belong to the same plaque if and only if their $V$-coordinates agree.
  Since \(\psi_j(z_*)\to z_*\), we have $z_*,\psi_j(z_*)\in V$ for sufficiently large \(j\).
  Also, since \(\psi_j(z_*)\neq z_*\), \([z_*,0]\) and \([\psi_j(z_*),0]\) belong to distinct plaques in \(B\).
  On the other hand, write a lift of \(\psi_j\) as $\gamma_j(z,w)=\bigl(\psi_j(z),w+f_j(z)\bigr)$. Since $\gamma_j\bigl(z_*,-f_j(z_*)\bigr)=\bigl(\psi_j(z_*),0\bigr)$, we have $[\psi_j(z_*),0]=[z_*,-f_j(z_*)]=\Phi_{-f_j(z_*)}([z_*,0])$.
  Thus, these two points belong to the same global characteristic orbit. Hence, any flow box \(B\) containing \(x_*\) contains two distinct plaques belonging to one characteristic orbit.
  Therefore, \(\Phi_t\) is not Palais-regular.
\end{enumerate}
\end{proof}

\subsection{Construction of type \texorpdfstring{$(\mathrm{SF})$}{(SF)}}

\begin{lemma}
\label{lem:sfcons-presentation}
Let \((v,s)\) be an admissible pair, let \(P\subset G_{\mathfrak u}\) be a discrete subgroup, and let \(B:=\D/P\) be the quotient orbifold.
Then \(B\) is noncompact.
Furthermore, there exist finite or countable index sets \(I,J\), group elements \(a_i,c_j\in P\),
points \(p_j\in\D\), and integers \(k_j>1\) such that
\begin{equation}
 \begin{gathered}
 P=\left\langle a_i\ (i\in I),\ c_j\ (j\in J)
   \ \middle|\ c_j^{k_j}=1\ (j\in J)\right\rangle,\\
 P_{p_j}=\langle c_j\rangle,\qquad
 k_j=|P_{p_j}|>1,\qquad
 \dd c_j(p_j)=\e^{2\pi\ii/k_j}.
 \end{gathered}
 \label{eq:sfcons-p-presentation}
\end{equation}
holds, and the images of \(p_j\) give all cone points of \(B\) without repetition.
\end{lemma}

\begin{proof}
\(f:=\e^{2v}|s|^2\) is a smooth \(P\)-invariant function, and
\(0\leq f<1\) and \(f\not\equiv0\).
Writing \(v_z=(v_x-\ii v_y)/2\), we have
\begin{equation}
 \Delta_0 f
 =2f\Delta_0v+4\e^{2v}|s'+2v_zs|^2\geq0,
 \label{eq:sfcons-noncompact-base}
\end{equation}
and it is strictly positive on $\{s\ne0\}=\{f\ne0\}$.
If \(B\) were compact, \(f\) would attain a positive maximum on \(B\).
At a maximum point \(p\in B\), we have \(f(p)>0\) and \(\Delta_0f(p)\leq0\).
Thus, $0<\Delta_0f\le0$, which is a contradiction.
Therefore, \(B\) is noncompact.

A discrete group with a noncompact disk quotient is a free product of infinite cyclic groups and finite cyclic groups corresponding to the stabilizers of cone points \cite[p.~204, equation~(1), and p.~205, Remark~2]{Tukia1979SF}.
Taking \(a_i\) as the generators of the infinite cyclic factors, \(c_j\) as the generators of the finite cyclic factors, and \(p_j\) as the fixed point of \(c_j\) gives the desired presentation.
If necessary, replace the generator of each finite cyclic factor by a power relatively prime to its order so that \(\dd c_j(p_j)=\e^{2\pi\ii/k_j}\).
\end{proof}

\begin{lemma}
\label{lem:P free group rep}
Let $P=\langle a_i, c_j\  |\ c_j^{k_j}=1\rangle$ be the discrete subgroup in Lemma~\ref{lem:sfcons-presentation}.
Let $F$ be the free group generated by $a_i,c_j$.
Let $N\subset F$ be the smallest normal subgroup containing \(c_j^{k_j}\). Then
\begin{align*}
N=
\left\{
\prod_{\nu=1}^{N_0}
U_\nu c_{j_\nu}^{\varepsilon_\nu k_{j_\nu}}U_\nu^{-1}
\;\middle|\;
N_0\in\Z_{\geq0},\ 
j_\nu\in J,\ 
\varepsilon_\nu\in\{1,-1\},\
U_\nu\in F
\right\}
\end{align*}
and $P\simeq F/N$.
Here, $N_0=0$ denotes the empty product.
\end{lemma}

In what follows, we choose \(p_j\) and \(c_j\) together so that they are consistent with \eqref{eq:sfcons-p-presentation}.
Since \(k_j\) and \(\operatorname{ord}_{p_j}(s)\) depend only on the orbit,
the realization conditions below do not depend on this choice.

\begin{lemma}
\label{lem:sfcons-normalized-lift-powers}
Let \(P\subset G_{\mathfrak u}\) be a discrete subgroup, let
\(I,J,a_i,c_j,p_j,k_j\) be the data in
Lemma~\ref{lem:sfcons-presentation}, and set \(d_j:=\operatorname{ord}_{p_j}(s)\).
Then there exists \(L_j\in\mathcal O(\D)\) satisfying $e^{L_j}=\frac{s}{s\circ c_j},\ L_j(p_j)=-\frac{2\pi\ii d_j}{k_j}$.
Furthermore, for \(L_j\) satisfying this condition,
setting \(\widetilde c_j:=F_{c_j,L_j}\) gives
\begin{equation}
 \widetilde c_j^{\,k_j}=\tau^{-d_j}
 \label{eq:sfcons-lift-powers}
\end{equation}
Here, \(\tau\) is the map defined in \eqref{eq:tau def}.
\end{lemma}

\begin{proof}
Set \(T:=2\pi/\sqrt{2K}\).
Since \(c_j\in G_{\mathfrak u}\), Lemma~\ref{lem:singular-potential-gauge} gives
\(s/(s\circ c_j)\in\mathcal O^*(\D)\). Since \(\D\) is simply connected, its holomorphic logarithm exists.
Writing \(s(z)=(z-p_j)^{d_j}a_j(z)\), \(a_j(p_j)\neq0\), we have
\[
 \left.\frac{s}{s\circ c_j}\right|_{p_j}
 =\lim_{z\to p_j}\left(\frac{z-p_j}{c_j(z)-p_j}\right)^{d_j}
   \frac{a_j(z)}{a_j(c_j(z))}
 =\dd c_j(p_j)^{-d_j}
 =\e^{-2\pi\ii d_j/k_j}.
\]
Therefore, there exists $L_j$ such that \(L_j(p_j)=-2\pi\ii d_j/k_j\).

Set \(\widetilde c_j:=F_{c_j,L_j}\).
Let $\pi$ be the map given in \eqref{eq:pi def}.
Since \(\pi(\widetilde c_j^{\,k_j})=c_j^{k_j}=\id\),
there exists \(m\in\Z\) such that \(\widetilde c_j^{\,k_j}=\tau^m\).
The action of \(\widetilde c_j\) is
\[
\begin{aligned}
\widetilde c_j(p_j,w)
=F_{c_j,L_j}(p_j,w)
=\left(p_j,w+\frac{1}{\sqrt{2K}}\Im\left(-\frac{2\pi\ii d_j}{k_j}\right)\right)
=\left(p_j,w-\frac{d_jT}{k_j}\right).
\end{aligned}
\]
Thus, on \(\{p_j\}\times\R\), the action of \(\widetilde c_j^{k_j}\) is \(w\mapsto w-d_jT\), so \(m=-d_j\).
\end{proof}

\begin{screen}
\begin{theorem}
\label{thm:SF-deck-classification}
Let $\mathcal M_K(v,s)$ be a model space, and let \(\Gamma\subset\StrAut(\mathcal M_K(v,s))\) be a subgroup acting freely and properly discontinuously.
Then the following two conditions are equivalent.
\begin{enumerate}
\item The characteristic flow of \(M_\Gamma=\mathcal M_K(v,s)/\Gamma\) is of type \textup{(SF)}.
\item Let $P=\left\langle a_i\ (i\in I),\ c_j\ (j\in J)\ \middle|\ c_j^{k_j}=1\ (j\in J\ne\emptyset)\right\rangle\subset G_{\mathfrak u}$ be a subgroup given in Lemma \ref{lem:sfcons-presentation}, and let \(d_j:=\operatorname{ord}_{p_j}(s)\).
Assume that there exists a positive integer \(r\) such that, for each \(j\in J\), there exists an integer \(b_j\) satisfying \(\gcd(b_j,k_j)=1\) and \(rb_j+d_j\equiv0\pmod{k_j}\).
Let $A_i$ be a lift of $a_i$, and let $C_j:=\tau^{n_j}\tilde{c}_j$ be a lift of $c_j$, where $\tilde{c}_j$ is a lift of $c_j$ defined in Lemma~\ref{lem:sfcons-normalized-lift-powers} and $n_j:=\frac{rb_j+d_j}{k_j}$.
Then $\Gamma$ is given by $\Gamma:=\left\langle\tau^r,\ A_i\ (i\in I),\ C_j\ (j\in J)\right\rangle$.
\end{enumerate}
\end{theorem}
\end{screen}

\begin{proof}
Set \(T:=2\pi/\sqrt{2K}\).
Let $\pi$ and $\tau$ be the maps given in \eqref{eq:pi def} and \eqref{eq:tau def}, respectively.
We show that condition 1 implies condition 2.
\begin{enumerate}
\item By condition 1, we have $ \Gamma\cap\langle\tau\rangle=\langle\tau^r\rangle$ for some \(r\in\Z_{>0}\), and \(\pi(\Gamma)\) is discrete.\\
\hspace*{2em} Proof : By Lemma~\ref{lem:flow-type-criterion}\textup{(i)}, \(\Gamma\cap\langle\tau\rangle\neq\{1\}\), so we can write $ \Gamma\cap\langle\tau\rangle=\langle\tau^r\rangle$ for some \(r\in\Z_{>0}\).
Furthermore, \(\pi(\Gamma)\) is discrete by Lemma~\ref{lem:flow-type-criterion}\textup{(ii)}.

\item \(P:=\pi(\Gamma)\) has a nonidentity torsion element, that is, \(J\neq\varnothing\) in the presentation of Lemma~\ref{lem:sfcons-presentation}.\\
\hspace*{2em} Proof :  Let $\Phi_t$ be the characteristic flow of $\mathcal{M}_K(v,s)/\Gamma$.
Since $\Gamma\cap\langle\tau\rangle=\langle\tau^r\rangle$ gives \(\ker(t\mapsto\Phi_t)=rT\Z\), the effective acting group of the characteristic flow is \(\R/(rT\Z)\).
If this action were free, it would not be of type (SF), so this action is not free.
For some \([z,w]\in \mathcal{M}_K(v,s)/\Gamma\) and some nonzero \([t]\in\R/(rT\Z)\), we have
\([z,w+t]=[z,w]\).
Therefore, there exists \(\gamma\in\Gamma\) such that $\gamma(z,w)=(z,w+t)$, and \(\psi:=\pi(\gamma)\) fixes \(z\).
If \(\psi=\id\), then \(\gamma\in\Gamma\cap\ker\pi=\langle\tau^r\rangle\), so \([t]=0\), which is a contradiction.
Thus, \(\psi\neq\id\).
Since \(\pi(\Gamma)\) is discrete, the point stabilizer \(P_z\) is finite, and \(\psi\) is a nonidentity torsion element.
By Lemma~\ref{lem:sfcons-presentation}, we can choose a set of \(a_i,c_j,p_j,k_j\) appearing in condition 2, and \(J\neq\varnothing\) in this presentation.

\item Let \(\widetilde c_j\) be a lift of $c_j$ defined by Lemma~\ref{lem:sfcons-normalized-lift-powers}.
An arbitrary lift \(C_j\in\Gamma\) of $c_j$ is given by $C_j=\tau^{n_j}\widetilde c_j$, where \(n_j\) is a suitable integer.\\
\hspace*{2em} Proof :  Since \(C_j\widetilde c_j^{-1}\in\ker\pi=\langle\tau\rangle\), we can write \(C_j=\tau^{n_j}\widetilde c_j\) for some \(n_j\in\Z\).

\item We have \(rb_j+d_j\equiv0\pmod{k_j}\).\\
\hspace*{2em} Proof :  Since \(C_j\in\Gamma\) and \(c_j^{k_j}=\id\), we have \(C_j^{k_j}\in\Gamma\cap\ker\pi=\langle\tau^r\rangle\), so we can write \(C_j^{k_j}=(\tau^r)^{b_j}\) for some \(b_j\in\Z\).
Since \(\tau\) is central, $(\tau^r)^{b_j}=C_j^{k_j}=(\tau^{n_j}\widetilde c_j)^{k_j}=\tau^{k_jn_j}\widetilde c_j^{\,k_j}=\tau^{k_jn_j-d_j}$.
Here, the last equality uses \eqref{eq:sfcons-lift-powers}.
Since \(\tau\) has infinite order, we obtain \(k_jn_j-d_j=rb_j\).

\item We have \(\gcd(b_j,k_j)=1\).\\
\hspace*{2em} Proof : Assume that \(g_j:=\gcd(b_j,k_j)>1\).
The translation amount of \(C_j\) on \(\{p_j\}\times\R\) is $\left(n_j-\frac{d_j}{k_j}\right)T=\frac{rb_j}{k_j}T$.
Since $D_j:=C_j^{k_j/g_j}(\tau^r)^{-b_j/g_j}\in\Gamma$ satisfies $D_j(p_j,w)=\left(p_j,w+\frac{k_j}{g_j}\frac{rb_j}{k_j}T-\frac{b_j}{g_j}rT\right)=(p_j,w)$, it fixes \(\{p_j\}\times\R\) pointwise.
On the other hand, \(\pi(D_j)=c_j^{k_j/g_j}\neq\id\), so \(D_j\neq1\), contradicting the freeness of $\Gamma$.

\item For each $i\in I$, choose a lift $A_i\in\Gamma$ of $a_i$. Then \(\tau^r,A_i,C_j\) generate \(\Gamma\).\\
\hspace*{2em} Proof : Set \(\Gamma_0:=\langle\tau^r,A_i\ (i\in I),C_j\ (j\in J)\rangle\).
For every \(\gamma\in\Gamma\), choosing a word \(W(a_i,c_j)\) in \(a_i,c_j\) representing \(\pi(\gamma)\) gives $\gamma W(A_i,C_j)^{-1}\in\Gamma\cap\ker\pi =\langle\tau^r\rangle\subset\Gamma_0$.
Therefore, \(\gamma\in\Gamma_0\), and we obtain \(\Gamma=\Gamma_0\).
\end{enumerate}

We show that condition 2 implies condition 1.
\begin{enumerate}
\item We have $\pi(\Gamma)=P,\  \Gamma\cap\langle\tau\rangle=\langle\tau^r\rangle$.\\
\hspace*{2em} Proof :  We have $\pi(\Gamma)=\langle\pi(\tau^r),\pi(A_i),\pi(C_j)\rangle=\langle a_i,c_j\rangle=P$.
Also, \eqref{eq:sfcons-lift-powers} and \(k_jn_j-d_j=rb_j\) give \(C_j^{k_j}=\tau^{n_jk_j}\tilde{c}_j^{k_j}=\tau^{n_jk_j-d_j}=(\tau^r)^{b_j}\).
Let \(F,N\) be the free group and normal subgroup in Lemma~\ref{lem:P free group rep}.
For a word \(W\in F\), \(W(a_i,c_j)\) denotes the element evaluated in \(P\), and \(W(A_i,C_j)\) denotes the element evaluated in \(\Gamma\) after replacing each letter by the corresponding lift.
Take \(\gamma\in\Gamma\cap\ker\pi\). Since \(\tau^r\) is central, we can write \(\gamma=(\tau^r)^mW(A_i,C_j)\) for some \(m\in\Z\) and \(W\in F\).
Since \(\pi(\gamma)=\id\) gives \(W(a_i,c_j)=1\), we have \(W\in\ker(F\to P)=N\).
Therefore, for some \(N_0\in\Z_{\geq0}\), \(U_\nu\in F\), \(j_\nu\in J\), and \(\varepsilon_\nu\in\{1,-1\}\), we can write in \(F\)
\[
 W=\prod_{\nu=1}^{N_0}U_\nu c_{j_\nu}^{\varepsilon_\nu k_{j_\nu}}U_\nu^{-1}
\]
\(N_0=0\) denotes the empty product.
Evaluating this equality after replacing the letters by lifts and using \(C_j^{k_j}=(\tau^r)^{b_j}\) and the centrality of \(\tau^r\) gives
\begin{align*}
 W(A_i,C_j)
 =\prod_{\nu=1}^{N_0}U_\nu(A_i,C_j)C_{j_\nu}^{\varepsilon_\nu k_{j_\nu}}U_\nu(A_i,C_j)^{-1}
 =\prod_{\nu=1}^{N_0}(\tau^r)^{\varepsilon_\nu b_{j_\nu}}\in\langle\tau^r\rangle.
\end{align*}
Therefore, \(\gamma\in\langle\tau^r\rangle\), and we obtain \(\Gamma\cap\ker\pi\subset\langle\tau^r\rangle\).
The reverse inclusion is clear from the generating presentation in condition 2.

\item \(\Gamma\) acts properly discontinuously on \(\D\times\R\).\\
\hspace*{2em} Proof :  Take an arbitrary nonempty compact set \(C\subset\D\) and an arbitrary nonempty compact interval \(\mathcal I\subset\R\), and set \(Q:=C\times\mathcal I\).
Since \(Q=C\times\mathcal I\) and \(P\) is discrete, Lemma~\ref{lem:AutD-action-proper} shows that there are only finitely many projections \(\pi(\gamma)\in P\) of \(\gamma\in\Gamma\) satisfying \(\gamma Q\cap Q\neq\varnothing\).
Fix any one of them, \(\psi\), choose a lift \(\gamma_\psi\in\Gamma\) of \(\psi\), and write \(\gamma_\psi(z,w)=(\psi(z),w+f_\psi(z))\), \(f_\psi\in C^\infty(\D,\R)\).
For any lift \(\gamma\in\Gamma\) of \(\psi\) satisfying \(\gamma Q\cap Q\neq\varnothing\), we have \(\gamma\gamma_\psi^{-1}\in\Gamma\cap\ker\pi=\langle\tau^r\rangle\), so we can write \(\gamma=\tau^{r\ell}\gamma_\psi,\ \ell\in\Z\).
For some \((z,w)\in Q\), we have \(\gamma(z,w)\in Q\).
The \(w\)-coordinates of the initial and final points belong to \(\mathcal I\), and \(f_\psi\) is bounded on \(C\), so \(r\ell T\) is also bounded.
Since \(rT>0\), there are only finitely many possible integers \(\ell\) for each \(\psi\).
There are also only finitely many such \(\psi\), so \(\{\gamma\in\Gamma:\gamma Q\cap Q\neq\varnothing\}\) is finite.
Every compact set is contained in \(Q\) of this form, so proper discontinuity follows.

\item The action of \(\Gamma\) on \(\D\times\R\) is free.\\
\hspace*{2em} Proof :  The equality \(\tilde c_j(p_j,w)=(p_j,w-d_jT/k_j)\) obtained in the proof of Lemma~\ref{lem:sfcons-normalized-lift-powers} and \(n_j=(rb_j+d_j)/k_j\) in condition 2 give $C_j(p_j,w)=\left(p_j,w+\left(n_j-\frac{d_j}{k_j}\right)T\right)=\left(p_j,w+\frac{rb_jT}{k_j}\right)$.
Take \(\gamma\in\Gamma\) such that \(\gamma(z,w)=(z,w)\), and set \(\psi:=\pi(\gamma)\). Then \(\psi(z)=z\).
If \(\psi=\id\), then \(\Gamma\cap\ker\pi=\langle\tau^r\rangle\) gives \(\gamma=\tau^{r\ell}\), \(\ell\in\Z\), and \(r\ell T=0\) gives \(\gamma=1\).
If \(\psi\neq\id\) is impossible, then the identity is the only element with a fixed point, and the action is free.
In what follows, assume that \(\psi\neq\id\) and derive a contradiction.
Since \(P_z\) is a nontrivial finite cyclic group (Corollary~\ref{cor:point-stabilizer-cyclic}), the image of \(z\) in \(B\) is a cone point.
By Lemma~\ref{lem:sfcons-presentation}, \(z\) belongs to the same \(P\)-orbit as some \(p_j\).
Therefore, \(z=\psi_0(p_j)\) for some \(j\in J\) and \(\psi_0\in P\).
Since \(P_{p_j}=\langle c_j\rangle\), we can write \(\psi_0^{-1}\psi\psi_0=c_j^a\), \(1\leq a<k_j\).
By \(\pi(\Gamma)=P\), take a lift \(\gamma_{\psi_0}\in\Gamma\) of \(\psi_0\), and write \(\gamma_0:=\gamma_{\psi_0}^{-1}\gamma\gamma_{\psi_0}\), \(\gamma_{\psi_0}^{-1}(z,w)=(p_j,w_0)\).
\(\gamma_0\) fixes \((p_j,w_0)\), and \(\pi(\gamma_0)=c_j^a=\pi(C_j^a)\).
Thus, \(\gamma_0C_j^{-a}\in\Gamma\cap\ker\pi=\langle\tau^r\rangle\), so we can write \(\gamma_0=\tau^{r\ell}C_j^a\) for some \(\ell\in\Z\).
The translation amount of \(C_j\) obtained above gives $\gamma_0(p_j,w_0)=\left(p_j,w_0+rT\left(\frac{ab_j}{k_j}+\ell\right)\right)=(p_j,w_0)$.
Therefore, \(ab_j+k_j\ell=0\), and \(\gcd(b_j,k_j)=1\) gives \(k_j\mid a\), which contradicts \(1\leq a<k_j\).

\item The characteristic flow of \(M_\Gamma\) is of type \textup{(SF)}.\\
\hspace*{2em} Proof :  Let \(\Phi_t\) be the characteristic flow of \(M_\Gamma\).
The first item and the proof of Lemma~\ref{lem:flow-type-criterion} give \(\ker(t\mapsto\Phi_t)=rT\Z\), so its effective acting group is \(\R/(rT\Z)\).
Since \(J\neq\varnothing\), take one \(j\in J\).
By \(C_j(p_j,w)=(p_j,w+rb_jT/k_j)\) obtained in the preceding item, \([rb_jT/k_j]\in\R/(rT\Z)\) fixes \([p_j,w]\in M_\Gamma\).
Since \(k_j>1\) and \(\gcd(b_j,k_j)=1\), we have \(rb_jT/k_j\notin rT\Z\), and this circle action is not free.
Also, since \(\xi\) is nowhere zero, this circle action is locally free.
Thus, the characteristic flow is of type \textup{(SF)}.
\end{enumerate}
\end{proof}

Exceptional fibers correspond to the nontrivial point stabilizers of the projection group \(P\).
The relation with the order of vanishing of \(s\) is given by the congruence condition \(rb_j+d_j\equiv0\pmod{k_j}\).
Exceptional fibers can exist even when $s$ has no zeros.

\subsection{Constructions of types \texorpdfstring{$(\mathrm{IR}_0)$, $(\mathrm{IR}_1)$}{IR0, IR1}}

\begin{lemma}
\label{cor:holomorphic-gauge}
Let $(v,s)$ be an admissible pair.
For $h\in\mathcal O^*(\D)$, $(v',s')=(v-\log |h|,hs)$ is also an admissible pair.
Also, $F_h(z,w)=\left(z,w-\frac{1}{\sqrt{2K}}\Im \log h(z)\right)$ gives a structure-preserving isomorphism \(\mathcal M_K(v,s)\to\mathcal M_K(v',s')\).
In particular, if \(s\) is nowhere zero, we can obtain \(s'=1\) by taking \(h=1/s\).
Also, for \(\chi\in\Aut(\D)\), $(v^\chi,s^\chi):=(v\circ\chi^{-1},s\circ\chi^{-1})$ is an admissible pair, and \(F_\chi(z,w)=(\chi(z),w)\) gives a structure-preserving isomorphism \(\mathcal M_K(v,s)\to\mathcal M_K(v^\chi,s^\chi)\).
\end{lemma}

\begin{proof}
Since \(\log |h|=\Re \log h\) is harmonic, \(\Delta_0v'=\Delta_0(v-\log|h|)=\Delta_0v\).
Also, since $e^{v'}|s'|=e^v/|h|\cdot |h||s|=\e^v|s|$, $(v',s')$ is an admissible pair.
Applying Proposition~\ref{prop:generalized-rigidity} with
\(\psi=\id\), \((V,S)=(v-\log|h|,hs)\), and \(L=-\log h\) gives the formula for $F_h$.
The rest follows from Lemma~\ref{lem:singular-potential-gauge}.

For \(\chi\in\Aut(\D)\), we have \(\Delta_0v^\chi=|(\chi^{-1})'|^2(\Delta_0v)\circ\chi^{-1}>0\) and \(\e^{v^\chi}|s^\chi|=(\e^v|s|)\circ\chi^{-1}<1\), so $(v^\chi,s^\chi)$ is also admissible.
Taking \(\psi=\chi\) in Proposition~\ref{prop:generalized-rigidity}, we have $s/s^\chi\circ\chi=1$, so choosing the branch of $\log$ to be \(L=0\) shows that \(F_\chi(z,w)=(\chi(z),w)\) is a structure-preserving isomorphism \(\mathcal M_K(v,s)\to\mathcal M_K(v^\chi,s^\chi)\).
\end{proof}

Below, ``after a suitable holomorphic gauge'' means replacing the model using these structure-preserving isomorphisms \(F_h\) or $F_\chi$ and also conjugating the acting group as \(\Gamma\mapsto F\Gamma F^{-1}\).

\begin{lemma}
\label{lem:single-zero}
Let \(\mathcal M_K(v,s)\) be a model space and let \(\Gamma\subset\StrAut(\mathcal M_K(v,s))\) be a free discrete subgroup.
Assume that \(s^{-1}(0)\neq\varnothing\) and that \(\pi(\Gamma)\) is nondiscrete.
Then \(s\) has exactly one zero, and after a suitable holomorphic gauge, we can write $s(z)=z^d,\ v=v(|z|),\ d\geq1$.
Furthermore, every element of \(\pi(\Gamma)\) is then a rotation fixing \(0\).
\end{lemma}

\begin{proof}
Since \(\pi(\Gamma)\) preserves \(\mathfrak u\), it preserves \(s^{-1}(0)\).
Take a sequence of nonidentity elements \(\psi_j\in\pi(\Gamma)\) such that \(\psi_j\to\id\).
If there were two distinct points \(p_1,p_2\in s^{-1}(0)\), then for each \(i=1,2\), we would have $\psi_j(p_i)\in s^{-1}(0)$ and $\psi_j(p_i)\longrightarrow p_i$.
Since \(s\not\equiv0\), its zeros are isolated, so for sufficiently large \(j\), we have
\(\psi_j(p_1)=p_1\) and \(\psi_j(p_2)=p_2\).
Lemma~\ref{lem:disk-automorphism-normal-form} gives \(\psi_j=\id\), a contradiction.
Thus, $s^{-1}(0)=\{p\}$, and every element of \(\pi(\Gamma)\) fixes \(p\).

Moving \(p\) to \(0\) makes \(\pi(\Gamma)\) a subgroup of the rotation group.
Its closure is the whole of \(S^1\).
Indeed, nondiscreteness gives rotations with arbitrarily small nonzero angles \(\delta>0\), and their integer multiples \(0,\delta,\ldots,\lfloor2\pi/\delta\rfloor\delta\) form a \(\delta\)-net of the circle.
Thus, \(\pi(\Gamma)\) is dense.

Since \(\mathfrak u\) is continuous on the punctured disk, it is a radial function.
Writing \(s=z^d a\), with \(a\in\mathcal O^*(\D)\), and using the gauge \(a^{-1}\), we obtain
\(s'=z^d\) and \(v'=\mathfrak u-d\log|z|\).
The function \(v'\) is a smooth function obtained by the gauge, and since it is radial on the punctured disk, it is radial on \(\D\).
\end{proof}

\begin{lemma}
\label{lem:rotational-deck}
Let \(\mathcal M_K(v,s)\) be a model space such that \(s=z^d\), \(v=v(|z|)\), and \(d\geq0\), and let \(\Gamma\subset\StrAut(\mathcal M_K(v,s))\) be a free discrete subgroup.
If \(\pi(\Gamma)\) is nondiscrete in the rotation group, then for some \(\alpha\notin\Q\) and \(n\in\Z\), we can write
\begin{align}
\Gamma=\langle\widehat R_{\alpha,n}\rangle\cong\Z,\ 
 \widehat R_{\alpha,n}(z,w)
  =\bigl(\e^{2\pi\ii\alpha}z,w+(n-d\alpha)T\bigr),\label{eq:rotational-generator}\\
  T=\frac{2\pi}{\sqrt{2K}},\ 
  n-d\alpha\neq0.\nonumber
\end{align}
Then \(M_\Gamma\simeq S^1\times\R^2\), and \(z=0\) is the unique closed characteristic orbit.
The other orbits are dense in the two-tori \(|z|=\mathrm{const}\).
\end{lemma}

\begin{proof}
We prove \eqref{eq:rotational-generator}.
\begin{enumerate}
  \item For \(\beta\in\R\), the lifts of the rotation \(R_\beta(z):=\e^{2\pi\ii\beta}z\) are $\widehat R_{\beta,n}(z,w)=\bigl(\e^{2\pi\ii\beta}z,w+(n-d\beta)T\bigr), \ n\in\Z$.\\
  \hspace*{2em} Proof :  Since \(\mathfrak u=v+d\log|z|\) is radial, \(R_\beta\in G_{\mathfrak u}\).
  Also, $ \frac{s}{s\circ R_\beta} =\frac{z^d}{(\e^{2\pi\ii\beta}z)^d} =\e^{-2\pi\ii d\beta}$ extends to a nowhere-zero holomorphic function on $\D$.
  Its holomorphic logarithms are \(2\pi\ii(n-d\beta)\), for \(n\in\Z\).
  Proposition~\ref{prop:generalized-rigidity} gives the above formula.

  \item We have \(\Gamma\cong\Z\).\\
  \hspace*{2em} Proof :  Since $\pi(\Gamma)$ is a rotation group in $\D$, each \(\gamma\in\Gamma\) preserves \(\{0\}\times\R\), so we define a homomorphism \(\delta:\Gamma\to\R\) by \(\gamma(0,w)=(0,w+\delta(\gamma))\).
  If \(\delta(\gamma)=0\), then \(\gamma\) fixes \((0,0)\), so freeness of the \(\Gamma\)-action gives \(\gamma=1\).
  Thus, \(\delta\) is injective.
  Also, for a compact interval \(I\subset\R\), if \(\delta(\gamma)\in I\), then \(\gamma(0,0)\in\{0\}\times I\).
  Properness of the \(\Gamma\)-action implies that there are only finitely many such \(\gamma\), so \(\delta(\Gamma)\) is a discrete subgroup of \(\R\).
  Since \(\pi(\Gamma)\) is nondiscrete, \(\Gamma\) is infinite, and \(\delta(\Gamma)\) is a nontrivial discrete subgroup of \(\R\).
  Hence \(\delta(\Gamma)\cong\Z\), and injectivity of \(\delta\) gives \(\Gamma\cong\Z\).

  \item A generator of \(\Gamma\) is
  \(\widehat R_{\alpha,n}\) for some \(\alpha\notin\Q\) and \(n\in\Z\), and \(n-d\alpha\neq0\).\\
  \hspace*{2em} Proof :  Since $\Gamma\simeq\Z$, taking a generator \(\gamma\) of \(\Gamma\), we can write \(\gamma=\widehat R_{\alpha,n}\).
  If \(\alpha\in\Q\), then \(\langle R_\alpha\rangle\) is finite, and if \(\alpha\notin\Q\), it is
  dense by Lemma~\ref{lem:disk-automorphism-dynamics}(i).
  Thus, nondiscreteness of \(\pi(\Gamma)=\langle R_\alpha\rangle\) gives \(\alpha\notin\Q\).
  Also, since \(\delta\) is injective and \(\gamma\neq1\),
  \(\delta(\gamma)=(n-d\alpha)T\neq0\), so \(n-d\alpha\neq0\).
\end{enumerate}
We prove the remaining assertions.
Setting \(t:=w/((n-d\alpha)T)\), we have \(\gamma(z,t)=(\e^{2\pi\ii\alpha}z,t+1)\).
Define the map \(f:(\D\times\R)/\Gamma\to\D\times S^1\) by
  \begin{equation}
    f([z,t]):=\bigl(\e^{-2\pi\ii\alpha t}z,\e^{2\pi\ii t}\bigr).
    \label{eq:IR1-diffeomorphism}
  \end{equation}
For a given $(\zeta,u)\in\D\times S^1$, choose any one \(t\in\R\) such that \(u=\e^{2\pi\ii t}\), and define the inverse map by $ f^{-1}(\zeta,u)=[\e^{2\pi\ii\alpha t}\zeta,t]$.
This definition is independent of the choice of $t$.
Indeed, if $t'=t+k,\ k\in\Z$, then $\left[\e^{2\pi\ii\alpha t'}\zeta,t'\right]=\left[\e^{2\pi\ii\alpha(t+k)}\zeta,t+k\right]=\left[\gamma^k\bigl(\e^{2\pi\ii\alpha t}\zeta,t\bigr)\right]=\left[\e^{2\pi\ii\alpha t}\zeta,t\right]$, so it is well-defined.
Thus, \(M_\Gamma\simeq\D\times S^1\simeq S^1\times\R^2\).
Under \eqref{eq:IR1-diffeomorphism}, the orbit \(z=0\) becomes \(\{0\}\times S^1\), so it is closed.
Also, an orbit through \(z\neq0\) is dense in the two-torus defined by $|z|=\textrm{const}$, so there are no closed orbits.
\end{proof}

\begin{screen}
\begin{theorem}
\label{thm:IR0-deck-classification}
Let \(\mathcal M_K(v,s)\) be a model space and let \(\Gamma\subset\StrAut(\mathcal M_K(v,s))\) be a free discrete subgroup, and let $M_\Gamma:=\mathcal M_K(v,s)/\Gamma$.
The characteristic flow of \(M_\Gamma\) is of type IR\(_0\) if and only if, after a suitable holomorphic gauge, \(s=1\), and there exist a hyperbolic or parabolic one-parameter subgroup \(A=\{A_t:t\in\R\}\subset G_v\), \(\alpha\in\R\setminus\Q\), and \(m\in\Z_{>0}\) such that we can write
\begin{align}
&\Gamma=\langle\gamma_1,\gamma_2\rangle\cong\Z^2,\\
&\gamma_1(z,w)=(A_1(z),w),\qquad
\gamma_2(z,w)=(A_\alpha(z),w+mT),\label{eq:IR0-deck-normal-form}\\
&T=\frac{2\pi}{\sqrt{2K}}.
\end{align}
\end{theorem}
\end{screen}

\begin{proof}
Assume IR$_0$, and we prove \eqref{eq:IR0-deck-normal-form}.
\begin{enumerate}
  \item We have \(M_\Gamma\simeq\T^2\times\R\), and $\pi(\Gamma)$ is nondiscrete. Also, after a holomorphic gauge, we can obtain
  \(s=1\).\\
  \hspace*{2em} Proof :  Theorem~\ref{thm:orbit-trichotomy} gives $M_\Gamma\simeq\T^2\times\R,\ \Gamma\cong\pi_1(M_\Gamma)\cong\Z^2$.
  In IR\(_0\), \(L\) is not closed, so this is not (SF), and the flow is Palais-nonregular. Thus, Lemma~\ref{lem:flow-type-criterion}(i) gives
  \(\Gamma\cap\langle\tau\rangle=\{1\}\).
  Furthermore, part (iii) of the same lemma shows that \(\pi(\Gamma)\) is nondiscrete.
  If \(s\) had a zero, Lemmas~\ref{lem:single-zero}, \ref{lem:rotational-deck} would give
  \(\Gamma\cong\Z\), a contradiction. Thus, \(s\) is nowhere zero, and
  the gauge \(s^{-1}\) gives \(s=1\).

  \item For some \(m\in\Z_{>0}\) and \(\lambda_0,a\in G_v\), we can write $\Gamma=\langle(\lambda_0,0),(a,m)\rangle$.\\
  \hspace*{2em} Proof :  In this gauge, $\StrAut(\mathcal M(v,1))=G_v\times\Z,\ (\psi,n)(z,w)=(\psi(z),w+nT)$.
  Let \(\nu:\Gamma\to\Z\) be the second component.
  If \(\nu(\Gamma)=\{0\}\), then \(\Gamma\subset G_v\times\{0\}\), and
  since the projection \(G_v\times\{0\}\to G_v\) is a homeomorphism,
  discreteness of \(\Gamma\) implies that \(\pi(\Gamma)\) is also discrete, a contradiction.
  Thus, \(\nu(\Gamma)=m\Z\) for some \(m\in\Z_{>0}\).
  Since \(\Gamma\cong\Z^2\) and \(\nu(\Gamma)\cong\Z\), we have \(\ker\nu\cong\Z\).
  Let \((\lambda_0,0)\) be a generator of \(\ker\nu\), and let \((a,m)\) be an element whose value under \(\nu\) is \(m\).
  For every \((\psi,km)\in\Gamma\), we have
  \((\psi,km)(a,m)^{-k}\in\ker\nu=\langle(\lambda_0,0)\rangle\), which gives the above generating expression.

  \item We can choose a hyperbolic or parabolic one-parameter subgroup
  \(A=\{A_t:t\in\R\}\subset G_v\) and \(\alpha\in\R\setminus\Q\) such that
  \(\lambda_0=A_1\) and \(a=A_\alpha\).\\
  \hspace*{2em} Proof :  Since \(\Gamma\cong\Z^2\) is abelian, \(a\lambda_0=\lambda_0a\).
  If \((\lambda_0,0)\) fixed a point in the disk, it would fix a point of \(\mathcal M(v,1)\),
  contrary to freeness of the \(\Gamma\)-action. Thus, \(\lambda_0\) is hyperbolic or parabolic.
  By Lemma~\ref{lem:disk-automorphism-dynamics}(ii),(iii), the centralizer of \(\lambda_0\) is
  a one-parameter subgroup \(A\cong\R\), and \(a\in A\).
  Choosing a parameter and writing \(\lambda_0=A_1\) and \(a=A_\alpha\), we have \(\pi(\Gamma)\simeq\Z+\alpha\Z\).
  If \(\alpha\in\Q\), then $\Z+\alpha\Z\simeq\Z$ would be discrete, which is impossible because \(\pi(\Gamma)\) is nondiscrete.
  Hence \(\alpha\notin\Q\).
  Also, \(\pi(\Gamma)\) is dense in \(A\), and every element preserves \(v\), so continuity of \(v\) gives \(A\subset G_v\). These observations give \eqref{eq:IR0-deck-normal-form}.
\end{enumerate}

Assume $s=1$ and \eqref{eq:IR0-deck-normal-form}, and we prove IR$_0$.
\begin{enumerate}
  \item There are coordinates \(\D\simeq\R_u\times J\), where $J$ is an open interval, such that $A_t(u,y)=(u+t,y)$.\\
  \hspace*{2em} Proof :  Identify \(\D\) with the upper half-plane \(\mathbb H\) by the Cayley transformation.
  If $A$ is hyperbolic, by Lemma~\ref{lem:disk-automorphism-dynamics}(ii), take $A_1(\zeta)=c\zeta,\ \zeta\in\mathbb H$, and write $\zeta=r\e^{\ii\theta},\ r>0,\quad 0<\theta<\pi$.
  Then define $f:\mathbb{H}\to \R_u\times (0,\pi)$ by $f(\zeta):=(\log r/\log c, \theta)$.
  If $A$ is parabolic, by Lemma~\ref{lem:disk-automorphism-dynamics}(ii), take $A_1(\zeta)=\zeta+b,\ \zeta\in\mathbb H$, and write $\zeta=x+\ii y,\ x\in\R,\ y>0$.
  Then define $f:\mathbb{H}\to \R_u\times (0,\infty)$ by $f(\zeta):=(x/b,y)$.
  The above $f$ is a diffeomorphism and satisfies $A_t(u,y)=(u+t,y)$ by construction.

  \item We have \(M_\Gamma\simeq\T^2\times\R\).\\
  \hspace*{2em} Proof :  \(\Gamma\) acts on the \((u,w)\)-plane as the lattice group $\Z(1,0)+\Z(\alpha,mT)$ and does not act on \(J\).
  Thus, \(M_\Gamma\simeq\T^2\times J\simeq\T^2\times\R\).

  \item There are no closed characteristic orbits, and each orbit is dense in the corresponding invariant two-torus.\\
  \hspace*{2em} Proof :  In the $(u,w)$-plane, the characteristic direction is $(0,1)=-\frac{\alpha}{mT}(1,0)+\frac{1}{mT}(\alpha,mT)$ and is proportional to \((-\alpha,1)\) in lattice coordinates.
  Since \(\alpha\notin\Q\), each orbit is dense in the two-torus, and there are no closed orbits.
\end{enumerate}
\end{proof}

\begin{screen}
\begin{theorem}
\label{thm:IR1-deck-classification}
Let \(\mathcal M_K(v,s)\) be a model space and let \(\Gamma\subset\StrAut(\mathcal M_K(v,s))\) be a free discrete subgroup, and let $M_\Gamma:=\mathcal M_K(v,s)/\Gamma$.
The characteristic flow of \(M_\Gamma\) is of type IR\(_1\) if and only if, after a suitable holomorphic gauge, there exist \(d\in\Z_{\geq0}\), \(\alpha\in\R\setminus\Q\), and \(n\in\Z\) such that we can write
\begin{align}
s(z)=z^d,\qquad v=v(|z|),\\
\Gamma=\langle\gamma\rangle\cong\Z,\ 
\gamma(z,w)=\bigl(\e^{2\pi\ii\alpha}z,w+(n-d\alpha)T\bigr),\label{eq:IR1-deck-normal-form}\\
T=\frac{2\pi}{\sqrt{2K}},\  n-d\alpha\neq0.
\end{align}
\end{theorem}
\end{screen}

\begin{proof}
Assuming IR$_1$, we prove \eqref{eq:IR1-deck-normal-form}.
\begin{enumerate}
  \item We have $M_\Gamma\simeq S^1\times\R^2$, and \(\pi(\Gamma)\) is nondiscrete.\\
  \hspace*{2em} Proof :  Theorem~\ref{thm:orbit-trichotomy} gives $M_\Gamma\simeq S^1\times\R^2$.
  Furthermore, since $\Gamma$ is a Deck group, $\Gamma\cong\Z$.
  In IR\(_1\), \(L\) is not closed, so this is not the Seifert fiber case, and the flow is Palais-nonregular. Thus, Lemma~\ref{lem:flow-type-criterion}(i) gives $\Gamma\cap\langle\tau\rangle=\{1\}$.
  Lemma~\ref{lem:flow-type-criterion} (iii) shows that \(\pi(\Gamma)\) is nondiscrete.

  \item If \(s\) has a zero, Lemmas~\ref{lem:single-zero}, \ref{lem:rotational-deck} give the case \(d\geq1\) of \eqref{eq:IR1-deck-normal-form}.

  \item If \(s\) is nowhere zero, we obtain the case $d=0$ of \eqref{eq:IR1-deck-normal-form}.\\
  \hspace*{2em} Proof :  Use a gauge to set \(s=1\), and write a generator of \(\Gamma\) as \((a,n)\in G_v\times\Z\).
  The group \(\langle a\rangle=\pi(\Gamma)\) is nondiscrete.
  If \(a\) had no fixed point in \(\D\), then Lemma~\(\ref{lem:disk-automorphism-dynamics}\)(ii) would imply that \(\langle a\rangle\) is discrete, contrary to the assumption.
  Thus, \(a\) has a fixed point in \(\D\).
  By part (i) of the same lemma, \(a\) is conjugate to a rotation.
  If its rotation angle were rational, \(\langle a\rangle\) would be finite and hence discrete, so the rotation angle must be irrational.
  Thus, \(a\) is conjugate to an irrational rotation.
  Moving its fixed point to \(0\), we have \(a(z)=\e^{2\pi\ii\alpha}z\) and \(\alpha\notin\Q\).
  The equality \(v\circ a=v\) and density of the powers of the rotation imply that \(v\) is radial.
  Freeness at \(z=0\) gives \(n\neq0\).
  Thus, we obtain the case $d=0$ of \eqref{eq:IR1-deck-normal-form}.
\end{enumerate}
Conversely, assume the data in \eqref{eq:IR1-deck-normal-form}.
For $k\neq0$, we can see $\gamma^k$ has no fixed point, thus, the action of $\Gamma=\langle\gamma\rangle$ is free.
Let $Q\subset\D\times\R$ be a nonempty compact set.
If $\gamma^k(Q)\cap Q\neq\varnothing$, then we have $|k(n-d\alpha)T|\leq\max_Q w-\min_Q w$.
Only finitely many integers $k$ satisfy this inequality, so the action is properly discontinuous.
Since $\alpha\notin\Q$, the projection $\pi(\Gamma)$ is nondiscrete in the rotation group.
Lemma~\ref{lem:rotational-deck} now gives a quotient of type IR\(_1\).
\end{proof}

\end{document}